\documentclass[a4paper,reqno,11pt]{amsart}
\usepackage[left=2.8cm,right=2.8cm,top=2.8cm,bottom=2.8cm]{geometry}
\usepackage[utf8]{inputenc}	
\usepackage{indentfirst} %Per avere il rientro nella prima riga di ogni capitolo.
\usepackage{amsmath,amsfonts,amssymb,,mathrsfs}
\usepackage{bbm}
\usepackage[colorlinks=true]{hyperref}
\usepackage{enumerate}
\usepackage{graphicx}
\usepackage{xcolor}
\usepackage{todonotes}
\usepackage[shortlabels]{enumitem}

\newcommand{\restr}{%
  \,\raisebox{-.127ex}{\reflectbox{\rotatebox[origin=br]{-90}{$\lnot$}}}\,%
}

\theoremstyle{theorem}
\newtheorem{definition}{Definition}[section]
\theoremstyle{theorem}
\newtheorem{remark}[definition]{Remark}
\theoremstyle{theorem}
\newtheorem{theorem}[definition]{Theorem}
\theoremstyle{theorem}
\newtheorem{lemma}[definition]{Lemma}
\theoremstyle{remark}

\theoremstyle{theorem}
\newtheorem{corollary}[definition]{Corollary}
\theoremstyle{theorem}
\newtheorem{proposition}[definition]{Proposition}
\theoremstyle{theorem}

\numberwithin{equation}{section}
\allowdisplaybreaks

\def\Xint#1{\mathchoice
    {\XXint\displaystyle\textstyle{#1}}%
    {\XXint\textstyle\scriptstyle{#1}}%
    {\XXint\scriptstyle\scriptscriptstyle{#1}}%
    {\XXint\scriptscriptstyle\scriptscriptstyle{#1}}%
    \!\int}
\def\XXint#1#2#3{{\setbox0=\hbox{$#1{#2#3}{\int}$}
      \vcenter{\hbox{$#2#3$}}\kern-.5\wd0}}

\def\mint{\Xint-}

\definecolor{ao}{rgb}{0.0, 0.5, 0.0}

\DeclareMathOperator*{\einf}{ess\inf}
\DeclareMathOperator*{\esup}{ess\sup}

\newcommand{\R}{\mathbb{R}}

\newcommand{\ep}{\epsilon}
\newcommand{\UUU}{\color{black}}

\newcommand{\EEE}{\color{black}}

\title[non-local non-homogeneous phase transitions]{non-local non-homogeneous phase transitions: regularity of optimal profiles\\ and sharp-interface limit.}
\date{\today}

\author[E. Davoli]{Elisa Davoli}
\author[E. Tasso]{Emanuele Tasso}

\AtEndDocument{
	\bigskip
	\footnotesize{
		\noindent
		E. Davoli, E. Tasso.
		
		\noindent
		Institute of Analysis and Scientific Computing,
		TU Wien,
		
		\noindent
		Wiedner Hauptstra\ss e 8-10, 1040 Vienna, Austria.
		
		\noindent
		E-mails:
		\href{mailto:elisa.davoli@tuwien.ac.at}{\tt elisa.davoli@tuwien.ac.at},
		\href{mailto:emanuele.tasso@tuwien.ac.at}{\tt emanuele.tasso@tuwien.ac.at}, 
}}

\begin{document}
	
	\begin{abstract}
We provide a novel sharp-interface analysis via Gamma-convergence for a non-local and non-homogeneous diffuse-interface model for phase transitions, featuring an interplay between a non-local interaction kernel and a spatially dependent double-well potential. This interaction requires the development of new strategies both for the Gamma-liminf inequality and for the construction of recovery sequences. A key element of our approach is an asymptotic calibration, used to establish the Gamma-liminf lower bound. The study of the optimality of the lower bound hinges upon a novel analysis of the regularity dependence of one-dimensional optimal profiles on a family of parameters. In particular, we show how such regularity is influenced by the singularity of the interaction kernel at the origin, providing a precise and previously unexplored link between the two. Our results rely solely on the assumption of H\"older continuity for the moving wells, and also account for the compactness of sequences with equibounded energies.

		\medskip
  


	\end{abstract}
	
	\maketitle
	
	{\parskip=0em 
    \hypersetup{linkcolor=blue}
    \tableofcontents}

\section{Introduction}

The Modica-Mortola energy functional, also referred to as the Ginzburg-Landau free energy in physics, has been widely used to rigorously provide variational approximations of sharp-interface models of phase transitions by means of diffuse-interface descriptions. Originating from the theories of van der Waals and Cahn-Hilliard (cf. \cite{vander} and \cite{cahn}), this functional forms the basis of many studies on liquid-liquid phase separation. While it has been extensively analyzed, its classical form is limited to isothermal settings and cannot adequately model phase transitions driven by thermal activation. 

Non-isothermal models address this limitation by incorporating temperature-dependent effects into the energy functional (cf. \cite{pen,alt}). Such extensions are crucial to capture the interplay between thermal activation and phase separation. A first step to eventually incorporate temperature dependence in the modeling is to assume a prescribed temperature distribution throughout the specimen, possibly affecting both the value of the potential energy and the position of its wells. This naturally leads to the study of non-homogeneous phase transitions. Non-local formulations further broaden the scope by accounting for long-range interactions, which are relevant in various physical systems. Examples include models arising in equilibrium statistical mechanics, such as the continuum limit of Ising spin systems on lattices (cf. \cite{abcp} and references therein).

 In this paper, we analyze sharp-interface limits for non-local and non-homogeneous phase transitions. Specifically, for a bounded domain $\Omega \subset \mathbb{R}^m$, for every $\ep >0$ we consider the functional
\begin{equation}
    \label{e:functionalintro}
    F_\ep(u;\Omega):= \frac{1}{4\ep} \int_{\Omega \times \Omega} \frac{1}{\ep^m}J\bigg(\frac{x'-x}{\ep}\bigg)|u(x')-u(x)|^2 \, dx'dx + \frac{1}{\ep} \int_{\Omega}W(x,u(x))\, dx,
\end{equation}
where $u \colon \Omega \to \mathbb{R}$ represents the concentration parameter, $J$ is a ferromagnetic Kac potential (cf. \cite{kac}) vanishing at infinity, $W$ is a space-dependent double-well potential, and 
\[
{\bf z}(x):= \{t \in \mathbb{R} \ | \ W(x,t)=0\} =\{z_1(x),z_2(x)\},
\]
where $z_i \colon \mathbb{R}^m \to \mathbb{R}$ are prescribed H\"older-continuous functions for $i=1,2$.

We characterize here the asymptotic behavior via Gamma-convergence of the functionals \eqref{e:functionalintro} as $\epsilon\to 0$. While local models of similar energy functionals have been analyzed in previous works, including scalar and vectorial settings, the combined effects of non-locality and space-dependence of double-well potential remain largely unexplored. The main contribution of this paper is to address this gap by rigorously analyzing the interaction between these two features. 

\subsection{Previous contributions and relevant works.} Consider a fluid confined within a container $\Omega \subset \mathbb{R}^m$ and assume that the configuration of the system is described by a function $u \colon \Omega \to \mathbb{R}$. Given a double well potential $W$, the associated free energy functional, which we refer to as, \emph{local} Modica-Mortola functional, reads as
\begin{equation}
\label{e:Eepintro}
   E_\ep(u;\Omega)=\ep\int_{\Omega}|\nabla u(x)|^2 \, dx + \frac{1}{\ep} \int_\Omega W(u(x))\, dx.
\end{equation}
As $\ep \to 0$, this functional $\Gamma$-converges to a sharp-interface energy measuring the surface area of the interface between phases. This result, conjectured by Gurtin (cf. \cite{gur1}) and rigorously established by Carr, Gurtin, and Slemrod for $m=1$ (cf. \cite{gur2}), and by Modica (cf. \cite{mod2}) and Sternberg (cf. \cite{ster1}) for $m \geq 2$, has motivated a vast literature on phase transition models. An important generalization involves vector-valued phase variables $u \colon \Omega \to \mathbb{R}^k$, which describe systems with multiple stable phases. The case of a phase variable $u \colon \Omega \to \mathbb{R}^k$ with a potential supporting $n = 2$ stable phases was addressed by Sternberg (cf. \cite{ster2}) for $k = 2$ and by Fonseca and Tartar (cf. \cite{fontar}) for $k \geq 2$, while Baldo (cf. \cite{bal}) studied the general scenario of $n \geq 2$ stable phases. The case where the zero level set of $W$ has a more complicated topology has been mainly studied by three authors. Sternberg (cf. \cite{ster3}) considered the specific situation where the potential vanishes on two disjoint $C^{1,2}$ curves in $\mathbb{R}$. Ambrosio (cf. \cite{amb}) investigated the case where the set ${\bf z}$ of zeros of $W$ is a generic compact set in $\mathbb{R}^k$. Bouchitt\'e (cf. \cite{buc}) treated the scalar case, namely, $k=1$, where the zero set of the potential is represented by two moving wells $z_1,z_2 \colon \Omega \to \mathbb{R}$ which are Lipschitz continuous. More recently, Cristoferi and Gravina were able to rigorously derive the asymptotic analysis by Gamma-convergence of the local and vectorial Modica-Mortola functionals with space dependent potential $W(x,t)$, where the zero set ${\bf z}(x)= \{z_1(x),\dotsc,z_n(x)\}$ is described by locally Lipschitz functions $z_i \colon \Omega \to \mathbb{R}^k$ for $i=1,\dotsc,n$ (cf. \cite{crigra}). 

The analysis of local Modica-Mortola functionals, whether for fixed or moving wells, relies on the concept of geodesic distance $\text{d}_W \colon \Omega \times \mathbb{R}^k \times \mathbb{R}^k \to [0,\infty]$ associated with the potential $W$ (cf. \cite[Definition 1.2]{crigra}). Loosely speaking, this distance measures the optimal way to connect two wells via a curve, with respect to the metric induced by the conformal factor $\sqrt{W}$. 
% Specifically, for $x \in \Omega$ and $p, q \in \mathbb{R}^k$, this distance is defined as the solution of the following minimization problem:
% \[
% \text{d}_W(x,p,q):= \inf \bigg\{\int_{-1}^1 \sqrt{W(x,\gamma(t))}|\dot{\gamma}(t)|\, dt \ \big| \ \gamma \in W^{1,1}((-1,1)),\ \gamma(-1)=p, \ \gamma(1)=q \bigg\}.
% \]
% With this definition, 
The Gamma-limit $E_0(u;\Omega)$, computed with respect to the $L^1$-topology, takes the following general form:
\begin{equation}
\label{e:locallimit}
    E_0(u;\Omega):=
    \begin{cases}
       \int_{J_u} \text{d}_W(x,u^+(x),u^-(x)) \, d\mathcal{H}^{n-1}(x), & \text{if } u \in BV(\Omega;{\bf z}(x)), \\
       +\infty, & \text{otherwise in } L^1(\Omega;\mathbb{R}^k),
    \end{cases}
\end{equation}
where $J_u$ denotes the jump set of $u$, $u^\pm(x)$ denote the traces on the two sides at $x \in J_u$, $\mathcal{H}^{n-1}$ is the $(n-1)$-dimensional Hausdorff measure, and $BV(\Omega;{\bf z}(x))$ represents the space of functions of bounded variation in $\Omega$ such that $u(x) \in {\bf z}(x)$ almost everywhere in $\Omega$.

In this paper, however, our primary focus will be on a class of \emph{non-local} Modica-Mortola-type functionals, as specified in \eqref{e:functionalintro}. In the case of homogeneous potential, such a class of functionals were rigorously analyzed by Alberti and Bellettini in \cite{alb-bel2} (see also the recent generalization \cite{caldwell}). Since the techniques developed in \cite{alb-bel2} are central to our work, and they differ significantly from the methods employed in the local cases mentioned earlier, we provide a more detailed discussion of this latter framework. In \cite{alb-bel2}, for every positive parameter $\epsilon >0$, the authors consider free energy functionals of the form  
\begin{equation}
\label{e:albbelintro}
    F_\ep(u;\Omega) :=\frac{1}{4\ep} \int_\Omega \int_\Omega \frac{1}{\ep^m}J\bigg(\frac{x' - x}{\ep}\bigg)|u(x') - u(x)|^2 \, dx' \, dx + \frac{1}{\ep}\int_\Omega W(u(x)) \, dx, 
\end{equation}
where $u \colon \Omega \to \mathbb{R}$, $J$ is a positive and symmetric interaction kernel that vanishes at infinity, and $W$ is a double-well potential that attains zero only at $\pm 1$. Setting $\ep =1$, for an energy-minimizing configuration $u$ of $F$, the second term in $F$ drives $u$ toward the pure states $\pm 1$, leading to phase separation, while the first term represents the interaction energy, penalizing spatial inhomogeneities in $u$ and enforcing surface tension. So, under the mass constraint $\int_\Omega u =c$, for very large value of $W$ compared to $J$, the second term of $F$ prevails, and minimizers tends to take values close to $\pm 1$ while the transition between the two phases $\pm 1$ occurs on a thin layers. In \cite{alb-bel2} the authors investigate this situation by passing to the thermodynamic limit, namely, they study the asymptotic behaviour as $\ep \to 0^+$. It is worth noting that the parameter $1/4\ep$ in front of the first term of $F_\ep$ precisely reflects the scaling of the parameter in front of the first term in \eqref{e:Eepintro}, as becomes evident when choosing $J(h) = h^{-2}$. The main result in \cite{alb-bel2} is that, as $\ep \to 0$, the functionals $F_\ep$ Gamma-converge in $L^1$ to a limiting energy $F_0$, which is finite only when $u = \pm 1$ almost everywhere. In that case, the limiting energy is provided by the area of the interface $J_u$ weighted by an anisotropic surface tension $\sigma$. Specifically, assuming $J$ to be a \emph{positive} and \emph{even} kernel satisfying
\begin{equation}
\label{e:regimedecay}
    \int_{\mathbb{R}^m} J(h)|h| \, dh < \infty,
\end{equation}
and under very general assumptions on the potential $W$ (cf. \cite[Subsection 1.2]{alb-bel2}) it holds
\begin{equation}
\label{e:gammalimab}
    F_0(u;\Omega):=
    \begin{cases}
        \int_{J_u} \sigma(\nu_u) \, d\mathcal{H}^{n-1}, &\text{ if $u \in BV(\Omega;\{-1,1\})$}\\
        +\infty &\text{ otherwise in $u \in L^1(\Omega)$},
    \end{cases}
\end{equation}
where $\sigma \colon \mathbb{S}^{n-1} \to (0,\infty)$ with $\sigma(\xi)=\sigma(-\xi)$ while $\nu_u(x)$ denotes any unit normal at $x \in J_u$. The definition of surface tension is established through a cell formula via an infimum problem, which is closely related to the so-called optimal profile problem (see \eqref{e:optprointro} below). 

While the existence of solutions to the optimal profile problem is proven in a different work \cite{alb-bel1}, the authors do not explicitly use this result in the derivation of the Gamma-limit in \eqref{e:gammalimab}. Instead, their approach to the $\Gamma$-liminf inequality relies on a blow-up method, inspired by Fonseca and Müller’s work \cite{fonmul}. Through a density argument for finite perimeter sets, they reduce the construction of a recovery sequence to the case of polyhedral sets, for which they use only the notion of quasi-minimizers for the optimal profile problem. A key quantity introduced in their paper, and which plays a central role throughout, is the so-called locality defect $\Lambda_\epsilon(u, A, A')$. Roughly speaking, this quantity measures how far the set function $A \mapsto F_\epsilon(u; A)$ deviates from being an additive set function. The decay property of the locality defect as $\epsilon \to 0$ is a crucial tool for ensuring the locality of the limiting functional $F_0$. The precise definition of this quantity and the related properties can be found in Subsection \ref{s:locdef}. 

Another notable contribution concerning nonlocal phase transitions is the work of Savin and Valdinoci \cite{sav}, where they study a similar problem with specific kernels associated with the Riesz potential, allowing for more severe singularities at the origin. Specifically, they consider $J(h) = |h|^{-n-2s}$, with $s$ varying in $(0,1)$. The asymptotic behavior of the functionals significantly depends on whether $s \geq 1/2$ or $s < 1/2$. For $s \geq 1/2$, the limit functional is local and isotropic taking the form of $F_0$, as described in \cite[Theorem 1.4]{sav}. Conversely, for $s < 1/2$, $F_0$ remains non-local, giving rise to the notion of non-local perimeter. It is important to note that, for $s < 1/2$ the decay property of the defect functional is lost, and \emph{this loss is ultimately responsible for the non-locality of the limit functional}. We will work here under a suitable integrability condition on the kernel $J$ (see \eqref{e:growinfjintro}) excluding the case of fractional kernels with $s<1/2$. The study of non-homogeneous phase transitions in this latter regime thus remains an open question.

\subsection{Main results of the paper} To describe our results more effectively, we introduce the following space. Given real numbers $a$ and $b$, and for every $x \in \Omega$, we define the space $X_a^b$ as 
\[
X_a^b := \{\gamma \colon \mathbb{R} \to \mathbb{R} \ | \  \lim_{t \to -\infty} \gamma(t)=a, \ \lim_{t \to \infty} \gamma(t)=b \}.
\]  

The main contribution of this paper is to prove that, under suitable regularity assumptions on the kernel $J$ (see Subsection \ref{subs:adminterk}) and the potential $W$ (see Subsection \ref{subs:dw}), the family of functionals in \eqref{e:functionalintro} Gamma-converges in the $L^1$-topology to a local functional of the form:
\[
F_0(u;\Omega):=
\begin{cases}
\int_{J_u} \sigma(x,u^+(x),u^-(x),\nu_u(x)), &\text{if } u \in CP(\Omega;{\bf z}(x)), \\
+\infty, &\text{otherwise in } L^1(\Omega),
\end{cases}
\]
where the anisotropic surface tension $\sigma \colon \Omega \times \mathbb{R} \times \mathbb{R} \times \mathbb{S}^{n-1} \to (0,\infty)$ is given by the \emph{optimal profile problem}:
\begin{equation}
\label{e:optprointro}
\sigma(x,a,b,\xi):= \inf_{\gamma \in X_{a}^{b}} \frac{1}{4} \int_{\mathbb{R}}\int_{\mathbb{R}} J^\xi(t-t')|\gamma(t)-\gamma(t')|^2 \, dt dt' + \int_{\mathbb{R}} W(x,\gamma(t)) \, dt.
\end{equation}
In the expression above, $J^\xi \colon \mathbb{R} \to (0,\infty)$ is, loosely speaking, described by the integration of $J$ on all possible affine $(n-1)$-planes orthogonal to $\xi$ (see \eqref{e:kerneljxi1}). The space $CP(\Omega;{\bf z}(x))$ consists of all measurable functions $u \colon \Omega \to \mathbb{R}$ such that $u(x) \in {\bf z}(x)$ almost everywhere in $\Omega$, and the transition between the wells $z_1$ and $z_2$ occurs on the reduced boundary of some finite perimeter set in $\Omega$. Notice that, since in the present paper we assume only H\"older regularity on the moving wells, one does not have, in general, that $CP(\Omega;{\bf z}(x))$ is a subspace of $BV(\Omega)$.
The optimal profile problem in \eqref{e:optprointro} plays a role similar to the geodesic minimization relative to the distance $d_W$ discussed in the local case. Our asymptotic analysis is also complemented by a compactness study of minimizing sequences (see Proposition \ref{p:compactness}).

We first outline the hypotheses concerning the potential underlying the aforementioned Gamma-convergence result. For the discussion of the hypotheses concerning the kernel, we direct the reader to the next two subsections of this introduction, which focus on the Gamma-liminf inequality and the construction of the recovery sequence, respectively. First, following the approach in \cite{crigra}, we impose the natural condition that the wells are well separated (cf. \ref{H2}). Second, as the main novelty in the potential component compared to \cite{buc,crigra}, where only the Lipschitz case was covered, we assume that the moving wells $z_1, z_2 \colon \mathbb{R}^m \to \mathbb{R}$ as well as the partial derivative $\partial_t W$ are $\alpha$-Hölder continuous for some parameter $\alpha \in \left(\frac{1}{2}, 1\right]$. Around the wells, we broadly assume the existence of an increasing function $f \colon [0,\infty) \to [0,\infty)$ with $f(t)=0$ only at $t=0$, such that
\[
\delta \, f\big(\text{dist}(t, {\bf z}(x))\big) \leq W(x,t) \leq \frac{1}{\delta} \, f\big(\text{dist}(t, {\bf z}(x))\big), \ \ \text{ for some $\delta \in (0,1)$},
\]
and that $\partial_t^2 W(x,t)$ is strictly positive, uniformly with respect to both $x$ and $t$ (cf. \ref{H1}, \ref{H3}, and \ref{H4}). We observe that, under these hypotheses, the behavior of the function $f$ near the origin is necessarily quadratic. By potentially decoupling the regularity of the wells from the regularity of $\partial_t W$, one could allow for the possibility of considering functions $f$ whose rate of convergence to zero is slower than quadratic. However, we did not explore this analysis in the present paper. Finally, we assume that $W$ grows at least linearly in the $t$-variable at infinity, again uniformly with respect to both $x$ and $t$ (cf. \ref{H5}).

We stress that the $\alpha$-H\"older regularity for $\alpha \in (\frac{1}{2},1]$, besides being more general, is also a natural assumption while working in the regime that ensures the good decay property of the locality defects (specifically, under condition \eqref{e:regimedecay}). Indeed, any function $z \colon \Omega \to \mathbb{R}$ that is $\alpha$-Hölder continuous is an \emph{asymptotic ground state} for the energy. This means that, if $z$ also represents the zero set of the potential part of the energy, then $F_\ep(z; \Omega) \to 0$ as $\ep \to 0$, provided $\alpha > \frac{1}{2}$. Summarizing, the $\alpha$-Hölder regularity condition not only generalizes the classical Lipschitz case, but also provides a more natural framework for the sharp-interface analysis.

\subsection{The Gamma-liminf analysis} We are able to derive the desired Gamma-liminf inequality under very general integrability assumption on the interaction kernel (see Theorem \ref{t:gammaliminf2}). Specifically, we only assume
\begin{equation}
\label{e:growinfjintro}
    \int_{\mathbb{R}^m \setminus B_\rho(0)} J(h)|h| \, dh < \infty, \ \ \text{for some $\rho \in (0,\infty)$}.
\end{equation}
In particular, we do not assume any rate of integrability around the origin.

Our strategy does not rely on an abstract blow-up argument \`a la Fonseca-M\"uller, as adopted in \cite{alb-bel2}, but instead hinges upon the construction of a family of functionals, provided by the averages of one-dimensional energies, which we refer to as \emph{asymptotic calibration} (see Section \ref{s:dirblowup}). A fundamental quantity in the construction of such a calibration is the notion of the \emph{conjugate functional} introduced in \cite{alb-bel1}. Its role is twofold: given an integrable interaction kernel $J$ and an increasing function $u$, it enables (via the operator $H$ as introduced in Definition \ref{d:hdef}) the construction of a proper potential $W_u$ such that $u$ is the optimal profile associated with the corresponding energy (the so called \emph{inverse problem}). Additionally, it allows us to describe increasing minimizers of the optimal profile problem (see Subsection \ref{subs:conj1d}) as minimizers of a convex functional.

 The reason why we cannot rely on the strategy presented in \cite{alb-bel2} is due to two distinct factors. First, due to the non-homogeneity of $W$, rescaling the energy $F_\ep$ in both variables $x$ and $x'$ does not yield an energy that allows us to appropriately define an $\ep$-independent cell formula. Second, nothing is currently known about non-homogeneous optimal profile problems (see also the next subsection).
Consequently, an implicit approach describing the surface tension $\sigma$ based on an $\ep$-dependent cell formula will not allow us to directly deduce equality \eqref{e:optprointro} from the results on optimal profiles in \cite{alb-bel1}, as it can be done in the homogeneous case. Instead, additional work would still be required.

To briefly describe our technique, suppose we aim to prove the Gamma-liminf inequality for an admissible function $u$ that jumps along an affine $(n-1)$-plane $V$ orthogonal to some $\xi \in \mathbb{S}^{n-1}$. For a fixed $\rho > 0$, we properly cover $V$ with finitely many $n$-dimensional cubes $(Q(x_\ell))_\ell$, each having one face orthogonal to $\xi$. Assuming $u_\ep \to u$ in $L^1(\Omega)$, we rescale each restricted energy $F_\ep(u_\ep; Q(x_\ell))$ by $\ep$ \emph{only in the direction of} $\xi$. Using the aforementioned asymptotic calibrations, we estimate the cost, in terms of the cube size $Q(x_\ell)$, for bounding $\liminf_\ep F_\ep(u_\ep; Q(x_\ell))$ from below by $\sigma(x_\ell, z_1(x_\ell), z_2(x_\ell), \xi)$. Broadly speaking, the key \emph{double liminf inequality} in Proposition \ref{c:est_gamminf} affirms that the error depends on the modulus of continuity of $z_1$ and $z_2$ at $x_\ell$, the cube's center. The assumed H\"older continuity of the moving wells ensures that this error vanishes as the size of the covering cube approaches zero. This whole analysis is performed under the additional assumption of a globally integrable kernel. The setting of kernels $J$ only fulfillinf \eqref{e:growinfjintro} is recovered a posteriori via a truncation argument (see Theorems \ref{t:gammaliminf1} and \ref{t:gammaliminf2}).

We eventually note that the derivation of the Gamma-liminf inequality under condition \eqref{e:growinfjintro} is quite sharp. As already observed, in general, removing this condition leads to a different limiting non-local functional, based on the concept of non-local perimeter. This is precisely what happens in \cite{sav} with the choice $J(h) = |h|^{-n-2s}$ for $s \in (0, 1/2]$.
 This last observation underscores the importance of condition \eqref{e:growinfjintro} for the validity of the decay estimates of the defect functional (cf. Section \ref{s:locdef}), which is ultimately responsible for the locality of the limit functional.  
   
\subsection{The construction of recovery sequences}  
In the construction of a recovery sequence, the most delicate issue lies in the procedure of gluing together optimal profiles in such a way that the resulting functions retain the required regularity. Since a similar argument to the one in \cite{alb-bel2} allows us to reduce the construction of a recovery sequence to the case of finite perimeter sets with a polyhedral boundary, the gluing procedure needs to be carefully verified only when the jump of $u$ occurs on an affine $(n-1)$-plane $V$ orthogonal to some $\xi \in \mathbb{S}^{n-1}$. This requires a completely new approach compared to those used in the aforementioned paper. The primary reason for this is that, initially, a recovery sequence for $u$ would require the use of optimal profiles related to the following problem:
\begin{equation}
\label{e:nonhomoptpro}
\inf_{u \in X_{a}^{b}(\xi)} \frac{1}{4} \int_{C \times \mathbb{R}} \int_{\mathbb{R}^m} J(h) |u(x+h) - u(x)|^2 \, dx \, dh + \int_{C \times \mathbb{R}} W(x, u(x)) \, dx,
\end{equation}
where $C$ is some $(n-1)$-dimensional cube in the orthogonal space to $\xi$, and $u \colon \mathbb{R}^n \to \mathbb{R}$ belongs to $X_{a}^{b}(\xi)$ if $\lim_{t \to -\infty} u(y+t\xi) = a$ and $\lim_{t \to \infty} u(y+t\xi) = b$ for almost every $y \in C$. In fact, nothing is currently known (even regarding existence) for minimizers of general non-homogeneous optimal profile problems of the form given in \eqref{e:nonhomoptpro}. Indeed, in \cite{alb-bel1}, only the case of a homogeneous potential $W$ was addressed, showing the existence of one-dimensional solutions, i.e., solutions depending only on the direction $\xi$. From a more PDE-oriented perspective, it is known that minimizers of \eqref{e:nonhomoptpro} satisfy a non-local version of the Allen-Cahn equation involving a suitable integro-differential operator $L_J$, depending on $J$, which can be viewed as a generalization of the classical fractional Laplacian. While the study of solutions to the Allen-Cahn equation has been a major area of research in recent decades, both in the local and non-local cases, all the relevant literature analyzes the homogeneous case (cf. \cite{silv1, silv2, caff2, caff1, dong, cozzi, kass, serra}). Therefore, we cannot rely on this strategy and must develop a different one.

In contrast to the Gamma-liminf analysis, our procedure for proving the optimality of the identified lower bound relies on a quantification of the singularity of the kernel $J$ at the origin, while preserving a behavior at infinity compatible with \eqref{e:growinfjintro}. Since our argument relies on gluing one-dimensional optimal profiles, we focus on the corresponding Euler-Lagrange equations, which are given by a family (depending on $ x \in \Omega $ and $ \xi \in \mathbb{S}^{n-1} $) of one-dimensional generalized Allen-Cahn equations of the form:
\begin{equation}
\label{e:onedimoptintro123}
-L_{J^\xi} \gamma + \partial_t W(x,\gamma) = 0.
\end{equation}
Here, $\gamma \colon \mathbb{R} \to \mathbb{R}$ belongs to $X_{a}^{b}$ with $a =z_1(x)$ and $b=z_2(x)$, whereas $L_{J^\xi}$ is a suitably defined integro-differential operator associated with the one-dimensional kernel $J^\xi$ (see \eqref{e:regopt3} and \eqref{e:kerneljxi1}, respectively). Relying on Proposition \ref{t:eximin}, we know that increasing solutions of \eqref{e:onedimoptintro123} always exist. As a further step in our strategy, we need to ensure that these increasing optimal profiles are also continuous. This leads us to consider two distinct scenarios: the case of integrable and non-integrable kernels. For the case of an integrable kernel, assuming $\|J\|_{L^1} = 1$, we can ensure continuity in the class of increasing optimal profiles under the invertibility of the function $Q_x \colon [z_1(x), z_2(x)] \to \mathbb{R}$ defined as
\[
Q_x(t) := \partial_t W(x,t) + t, \ \ \text{for every $x \in \Omega$}.
\]
Loosely speaking, we require the potential $W(x, \cdot)$ to exhibit limited concavity when restricted to the region between the wells. This condition ensures that the optimal profiles maintain the desired regularity. At this point, we restrict ourselves to noting that this hypothesis aligns with prior literature, where analogous assumptions have been employed to guarantee the regularity of optimal profiles (cf. \cite[Theorem 3.1]{bate}) and that the invertibility condition can be directly used in \eqref{e:onedimoptintro123} to infer the desired continuity (see Proposition \ref{p:regoptintk}).

In the case of non-integrable kernels, a more refined approach is necessary. Referring to Definition \ref{def:classJ} for the precise definition of the relevant class, we briefly note that any admissible kernel possesses a singularity at the origin comparable to $|h|^{-m-\eta}$ for some $\eta \in (0,1)$. In particular, these singularities still satisfy \eqref{e:regimedecay} in a neighborhood of the origin. By applying, for instance, the result in \cite{cozzi}, we can infer local H\"older regularity and, in particular, continuity for any increasing optimal profiles (see Proposition \ref{p:holderreg1}). In general, the class of non-integrable kernels considered in this paper is slightly smaller than the class analyzed in \cite{alb-bel1,alb-bel2}. A prototypical example of a singularity at the origin that does not fall within our analysis is $|h|^{-m}$. The reason for this lies in the fact that such kernels are neither integrable nor, loosely speaking, possess enough coercivity at the origin to enforce the H\"older continuity of optimal profiles. Nevertheless, since once continuity is established, our argument works under the sole one-side constraint $J(h) \lesssim |h|^{-m-\eta}$ around the origin, one possible way to include singularities of the type $|h|^{-m}$ in our analysis would be to develop additional tools to ensure continuity directly for milder singularities at the origin. This would likely require introducing a finer regularization argument tailored to handle such critical cases, which we leave as a potential direction for future research.

It is also worth noting that, in both the integrable and non-integrable cases, our regularity assumption on $t \mapsto W(x,t)$ around the wells plays a crucial role in ensuring the continuity of optimal profiles. As shown in \cite{alb-bel1}, for $\|J\|_{L^1}=1$, the function $\text{sign}(t)$ \emph{is an optimal profile for any potential $W(t) \geq 1-t^2$ regardless of the choice of $J$}.

The main reason why we need to ensure the continuity of optimal profiles is to ensure the validity of a \emph{strong comparison principle} for solutions of \eqref{e:onedimoptintro123} (cf. \cite{cozzi}). Specifically, let $\gamma_x \colon \mathbb{R} \to [z_1(x), z_2(x)]$ denote an increasing minimizer of \eqref{e:optprointro} with $a = z_1(x)$ and $b = z_2(x)$, corresponding to a point $x \in \Omega$. These minimizers, invariant under translations, are particular solutions in $BV(\mathbb{R}) \cap L^\infty(\mathbb{R})$ of the Euler-Lagrange equation \eqref{e:onedimoptintro123}. 

In this context, we establish a novel H\"older selection result ensuring that optimal profiles can be chosen to vary $\beta$-H\"older continuously with respect to the variable $x$, where the exponent $\beta$ depends solely on the spatial $\alpha$-H\"older regularity of the partial derivative of $W$ with respect to $t$. This regularity is, in turn, governed by the behavior of the moving wells. Theorem \ref{p:lipselx} establishes the precise relationship $\beta = \alpha^2$, which plays a crucial role in constructing recovery sequences by gluing together optimal profiles while accounting for the singularity of the kernel at the origin.

As already mentioned, while the regularity in time of solutions to integro-differential equations of the form \eqref{e:onedimoptintro123} has been extensively studied in the literature, to the best of our knowledge, no results addressing parameter-dependent regularity are present in the literature. Theorem \ref{p:lipselx} is established by combining two key ingredients. 

The first is indeed a strong comparison principle, employed in Proposition \ref{p:lipdep}, to derive an intermediate result establishing that such a selection can be made for any pair of points $x$ and $x'$ in $\Omega$, with a H\"older constant independent of the specific choice of points. 

The second is a profound result by Fefferman and Shvartsman in \cite{fef}. Briefly, given a positive integer $k$, let $F$ be a set-valued map from a metric space $X$ into the family of all \emph{compact convex subsets} of a Banach space $Y$ with dimension at most $k$. In broad terms, these authors establish a finiteness principle for the existence of a Lipschitz selection of $F$, providing the sharp value of the finiteness constant (see Theorem \ref{t:finitelipsel}). 

In our case, the main challenge in applying this selection result is based on the fact that equation \eqref{e:onedimoptintro123} does not impose a convex constraint. As detailed in the proof of Proposition \ref{p:lipsel}, we overcome this obstacle by passing to the family of inverse maps, constituting the domain of the conjugate functional, which is convex and serves as the appropriate tool to bridge the result established in Proposition \ref{p:lipdep} with that of Fefferman and Shvartsman. 

Recalling that $J(h) \lesssim |h|^{-m-\eta}$ in all cases, the condition $\eta < \alpha^2$ required in Theorem \ref{t:gammanint} ensures a crucial property of the constructed recovery sequence $(u_\ep)$. Specifically, in the non-local term 
\[
\frac{1}{4\ep} \int_{\Omega} \int_{\Omega} J_\ep(x-x') |u_\ep(x) - u_\ep(x')|^2 \, dx \, dx',
\]
only interactions of the form $|u_\ep(y + t\xi) - u_\ep(y + t'\xi)|^2$, namely, those aligned with the direction $\xi$, contribute to the limiting energy as $\ep \to 0$. Conversely, interactions orthogonal to $\xi$, namely $|u_\ep(y + t\xi) - u_\ep(y' + t\xi)|^2$, vanish as $\ep \to 0$ and do not influence the limiting functional. 

This analysis demonstrates that, as $\ep \to 0$, the energy increasingly resembles the homogeneous case, where the recovery sequence depends exclusively on the $\xi$-direction. In that scenario, interactions orthogonal to $\xi$ are null for every $\ep > 0$. 

 \subsection{Outline of the paper}

In Section \ref{s:sec2}, we formalize the problem and present the main results, introducing the notation, the double-well potentials, the sequence of functionals, the surface tension, and the class of admissible interaction kernels. Section \ref{sec:1d} focuses on the one-dimensional problem, addressing the existence and characterization of optimal profiles using a conjugate-functional approach. It also explores the regularity properties of optimal profiles for both integrable and non-integrable kernels and their continuous dependence on parameters, culminating in a H\"older selection principle for optimal profiles. In Section \ref{s:locdef}, we introduce the locality defect and $\epsilon$-traces. Section \ref{s:sec5} establishes compactness and introduces slicing techniques essential for the analysis. The construction of recovery sequences and the $\Gamma$-limsup inequality are addressed in Section \ref{s:sec6}, while Section \ref{s:gammalimif} is dedicated to the $\Gamma$-liminf inequality. Finally, Section \ref{s:dirblowup} examines the directional blow-up of the energy at a point and, in particular, the double liminf inequality.

\section{ Setting of the problem and main results}
\label{s:sec2}
 In this section we specify our general assumptions, collect a few preliminary observations, and provide a precise statement of our main results.

\subsection{Notation}
In what follows, $\Omega$ is a bounded domain in $\mathbb R^m$, $m\in \mathbb N$, namely an open, connected set. Balls of radius $r$ and center $x$ in $\mathbb R^m$ will be denoted by $B_r(x)$. We will use the notation $\mathbbm{1}_A$ to indicate the characteristic function of a set $A\subset \mathbb R^m$, namely $\mathbbm{1}_A(x)=1$ if $x\in A$ and $\mathbbm{1}_A(x)=0$ otherwise in $\mathbb R^m$. Given $\alpha \in (0,1]$, the $\alpha$-H\"older semi-norm of a function $f$ on $\Omega$ will be denoted by $\lfloor f \rfloor_{C^{0,\alpha}}(\Omega)$,  where we recall that 
\begin{equation}
    \lfloor f \rfloor_{ C^{0,\alpha}}(\Omega) := \sup_{h \neq h'\in\Omega} \frac{|f(h)-f(h')|}{|h-h'|^\alpha}.
\end{equation}
We will use classical notation for Sobolev and Lebesgue spaces, as well as for functions of bounded variation $(BV)$. The space of bounded Radon measures on $\R$ will be denoted by $\mathcal{M}_b(\R)$.

\subsection{Double-well potentials} 
\label{subs:dw}

  The presence of different material phases in $\Omega$ is encoded by a function $W \colon \mathbb{R}^m\EEE \times \mathbb{R}\to [0,+\infty)$, with $W:(x,t)\mapsto W(x,t)$.  Throughout the paper we will always assume that $W$ satisfies the following conditions: 
 \begin{enumerate}[label=(H.\arabic*)]
 \setcounter{enumi}{0}
 \item \label{H1}
 $W$ is continuous, it is twice continuously-differentiable  with respect to \EEE $t$, and the partial derivative $\partial_t W$ is 
  locally $\alpha$-H\"older-continuous for some $\alpha \in \left(\frac12,1\right]$. 

    $W(x,t) =0$ if and only if $t \in \mathbf{z}(x):= \{z_1(x), z_2(x) \}$ for every $(x,t) \in \mathbb{R}^m\EEE \times \mathbb{R}$, where the wells $z_i \colon \mathbb{R}^m\EEE \to \mathbb{R}$, $i=1,2$, are  $\alpha$-H\"older continuous and $\alpha$ is the same exponent as above.
    \end{enumerate}
 
 The growth of $W$ from above and below will be encoded by an \UUU increasing function $f \colon [0,\infty) \to [0,\infty)$ which is quadratic in a neighbhorhood of the origin \EEE and with $f(t)=0$ if and only if $t=0$. To be precise, we will assume that for every compact set $K \subset \mathbb{R}^m\EEE$ there exists $\delta_K \in (0,1]$ satisfying
\begin{enumerate}[label=(H.\arabic*)]
\setcounter{enumi}{1}
    \item \label{H2}
    $
    \min_{x \in K} |z_1(x)-z_2(x)| > 8\delta_K
    $
    and, in particular, $z_1(x)<z_2(x), \ \ \text{ for every $x \in K$;}$
    \item \label{H3} 
     $
     \delta_K f(\text{dist}(t;\mathbf{z}(x))) \leq W(x,t) \leq \frac{1}{\delta_K}f(\text{dist}(t;\mathbf{z}(x))), \ \ \text{ for every}(x,t) \in K \times \mathbb{R},
     $
    \item \label{H4}$\partial^2_t W(x,t) \geq \delta_K$ for every $x \in K$ and $t \in [z_1(x)-\delta_K,z_1(x)+\delta_K] \cup [z_2(x)-\delta_K,z_2(x)+\delta_K]$,
    \item \label{H5}  $W(x,t) \geq \delta_K|t|$, for every $x \in K$ and every $|t| >\frac{1}{\delta_K}$.
    %\item there exist $R,S >0$ such that for every $x \in \mathbb{R}^m\EEE$ and for every $p \in \mathbb{R}^m \setminus B_R(0)$ we have $W(x,p) \geq S |p|$ and in addition 
    %\[
    %|z_i(x)-p|\leq R \ \ \text{ implies } \ \ W(x,p)= |z_i(x)-p|^2, \ \ i=1,\dotsc,k,
    %\]
    %for every $(x,p) \in \mathbb{R}^m\EEE \times \mathbb{R}^m$.  Em: (H3) can be probably weakened... \EEE
\end{enumerate}
\EEE 

\begin{remark}[On \ref{H1}--\ref{H5}]
    \label{r:convimpl}
    Assumptions \ref{H1} and \ref{H4} imply that $W(x,\cdot)$ is increasing on $[z_1(x),z_1(x)+\delta_K]$ and decreasing on $[z_2(x)-\delta_K,z_2(x)]$.
    Condition \ref{H5} guarantees that for every compact set $K\subset \mathbb{R}^m\EEE$ \EEE there exists $M_K>0$ such that $W(x,\pm M_K) \leq W(x,t) $ for every $|t| \geq M_K$ and $x \in K$. \UUU The assumption on the $\alpha$-H\"olderianity of $\partial_t W$ is coherent with the assumption on the $\alpha$-H\"olderianity of the moving wells $\{z_1(x),z_2(x)\}$ in \ref{H1}. Indeed, under a stronger assumption than \ref{H3}, where 
    \[
    W(x,t) \simeq f(\emph{dist}(t;{\bf z}(x)))= f(\emph{min}\{|t-z_1(x)|,|t-z_2(x)|\}),
    \]
    it is straightforward to verify that, since $\frac{d}{ds}f(s)= O(|s|)$ as $s \to 0$, then 
    \[
     \partial_t W(t,x) =f_1(t,x) (t-z_1(x)) + f_2(t,x)(t-z_2(x)),
    \]
    for a pair of locally Lipschitz-continuous functions $f_1$ and $f_2$.\EEE
     We eventually point out that $W$ and $z_1,z_2$ are defined in the whole space in order to avoid unnecessary technicalities for points approaching the boundary of the domain of definition.
\end{remark}

\subsection{The sequence of functionals and the surface tension}

Let $J \colon \mathbb{R}^m\EEE \to (0,\infty)$ be measurable, and let $W$ satisfy \ref{H1}--\ref{H5}. For every $\ep >0$, we define
\begin{equation}
    \label{e:kernelres}
    J_\ep(h) := \ep^{-n}J(h/\ep), \ \ \text{ for all $h \in \mathbb{R}^m\EEE$,}
\end{equation}
and we consider the following energy functional
\begin{equation}
    \label{e:functional}
    F_\ep(u;\Omega):= \frac{1}{4\ep} \int_{\Omega \times \Omega} J_{\ep}(x'-x)|u(x')-u(x)|^2 \, dx'dx + \frac{1}{\ep} \int_{\Omega}W(x,u(x))\, dx,
\end{equation}
defined for every function $u \in L^1(\Omega)$.

 Our main result concerns the identification of the $\Gamma$-limit of the sequence $(F_{\ep})$ in the strong $L^1$-topology and under suitable regularity assumptions on the kernel $J$. We first introduce the space of all possible limiting configurations.

\begin{definition}[The space $CP(\Omega;\mathbf{z}(x))$]
    We say that a measurable function $u \colon \Omega \to \mathbb{R}$ belongs to $CP(\Omega;\mathbf{z}(x))$ if there exists a finite perimeter set $E$ of $\Omega$ such that $u=z_1$ a.e. on $E$ and $u=z_2$ a.e. on $\Omega \setminus E$.
\end{definition}

If $u \in CP(\Omega;\textbf{z}(x))$, due to the continuity of the maps $z_i \colon \Omega \to \mathbb{R}$ and due to the structural properties of finite perimeter sets, for $\mathcal{H}^{n-1}$-a.e. $x \in \Omega$ we have $x \in J_u$ if and only if $x \in \partial^* E$, $\nu_u(x)= \pm \nu_{E}(x)$, and $[u](x) = z_1(x) - z_2(x)$. Nevertheless, in general it is not true that $CP(\Omega;\mathbf{z}(x)) \subset BV(\Omega)$, due to the fact that the maps $z_1,z_2$ are, a priori, not assumed to have bounded variation.

The limiting functional is given in terms of a  surface tension. In order to write its definition we first need to introduce auxiliary one-dimensional energy functionals.
For every $\xi \in  \mathbb{S}^{m-1}, \EEE$ define $J^\xi \colon \mathbb{R}\EEE \to [0,\infty]$ as
\begin{equation}
    \label{e:kerneljxi1}
    J^{\xi}(t) := \int_{\xi^\bot} J(y+t\xi) \, dy\quad\text{for all }t\in\mathbb{R}.
\end{equation}
We consider for every $(x,\xi) \in \overline{\Omega} \times  \mathbb{S}^{m-1} \EEE$ the functional
\begin{equation}
    \label{e:functionalxi148}
    F^\xi_{x}(\gamma) := \frac{1}{4} \int_{\mathbb{R} \times \mathbb{R}} J^\xi(t'-t) |\gamma(t')-\gamma(t)|^2 \, dt'dt
+\int_{\mathbb{R}} W(x,\gamma(t)) \, dt,
\end{equation}
for every measurable function $\gamma \colon \mathbb{R} \to \mathbb{R}$. 
We further introduce a notation for the space of admissible profiles attaining two given asymptotic values. For $a,b\in\R$, we set
\begin{equation}
\label{eq:Xab}
X_a^b:=\{\gamma\in L^1(\R):\,\lim_{t\to -\infty} \gamma(t)=a\text{ and }\lim_{t\to +\infty} \gamma(t)=b\}.
\end{equation}

\UUU
\begin{remark}
    To be precise, since we only assumed $\gamma \in L^1(\mathbb{R})$, the condition $\lim_{t \to +\infty} \gamma(t)=b$ should be intended as follows: for every $\epsilon >0$, there exists $M >0$ such that $\mathcal{L}^1(\{t \in [M,\infty) \ | \ |\gamma(t) -b| > \epsilon\}) =0$. Analogously for the condition $\lim_{t \to -\infty} \gamma(t)=a$.
\end{remark}
\EEE

Our limiting surface tension is defined according to the following infimum problem. 

\begin{definition}[The surface tension $\sigma$]
\label{def:surf}
    We define the surface tension $\sigma \colon \mathbb{R}^m\EEE \times  \mathbb{S}^{m-1} \EEE \to (0,\infty)$ as
\begin{equation}
\label{e:surftension}
    \sigma(x,\xi):= \inf_{\gamma \in X_{z_1(x)}^{z_2(x)}} F^\xi_{x}(\gamma).
\end{equation}
\end{definition}

Our $\Gamma$-limit will be given by the following functional.
\begin{equation}
\label{e:limit}
    F_0(u):=
    \begin{cases}
        \int_{J_u} \sigma(x,\nu_u(x)) \, d\mathcal{H}^{n-1}(x), &\text{ if } u \in CP(\Omega;\mathbf{z}(x)) \\
        +\infty, &\text{ if } u \in L^1(\Omega) \setminus CP(\Omega;\mathbf{z}(x)).
    \end{cases}
\end{equation}

\subsection{Admissible interaction kernels} 
\label{subs:adminterk}
 We collect here our assumptions for the interaction kernels $J$. We will always work with even kernels, namely such that $J(-h)=J(h)$ for every $h\in \R^m$. In our main result we will distinguish between two cases: we will either assume $J$ satisfy
 \begin{equation}
    \label{K1} 
    \int_{\R^m} J(h)(1+|h|)\,dh<+\infty,
    \end{equation}
 or we will require $J$ to belong to the class of not-$L^1$-integrable admissible kernels defined below.

\begin{definition}[The class $\mathcal{L}_m(\eta,\lambda,\rho)$]
\label{def:classJ}
 Let $J \colon \mathbb{R}^m \to (0,\infty)$ be even. Let $\eta,\lambda,\rho \in (0,1)$. We say that $J$ belongs to the class $\mathcal{L}_m(\eta,\lambda,\rho)$ if there exists $\tilde{J} \in L^1(\mathbb{R}^m)$ such that
\begin{align}
    \label{e:classlm1}
    &\frac{\lambda}{|h|^{m+\eta}} \leq  J(h) \leq \frac{1}{\lambda |h|^{m+\eta}}+\tilde{J}(h), \ \ \text{ for every $h \in B_\rho(0) \setminus \{0\}$}\\
    \label{e:classm2}
    & \int_{\mathbb{R}^m \setminus B_{\rho}(0)} J(h)|h| \, dh < \infty.
\end{align}
    
\end{definition}

 We observe that if $J\in \mathcal{L}_m(\eta,\lambda,\rho)$, then $J\notin L^1(\mathbb R^m)$. On the other hand, kernels $J\in \mathcal{L}_m(\eta,\lambda,\rho)$ satisfy the following integrability condition: 

\begin{enumerate}[label=(K.\arabic*)]
\vspace{1mm}
  \item
    \label{K2}$\int_{\mathbb{R}^m} J(h)|h| \, dh < \infty$.
\end{enumerate}
\vspace{1mm}
\UUU As explained in the introduction, condition \ref{K2} is fundamental for deriving the continuity property of the $\epsilon$-traces, as described in Section \ref{s:locdef}. In particular, this condition ensures that our limiting energy can be characterized as a local functional.\EEE

\subsection{Main results}
\UUU In this subsection, we rigorously state the main results of the paper and briefly outline the intermediate steps developed throughout the paper, which are instrumental in their proofs. \EEE

The first part of our analysis concerns the 1D setting. In particular, for $x\in \R^m$, we study minimizers in $X_{z_1(x)}^{z_2(x)}$ (recall \eqref{eq:Xab}) of the problem
\begin{equation}
\label{e:1d-intro}
       F_x(\gamma):=\frac{1}{4}\int_{\mathbb{R}\times \mathbb{R}} J(t'-t)|\gamma(t')-\gamma(t)|^2 \, dt'dt + \int_{\mathbb{R}} W(x,\gamma(t)) \, dt.
         \end{equation}
Our first main contribution reads as follows.

\begin{theorem}
\label{thm:sel-main}
    Let $K \subset \Omega\subset  \mathbb{R}^m$ be compact, let $J \colon  \mathbb{R} \to (0,+\infty)$ be even, and let $W \colon \Omega \times \mathbb{R} \to [0,\infty)$ satisfy \ref{H1}--\ref{H4}.  Let us further denote for every $x \in K$ 
   \[
   O_x := \{ \gamma \in X_{z_1(x)}^{z_2(x)} \ | \ \emph{$\gamma$ is increasing and continuous, and minimizes \eqref{e:1d-intro}} \}.
   \]
  Assume further that one of the following conditions is satisfied
    \begin{enumerate}[(i)]
       \item  $Q_x(t):= \partial_t W(x,t)+t$ has a continuous inverse for every $x \in K$ and $t \in [z_1(x),z_2(x)]$, and
       \begin{equation*}
          \int_{\mathbb{R} \setminus (-1,1)} J(h)|h| \, dh < \infty  \ \ \text{ with } \ \  \|J\|_{L^1(\mathbb{R})}=1
       \end{equation*}
       \item $J \in \mathcal{L}_1(\eta,\lambda,\rho)$ for some $\eta,\lambda,\rho \in (0,1)$.
   \end{enumerate}
   Then, the class $O_x$ is nonempty for every $x\in K$. Additionally, there exists a weight function $\overline{\sigma}_J \colon \mathbb{R} \to (0,+\infty)$ (depending also on $K$)
   such that, for every $\alpha \in (1/2,1]$  we find a family $(\gamma_x)_{x \in K}$ satisfying $\gamma_x \in O_x$ and a
   constant $L= L(K,J,\lfloor z_1 \rfloor_{C^{0,\alpha}}(\Omega),\lfloor z_2 \rfloor_{C^{0,\alpha}}(\Omega)) > 0$ satisfying 
   \begin{equation*}
            \|\gamma_x -\gamma_{x'}\|_{L^1(\mathbb{R};\overline{\sigma}_J)} \leq L\, \lfloor\partial_t W \rfloor_{C^{0,\alpha}(K \times [-M,M])} |x-x'|^{\alpha^2}, \ \ \text{ for $x,x' \in K$},
         \end{equation*}
          where $M:= 2\max_{x \in K} |z_1(x) \vee z_2(x)|$.
    \end{theorem}
The proof of Theorem \ref{thm:sel-main} is the focus of Section \ref{sec:1d} and relies on several intermediate results, which are proven in dedicated subsections. The existence of increasing minimizers for \eqref{e:1d-intro} is shown in Proposition \ref{t:eximin} in Subsections \ref{subs:existence1d}--\ref{subs:conj1d} by slightly generalizing previous results obtained in \cite{alb-bel1}. The regularity of such minimizers is the subject of Propositions \ref{p:regoptintk} and \ref{p:holderreg1}, in Subsections \ref{subs:cont1d-integrable} and
\ref{subs:cont1d-not-integrable}, respectively, obtained by extending some arguments in \cite{cozzi}. In particular, by combining Proposition \ref{t:eximin} with Propositions \ref{p:regoptintk} and \ref{p:holderreg1} we obtain the first part of the theorem, namely that the set $O_x$ is nonempty for every $x\in K$. The second part of the theorem follows by Proposition \ref{p:lipselx}, which in turn is the main result of Subsection \ref{subs:sel-1d}. To prove this latter characterization, we essentially rely on an abstract argument  established in \cite{fef}, as well as on an intermediate continuous dependence analysis established in Proposition \ref{p:contdep} and Subsection \ref{subs:cont-dep1d}.

Our second main result concerns the asymptotic behavior of the functionals $(F_\ep)_\ep$.
We split it into two theorems corresponding to the case of non-integrable and $L^1$-kernels, respectively. \EEE

\begin{theorem}[$\Gamma$-limit with non-integrable interaction kernels]
\label{t:gammanint}
    Let $W \colon \mathbb{R}^m\EEE \times \mathbb{R} \to [0,\infty)$ satisfy \ref{H1}--\ref{H5}, and let $\alpha \in \left(\frac12,1\right]$ be given by \ref{H1}. Assume in addition that $J\in\mathcal{L}_m(\eta,\lambda,\rho)$ for  $\eta \in (0,\alpha^2)$, $\lambda \in (0,1]$, and $\rho >0$. Then $(F_\ep)$ $\Gamma$-converges to $F_0$ in the $L^1$-topology.  
\end{theorem} 
\begin{theorem}[$\Gamma$-limit with integrable interaction kernels]
\label{t:gammaint}
    Let $W \colon \mathbb{R}^m\EEE \times \mathbb{R} \to [0,\infty)$ satisfy \ref{H1}--\ref{H5}. Assume in addition that $Q_x(t):= \partial_t W(x,t) + t$ has a continuous inverse for every $x \in \Omega$ and $t \in [z_1(x),z_2(x)]$, and that  $J$ satisfies \eqref{K1},
    with $\|J\|_{L^1(\mathbb{R}^m\EEE)}=1$. Then $(F_\ep)$ $\Gamma$-converges to $F_0$ in the $L^1$-topology.  
\end{theorem}

\UUU We conclude this section by providing a more detailed explanation of the proofs of the $\Gamma$-liminf and $\Gamma$-limsup inequalities, as well as indicate where the main results are established within the paper. 

The derivation of the $\Gamma$-liminf inequality is common to both theorems and, importantly, does not depend on the integrability assumptions of the kernel near the origin. Instead, it relies solely on the integrability assumption of $J$ at infinity, as provided by condition \ref{K3}. The $\Gamma$-liminf inequality is established in Section \ref{s:gammalimif}, with its proof centered on a core technical result called the \emph{double liminf inequality}, stated in Proposition \ref{c:est_gamminf}. The proof of the double liminf inequality occupies the entire Section \ref{s:dirblowup}. Here, we construct a family of asymptotic calibrations for the energies (cf. Subsections \ref{sub:8.1} and \ref{sub:8.2}), created by averaging the family of non-local one-dimensional functionals introduced in Definition \ref{d:onedimcalibration}. These calibrations ultimately satisfy the lower bound established in Proposition \ref{p:calibration}. In the final part of the section, specifically Subsection \ref{sub:doubleliminf}, these calibrations are used to explicitly perform a (directional) blow-up of the energy. This step combines two key results, Propositions \ref{p:oscbound1} and \ref{p:oscbound2}, to develop the technical tools culminating in the proof of Lemma \ref{l:est_gammainf} and, consequently, the aforementioned double liminf inequality.

The proof of the $\Gamma$-limsup upper bound relies on the fact that, using a density argument for finite perimeter sets and the upper semicontinuity of the surface tension (cf. Definition \ref{def:surf} and Proposition \ref{p:surftension}), it is possible to reduce the construction of a recovery sequence to the case of a finite perimeter set that is polyhedral (cf. Proposition \ref{t:recpol} and Theorem \ref{t:recgen}). A key challenge addressed in the proof of Proposition \ref{t:recpol} is gluing together optimal profiles in a sufficiently regular manner. This is where the proofs of Theorems \ref{t:gammanint} and \ref{t:gammaint} primarily differ. Under integrability assumptions on the interaction kernel $J$, the regularity required in the gluing procedure can be ensured under the simple invertibility assumption of the map $Q_x(t) = \partial_t W(x,t) + t$, which also guarantees the continuity of increasing optimal profiles (cf. Proposition \ref{p:regoptintk}). In contrast, the case of a non-integrable kernel is more delicate. Here, we establish in Theorem \ref{p:lipselx} an $\alpha^2$-H\"older selection of optimal profiles, where $\alpha$ is the H\"older regularity of the moving wells. The relationship between the parameter $\eta$, which prescribes the singularity of the kernel $J$ at the origin, and the parameter $\alpha$ is given by $\eta < \alpha^2$. Theorem \ref{p:lipselx} hinges on a comparison principle for optimal profiles (cf. Proposition \ref{p:lipdep}), combined with Theorem \ref{t:finitelipsel} borrowed from \cite{fef}. Finally, the key tool linking these last two results is the notion of the conjugate functional (cf. Subsection \ref{subs:conj1d}), which, roughly speaking, transforms the non-convex constraint represented by the Euler-Lagrange equation for optimal profiles \eqref{e:regopt2} into a convex one. This is a necessary step for the applicability of Theorem \ref{t:finitelipsel}.
\EEE

 \section{The 1D-Problem}
 \label{sec:1d}
This section is devoted to the proof of Theorem \ref{thm:sel-main}, and is organized as follows. Subsections \ref{subs:existence1d}--\ref{subs:cont-dep1d} deal with the 1D autonomous problem, namely consider the case in which $m=1$ and the double-well potential is independent of $x$. 
In particular, Subsections \ref{subs:existence1d} and \ref{subs:conj1d} recall results obtained in \cite{alb-bel1} and extend the existence theory to slightly more general integrability assumptions on the interaction kernels. In Subsection \ref{subs:existence1d} we show existence of optimal profiles for the autonomous problem, in Subsection \ref{subs:conj1d} their characterization via a conjugate functional. In Subsections \ref{subs:cont1d-integrable} and \ref{subs:cont1d-not-integrable}, we prove a first regularity result for optimal profiles in the autonomous case. The argument there are a generalization of the ideas in \cite{cozzi}.
Subsection \ref{subs:cont-dep1d} contains our first completely original result, namely a continuous dependence analysis for 1D optimal profiles of a family of autonomous non-local energies.  Eventually, in Subsection \ref{subs:sel-1d} we move to the non-autonomous case and prove Theorem \ref{thm:sel-main}.

\subsection{Existence}
\label{subs:existence1d}
This subsection is devoted to prove and recall some facts regarding the optimal profile problem in the autonomous case. We follow the approach in \cite{alb-bel1} which relies on the notion \emph{conjugate functional}.

Consider a kernel $J \colon \mathbb{R} \to [0,\infty)$ which is an even function and satisfies
\begin{enumerate}[label=(K.\arabic*)]
\setcounter{enumi}{1}
\vspace{1mm}
    \item\label{K3}
   $\int_{\mathbb{R} \setminus B_\rho(0)} J(h)|h| \, dh < \infty$ \ \ \text{ for some $\rho >0$},
\end{enumerate}
as well as a homogeneous double-well potential $W \colon \mathbb{R}\to [0,\infty)$ such that $W$ is continuous, \EEE
\begin{align}
\label{e:assu1}
W(t)=0 \ \ &\text{if and only if} \ \ t \in \{a,b\} \ \ \text{for some } a<b\\
\label{e:assu2}
& \ \ \ \ \ \ \lim_{t \to \pm \infty} W(t)= +\infty.
\end{align}
By letting the one-dimensional functional $F$ defined as
\begin{equation}
\label{e:onedim100}
F(\gamma) := \frac{1}{4} \int_{\mathbb{R} \times \mathbb{R}} J(t'-t) (\gamma(t')-\gamma(t))^2 \, dt'dt + \int_{\mathbb{R}} W(\gamma(t)) \, dt, \ \ \gamma \in L^1(\mathbb{R}),
\end{equation}
 our main existence result in this setting  reads as follows.

\begin{proposition}[Existence of minimizers for double well potentials]
\label{t:eximin}
     Let $J$ be even and satisfying \ref{K3}, and let $W$ be continuous and such that \eqref{e:assu1}--\eqref{e:assu2} hold true. 
    Assume also \EEE that 
    \begin{equation}
    \label{e:onedimmin}
\inf \{F(\gamma) \ | \ \gamma \in X_a^b  \} < \infty.
    \end{equation}
    Then the minimum problem \eqref{e:onedimmin} has a solution.  To be precise, there exists \EEE a minimizer $\gamma \in X_a^b$ which is increasing and satisfies $\gamma(t) \geq 0$ for $t >0$, $\gamma(t) \leq 0$ for $t < 0$.    
\end{proposition}
We point out that an analogous existence result has been established in \cite[Theorem 2.4]{alb-bel1} under the following  stronger \EEE integrability condition on $J$
    \begin{equation}
    \label{e:minalbbel}
    \int_{\mathbb{R}} J(h)(1+|h|) \, dh < \infty.
    \end{equation}
 The extension to the class of kernels fulfilling only \ref{K3} follows by a truncation argument which, for convenience of the reader, we detail below.  
\begin{proof}[Proof of Proposition \ref{t:eximin}]
     For kernels satisfying \ref{K3} but, a priori, not being in $L^1(\R)$, \EEE we consider for every  $N\in \mathbb{N}$ \EEE the  truncated \EEE kernel $J^N \colon \mathbb{R} \to [0,\infty)$ defined as as $J \wedge N$, and we denote by $F^N$ the associated energy functional. Since $J^N$ satisfies \eqref{e:minalbbel}, we already know that there exists an increasing function $\gamma^N $ in  $X_{a}^{b}$ \EEE such that $\gamma^N(t) \geq (a+b)/2$ for almost every $t \geq 0$ while $\gamma^N(t) < (a+b)/2$ for almost every $t < 0$ and
    \[
    \begin{split}
   F^N(\gamma^N)= \inf_{\gamma \in X_{a}^{b}} \frac{1}{4}\int_{\mathbb{R}\times \mathbb{R}} J^{N}(t-t') |\gamma(t)-\gamma(t')|^2 \, dtdt' + \int_\mathbb{R}  W\EEE(\gamma(t)) \, dt.
    \end{split}
    \]
  Since $\gamma^N$ is an increasing function in $X_a^b$, then $\gamma^N \in BV(\mathbb{R})$ and  $|D\gamma^N|(\mathbb{R})= |a-b|$ for every $N$. In particular, we deduce that there exists  a map $\gamma^0\in BV(\mathbb{R})$ and \EEE a subsequence $(N_m)$ such that 
    \[
    \gamma^{N_m}(t) \to \gamma^0(t), \ \ \text{ for a.e. $t$} \ \ \text{ and } \ \ D\gamma^{N_m} \rightharpoonup D\gamma^0, \ \ \text{ weakly* in  }\mathcal{M}_b(\R)\EEE.
    \]
    Since $F^N \leq F^{N+1}$ for every $N$,  from \EEE \eqref{e:onedimmin}  we \EEE infer $$\sup_{N} F^N(\gamma^N) \leq \inf_{\gamma \in X_a^b} F(\gamma) < \infty.$$ In particular, since  $ W>0$ \EEE for  $t \in (a,b)$, an application of Beppo-Levi's theorem   yields \EEE that also $\gamma^0$ is increasing and belongs to $X_{a}^{b}$. By further applying Fatou's lemma we  deduce that, up to a further non-relabelled subsequence, there holds \EEE
    \[
    \begin{split}
     \sup_{N } F^N(\gamma^N)
   &\geq \liminf_{m \to \infty}\frac{1}{4}\int_{\mathbb{R}\times \mathbb{R}} J^{N_m}(t-t') |\gamma^{N_m}(t)-\gamma^{N_m}(t')|^2 \, dtdt' + \int_\mathbb{R}  W\EEE(\gamma^{N_m}(t)) \, dt\\
    &\ \ \ \ \ \ \ \ \ \ \ \geq \frac{1}{4}\int_{\mathbb{R}\times \mathbb{R}} J(t-t') |\gamma^0(t)-\gamma^0(t')|^2 \, dtdt' + \int_\mathbb{R}  W\EEE(\gamma^0(t)) \, dt.
    \end{split}
    \]
    On the other hand, by considering $\gamma \in X_{a}^{b}$, since $\gamma$ is admissible in the minimization problem at level $N$  for every $N\in \mathbb{N}$, \EEE we also have
    \begin{align*}
    F(\gamma)
   &= \sup_{N }\frac{1}{4}\int_{\mathbb{R}\times \mathbb{R}} J^{N}(t-t') |\gamma(t)-\gamma(t')|^2 \, dtdt' + \int_\mathbb{R}  W(\gamma(t))\EEE \, dt\geq \sup_{N }F^N(\gamma^{N}).
    \end{align*}
    Hence we have proved that 
    \begin{equation}
    \label{e:supeq1}
    F(\gamma^0) \leq F(\gamma), \ \ \text{ for every $\gamma \in X_a^b$},
    \end{equation}
which yields the desired conclusion.
\end{proof}

\begin{remark}
    \label{r:eximin}
    If we  replace \ref{K3} with 
    $$\int_\R J(h)|h|dh<+\infty,$$ \EEE then \eqref{e:onedimmin} is satisfied (consider for example $\gamma= a \mathbbm{1}_{(-\infty,0)} + b \mathbbm{1}_{[0,\infty)}$). Therefore, under this additional integrability assumption on $J$ (which is anyhow weaker than \eqref{e:minalbbel}), the existence of a minimizer is always guaranteed by Proposition \ref{t:eximin}. 
\end{remark}

We conclude this subsection by writing down the Euler-Lagrange equation for minimizers of \eqref{e:minalbbel}.  \UUU Under the integrability assumption \ref{K2} on the kernel, if we further assume that $J$ is an even function and
that \EEE
\begin{equation}
\label{e:regopt1.1w}
\text{$W \colon \mathbb{R} \to [0,\infty)$ is cont. diff. satisfying \eqref{e:assu1}--\eqref{e:assu2} with $a=-1$, $b=1$},
\end{equation}
due to the simmetry of the kernel, it is possible to verify that the Euler-Lagrange equation for minimizers $\gamma \in X_{-1}^1 \cap BV(\mathbb{R})$ of \eqref{e:onedim100} writes as
\begin{equation}
    \label{e:regopt2}
    L_J\gamma = \dot{W}(\gamma), \ \ \text{ in the weak sense on $\mathbb{R}$},
\end{equation}
where $L_J$ is the non-local integral operator acting as \UUU (cf. \cite{cozzi,caff2}) \EEE
\begin{equation}
    \label{e:regopt3}
    L_J\gamma(t_0) := \frac{1}{2}\int_{\mathbb{R}} (\gamma(t_0+t)+\gamma(t_0-t)-2\gamma(t_0))J(t) \, dt.
\end{equation}

Observe that
\[
L_J \gamma \in L^1(\R) \quad \text{ and } \quad \int_{\R}L_J \gamma \, dt \leq  2|D\gamma|(\mathbb{R}) \int_{\R} J(t)|t| \, dt.
\]
In particular, under the integrability condition \ref{K2} on $J$, for every $\gamma \in BV(\R)$ the expression in \eqref{e:regopt3} is well defined for a.e. $t_0 \in \mathbb{R}$.

\subsection{Characterization via the conjugate functional}
\label{subs:conj1d}
In this subsection we assume \eqref{e:assu1}--\eqref{e:assu2} and that $J \colon \mathbb{R} \to [0,\infty)$ is an even function satisfying \eqref{e:minalbbel}. We summarize here the characterization of increasing optimal profiles established in \cite{alb-bel1}. The reader who is already familiar with the theory of non-local phase transitions can skip this subsections and resume directly from Subsection \ref{subs:cont1d-integrable}.

We begin our review of the results in \cite{alb-bel1} with a definition.
\begin{definition}
Given a function $\gamma \colon \mathbb{R} \to [a,b]$ we define $\gamma^{-1}\colon (a,b) \to \mathbb{R}$ as 
\[
\gamma^{-1}(t) := \inf \{s \in \mathbb{R} \ | \ \gamma(s) \geq t \}
\]
\end{definition}
Notice that $\gamma^{-1}(\gamma(x))= \inf \{t' \in \mathbb{R} \ | \ \gamma(t')=\gamma(t) \}$ for every increasing function, while $\gamma^{-1}(\gamma(t))=t$ for every strictly increasing function.

As in \cite{alb-bel1} we consider now the conjugate functional $F^\circ$ defined on every increasing function $v \colon (a,b) \to \mathbb{R}$ as
\[
F^\circ(v) := \int_{a <t<t'<b} K(v(t)-v(t')) \, dtdt' + \int_{a <t<b} W(t) \, d\dot{v}(t),
\]
where $\dot{v}$ is the distributional derivative of $v$ and $K \colon \mathbb{R} \to \mathbb{R}$ is given by
\[
f(t):= \max \{-t,0 \}, \qquad K(t):= J * f(t).
\]
By definition, $F^\circ$ is convex on $Y$. Additionally, the following identity holds true.

\begin{proposition}[Theorem 2.11 in \cite{alb-bel1})]
    \label{t:Fo=F}
    Under assumptions \eqref{e:assu1}--\eqref{e:assu2}, for every increasing function $u \in X_a^b$ there holds $F^\circ(u^{-1})=F(u)$.
\end{proposition}

As an outcome of the proof of the above result, we have in particular the following characterization.

\begin{proposition}[Proof of Theorem 2.11 in \cite{alb-bel1})]
\label{p:changevar}
    Suppose that $G \colon [a,b] \to [0,\infty)$ is a continuous function such that $G(a)=G(b)=0$. Then, given an increasing function $\gamma \in X_a^b$, the following change of variables formula holds true
    \begin{equation}
        \label{e:changevar}
        \int_{\mathbb{R}} G(\gamma(t)) \, dt = \int_a^b G(t) \, d\dot{\gamma}^{-1}(t).
    \end{equation}
\end{proposition}

The following operator provides a non-local counterpart to the classical conformal factor for local phase transitions.

\begin{definition}[Operator $H$]
\label{d:hdef}
By letting $Y$ denote the class of all increasing real-valued functions on $(a,b)$, we associate to any $v \in Y$ the function $Hv \colon (a,b) \to \mathbb{R}$ defined by
\begin{equation}
\label{e:hoperator}
Hv(s):= \int_{s<t<b} \int_{a <t'<s} -\dot{K}(v(t)-v(t')) \, dtdt'.
\end{equation}
\end{definition}

Denoting by $C_0([a,b])^+$ the $L^\infty$ closure of all positive and continuous functions with compact support in $(a,b)$, we have the following proposition.

\begin{proposition}[ Proposition 2.13 in \cite{alb-bel1})]
\label{p:unbounded}
The (non-linear) operator $H$ given by \eqref{e:hoperator} satisfies $H : Y \to C_0([a,b])^+$ and 
\[
Hv(s) \leq \bigg[1- \bigg(\frac{2s -a-b}{b-a} \bigg) ^2\bigg]\int_{\mathbb{R}} J(t) \, dt , \ \ s \in (a,b).
\]
Moreover $Hv$ is strictly positive in $(a,b)$ whenever $J$ has unbounded support or $v$ is continuous.
\end{proposition}

In \cite{alb-bel1} the authors showed the following fundamental characterization of increasing minimizers of \eqref{e:onedim100} in terms of the operator $H$.

\begin{proposition}[Corollary 2.16 in \cite{alb-bel1})]
\label{t:charmin}
    Let $\sigma \in Y$. Then $\sigma$ minimizes $F^\circ$ in $Y$ if and only if 
    \begin{align}
    &W \geq H\sigma, \ \ \text{ everywhere in $[a,b]$}\\
    &W = H\sigma, \ \ \text{ everywhere in the support of the measure $\dot{\sigma}$}.
    \end{align}
    In particular, an increasing function $\gamma \in X_a^b$ minimizes $F$ if and only if $\sigma:= \gamma^{-1}$ minimizes $F^\circ$ in $Y$.
\end{proposition}

\subsection{Regularity: The case of integrable kernels}
\label{subs:cont1d-integrable}

In this subsection we prove a first regularity result for bounded solutions $\gamma$ of \eqref{e:regopt2} under the following additional assumptions on $J$ and $W$:
    \begin{align}
         \label{e:regopt6.1234}
        & \ \ \ \ \ \ \ \ \ \ \ \int_{\mathbb{R} \setminus B_{\rho}(0)} J(h)|h| \, dh < \infty \ \ \text{and} \ \ \|J\|_{L^1(\mathbb{R})}=1 \\
         \label{e:regopt6.1234.1}
        &\text{$Q(t):= \dot{W}(t)+t$ admits a continuous (left) inverse $P \colon \mathbb{R} \to [-1,1]$}
    \end{align}
   Our result reads as follows.

    \begin{proposition}
        \label{p:regoptintk}
        Under assumptions \eqref{e:regopt1.1w} and \eqref{e:regopt6.1234}--\eqref{e:regopt6.1234.1}, any increasing minimizers of \eqref{e:onedim100} in the class $X_{-1}^1$ (see \eqref{eq:Xab}) is continuous.  
    \end{proposition}
    \begin{proof}
        By virtue of the assumption on $\|J\|_{L^1(\mathbb{R})}=1$, the Euler-Lagrange equation \eqref{e:regopt2} rewrites as
        \begin{equation}
        \label{e:regopt4}
        J * \gamma -\gamma = \dot{W}(\gamma), \ \ \text{ in the weak sense on $\mathbb{R}$,}
        \end{equation}
        where the symbol $*$ denotes the convolution operator. Since both members in equality \eqref{e:regopt4} are $L^1_{loc}$-functions, we deduce that 
         \begin{equation}
        \label{e:regopt4.1}
        J * \gamma = \dot{W}(\gamma) +\gamma, \ \ \text{ a.e. on $\mathbb{R}$.}
        \end{equation}
        By further using our invertibility assumption \eqref{e:regopt6.1234.1}, then \eqref{e:regopt4.1} becomes
         \begin{equation}
        \label{e:regopt5}
        P(J * \gamma) = \gamma, \ \ \text{ a.e. on $\mathbb{R}$,}
        \end{equation}
       Eventually, since $P(J * \gamma)$ is continuous, being the composition of two continuous functions, we deduce from \eqref{e:regopt5} that $\gamma$ coincides with a continuous function. 
    \end{proof}

    \subsection{Regularity: The case of non-integrable kernels}
    \label{subs:cont1d-not-integrable}
    In this subsection we deduce some regularity properties for maps
     $\gamma$ solving \eqref{e:regopt2} under the additional assumption that $J\in \mathcal{L}_1(\eta,\lambda,\rho)$ (see Definition \ref{def:classJ}) for some $\eta,\lambda,\rho\in (0,1)$.

Given a measurable function $\gamma \colon \mathbb{R} \to \mathbb{R}$, for $\eta>0$ we set 
$$\|\gamma\|_{L^1_\eta(\mathbb{R})}:= \int_{\mathbb{R}} \frac{|\gamma(t)|}{1+|t|^{1+\eta}} \, dt.$$
 We have the following H\"older regularity result. The proof mainly follows \cite[Proposition 3.9]{cozzi}.

    \begin{proposition}
        \label{p:holderreg1}
        Let $\eta,\lambda,\rho\in (0,1)$ and let $J \in\mathcal{L}_1(\eta,\lambda,\rho)$. Let $\gamma \in L^\infty(\R) \cap BV(\R)$ and $f \in L^\infty(\mathbb{R})$ satisfy
        \begin{equation}
        \label{e:regopt8}
        L_J\gamma=f \ \ \text{ in the weak sense on $\mathbb{R}$}
        \end{equation}
        (recall \eqref{e:regopt3}).
        Then, for every $\alpha \in (0,\eta)$ and $a>0$ we have 
        \begin{equation}
           \label{e:regopt7}
            \lfloor \gamma\rfloor_{C^{0,\alpha}((-a,a))} \leq C(\|f\|_{L^\infty((-2a,2a))} + \|\gamma\|_{L^\infty(\mathbb{R})} ), 
        \end{equation}
        where $C:=C(\eta,\alpha,\lambda,a)$.
    \end{proposition}
\begin{proof}
We first observe that it is enough to prove the result on $(-1,1)$, the case $(-a,a)$ follows in fact by a scaling argument. 

\textbf{Step 1}. We begin by proving the proposition under the following additional assumption on $J$
\begin{equation}
        \label{e:regopt6.1}
        \frac{\lambda}{|h|^{1+\eta}}  \leq J(h) \leq  \frac{1}{\lambda|h|^{1+\eta}}, \ \ \text{ for every $h \in \mathbb{R}$.}
    \end{equation}
         In view of \cite[Proposition 3.2]{cozzi} (see also \cite{dong}), under condition \eqref{e:regopt6.1}, if we further assume $\gamma \in C^2_{loc}((-2,2))$, there holds
 \begin{equation}
           \label{e:regopt7.1}
            \lfloor \gamma\rfloor_{C^{0,\alpha}((-a,a))} \leq C(\|f\|_{L^\infty((-2a,2a))} + \|\gamma\|_{L^1_{\eta}(\mathbb{R})} ), 
        \end{equation}

         Let now $\gamma\in L^{\infty}(\R) \cap BV(\R)$, and denote by $\varphi_\epsilon \colon \mathbb{R} \to [0,\infty)$ a family of even mollifiers, namely, for $\varphi \in C^\infty_c((-1,1))$ satisfying $\varphi(-t)=\varphi(t)$ for all $t\in\R$ and $\int_{(-1,1)} \varphi(t)dt =1$, and for every $\ep$, we set $\varphi_\epsilon(t):=\epsilon^{-1}\varphi(t\epsilon)$ for all $t\in\R$.
         By letting $\gamma_\epsilon := \gamma *\varphi_\epsilon$ and $f_\epsilon := f * \varphi_\epsilon$, a change of variables argument shows that 
         \[
          L_J\gamma_\epsilon=f_\epsilon, \ \ \text{ pointwise on $\mathbb{R}$.}
         \]
         Hence, being $\gamma_\epsilon \in C^\infty(\mathbb{R}) \cap L^\infty(\mathbb{R})$, by \eqref{e:regopt7.1} we have for every $\epsilon>0$
         \begin{equation}
            \label{e:regopt7.2}
            \lfloor \gamma_\epsilon\rfloor_{C^{0,\alpha}((-1,1))} \leq C(\|f_\epsilon\|_{L^\infty((-2,2))} + \|\gamma_\epsilon\|_{L^1_\eta(\mathbb{R})} ),
        \end{equation}
         for a constant $C$ independent from $\epsilon$. Since $\|f_\epsilon\|_{L^\infty(\mathbb{R})} \leq \|f\|_{L^\infty(\mathbb{R})}$ as well as $\|\gamma_\epsilon\|_{L^1_\eta(\mathbb{R})} \to  \|\gamma\|_{L^1_\eta(\mathbb{R})}$ as $\epsilon \to 0^+$, we can apply Ascoli-Arzel\`a Theorem to infer that, up to subsequences, $\gamma_\epsilon \to \tilde{\gamma}$ uniformly on $[-1,1]$ for some $\tilde{\gamma} \colon [-1,1] \to \mathbb{R}$. Clearly, the uniform convergence implies that \eqref{e:regopt7.2} is maintained at the limit, namely,
         \[
            \lfloor \tilde{\gamma}\rfloor_{C^{0,\alpha}((-1,1))} \leq C(\|f\|_{L^\infty((-2,2))} + \|\gamma\|_{L^1_\eta(\mathbb{R})} ).
        \]
        
        Eventually, since from the standard property of the mollification we know also that $\gamma_\epsilon \to \gamma$ in $L^1_{loc}(\mathbb{R})$, we deduce $\gamma=\tilde{\gamma}$ a.e. on $\mathbb{R}$, and hence \eqref{e:regopt7.1} is proved under the additional assumption \eqref{e:regopt6.1}.

        \textbf{Step 2}. Now we show how to pass from \eqref{e:regopt6.1} to the following more general assumption
        \begin{equation}
        \label{e:regopt6.1.1}    
        \frac{\lambda \mathbbm{1}_{(-\rho,\rho)}(h)}{|h|^{1+\eta}}  \leq J(h) \leq  \frac{1}{\lambda|h|^{1+\eta}}, \ \ \text{ for every $h \in \mathbb{R}$.}
        \end{equation}
        To this purpose, we define $J' \colon \mathbb{R} \to [0,\infty)$ to be any $C^\infty$-regular function  satisfying
        \[
        J'(t):=
        \begin{cases}
            0, &\text{ if } t \in (-\rho/2,\rho/2) \\
            \frac{\lambda}{|t|^{1+\eta}}, &\text{ if } t \in \mathbb{R} \setminus (-\rho,\rho).
        \end{cases}
        \]
        We consider the operator $L'$ defined exactly as in \eqref{e:regopt3} with $J$ replaced by $J'$. Notice that $L'$ is well defined for every bounded function and, in particular, for every $\gamma' \in L^\infty(\mathbb{R})$ we have
        \[
        L' \gamma'= \gamma' * J' - \|J'\|_{L^1(\mathbb{R})} \gamma' \ \ \text{ and } \ \ \|L' \gamma'\|_{L^\infty(\mathbb{R})} \leq c \|\gamma'\|_{L^\infty(\mathbb{R})}. 
        \]
        Notice that, since $\tilde{J}:= J+J'$ satisfies \eqref{e:regopt6.1}, then by letting $\tilde{L} := L_J + L'$, by Step 1 we have that \eqref{e:regopt7.1} applies to any weak solution $\tilde\gamma \in L^\infty(\R) \cap BV(\R)$ of $\tilde{L} \tilde\gamma = \tilde{f}$ with $\tilde{f} \in L^\infty(\mathbb{R})$. Finally, by noticing that our original solution $\gamma$ satisfies $\tilde{L} \gamma = \tilde{f}$ in the weak sense on $\mathbb{R}$, where $\tilde{f}=f + L'\gamma \in L^\infty(\mathbb{R})$, we infer exactly \eqref{e:regopt7} under the additional assumption \eqref{e:regopt6.1.1}.

        \textbf{Step 3}. Assume now that $J$ satisfies \ref{K3} and
        \begin{equation}
        \label{e:regopt6.1.1.1}
         \frac{\lambda}{|h|^{1+\eta}}  \leq J(h) \leq  \frac{1}{\lambda|h|^{1+\eta}}, \ \ \text{ for every $h \in (-\rho,\rho)$.}
        \end{equation}
         We observe that \ref{K3} together with \eqref{e:regopt6.1.1.1} imply that, by letting $\hat{J}(h) := J(h) \wedge [1/(\lambda|h|^{1+\eta})]$, then $J-\hat{J}=0$ on $(-\rho,\rho) \setminus \{0\}$ and that $J-\hat{J} \in L^1(\mathbb{R})$. Therefore, by considering the operators $\hat{L}$ and $L^-$ relative to the kernels $\hat{J}$ and $J - \hat{J}$, respectively, since $\hat{L} \gamma = f - L^- \gamma$ and $L^- \gamma \in L^{\infty}(\mathbb{R})$ we obtain as before the validity of \eqref{e:regopt7}. 

         \textbf{Step 4}. Let now $J\in\mathcal{L}_1(\eta,\lambda,\rho)$. We define again $\hat{J}(h) := J(h) \wedge [1/(\lambda|h|^{1+\sigma})]$, and we notice that $\hat{J}$ satisfies \eqref{e:regopt6.1.1.1} and \ref{K3}. Moreover $0\leq J-\hat{J} \leq J$ in $\mathbb{R}$ while for $h \in (-\rho,\rho) \setminus \{0\}$ we have $J(h)-\hat{J}(h)= 0$ if $J(h) \leq 1/(\lambda|h|^{1+\eta})$ and $J(h)-\hat{J}(h) \leq \tilde{J}(h) $ if $J(h) > 1/(\lambda|h|^{1+\eta})$. Hence $J-\hat{J} \in L^1(\mathbb{R})$ and we obtain the desired result by arguing again as before. 
       
\end{proof}
As a consequence of Proposition \ref{p:holderreg1}, we infer the following regularity result for increasing minimizers of \eqref{e:onedim100}.
\begin{corollary}
\label{c:holdrreg1}
     Let $\eta,\lambda,\rho\in (0,1)$. Let $J \in \mathcal{L}_1(\eta,\lambda,\rho)$ and let $W \colon \mathbb{R} \to [0,\infty)$ be a continuously differentiable function satisfying \eqref{e:assu1}--\eqref{e:assu2} with $a=-1$ and $b=1$. Then, any increasing minimizer $\gamma$ of \eqref{e:onedim100} in $X_{-1}^1$ satisfies $\gamma \in C^{0,\alpha}_{loc}(\mathbb{R})$ for every $\alpha \in (0,\eta)$.
\end{corollary}
\begin{proof}
    If $\gamma$ is an increasing minimizer of \eqref{e:onedim100} in the class $X_{-1}^1$, then $\gamma$ satisfies \eqref{e:regopt2} with $\dot{W}(\gamma) \in L^\infty(\mathbb{R})$. Therefore, the result follows directly from Proposition \ref{p:holderreg1}. 
\end{proof}

\begin{remark}
    \label{r:holderreg1}
    {Under the assumptions of Propositions \ref{p:regoptintk} and \ref{p:holderreg1}, since every increasing minimizer $\gamma$ of \eqref{e:onedim100} in the class $X_{-1}^1$ satisfies $L_J \gamma \in L^1(\mathbb{R})$, equation \eqref{e:regopt2} is satisfied pointwise a.e. in $\mathbb{R}$. Furthermore, since $\gamma$ is continuous, the right hand-side of equation \eqref{e:regopt2} is continuous as well, meaning that the function $L_J \gamma$ admits a continuous representative in its Lebesgue equivalence class.}
\end{remark}

\subsection{Continuous dependence of optimal profiles}
\label{subs:cont-dep1d}
    In this subsection we consider a family of double-well potentials $(W_{\rho})_{\rho}$, and prove a continuous dependence result for corresponding increasing minimizers of the associated one-dimensional non-local energies
    \begin{equation}
     \label{e:convene5825}
     F_\rho(\gamma):= \frac{1}{4} \int_{\mathbb{R} \times \mathbb{R}} J(t'-t) (\gamma(t')-\gamma(t))^2 \, dt'dt + \int_{\mathbb{R}} W_\rho(\gamma(t)) \, dt, \ \ \gamma \in L^1(\mathbb{R}),
     \end{equation}
    under the existence of maps $(\varphi_{\rho})_{\rho}$ connecting the wells of $(W_{\rho})_{\rho}$ potentials with those of $W_0$.

\begin{proposition}
\label{p:contdep}
    Let $J \colon \mathbb{R} \to [0,\infty)$ be an even function satisfying \ref{K2}. Let $(a_\rho)$ and $(b_\rho)$ denote two sequences of real numbers converging to $a_0$ and $b_0$, respectively, and let $W_\rho \colon \mathbb{R} \to [0,\infty)$ be a family of double-well potentials, satisfying
    \begin{enumerate}
\item $W_{\rho}(t) >0, \ \ \text{ for every $t$ in }  (a_\rho,b_\rho)$
\item $W_{\rho}(a_\rho)=W_{\rho}(b_\rho)=0$
\item $W_\rho \to W_0 \ \ \text{ strongly in } L^\infty_{loc}(\R)$.
\end{enumerate}
Furthermore assume that there exist $0<\delta <(b_0-a_0)/2$, as well as constants $0<c_1 \leq c_2 <\infty$, and maps $\varphi_{\rho} \colon \mathbb{R} \to \mathbb{R}$ such that 
\begin{enumerate}
\setcounter{enumi}{3}
\item   $\varphi_{\rho} \colon [a_{0},b_0] \to [a_{\rho},b_\rho]$, $\varphi_\rho(a_0)=a_\rho$, and $\varphi_\rho(b_0)=b_\rho$ \EEE
\item $\sup_{\rho>0} \lfloor\varphi_\rho\rfloor_{C^{0,1/2}([a_{0},b_0])} < \infty$
\item $\varphi_{\rho} \to \emph{id} \ \ \text{ strongly in }L^\infty_{loc}(\mathbb{R})$
\item $c_1 W_{\rho}(\varphi_{\rho}(t))\leq  W_0(t) \leq c_2 W_{\rho}(\varphi_{\rho}(t)) \ \ \text{ for every }t \in  [a_{0},a_{0} +\delta) \cup (b_{0}-\delta,b_0]$.
\end{enumerate}
     For every $\rho>0$, let $\gamma_\rho$ be an increasing minimizer of $F_\rho$ in $X_{a_\rho}^{b_\rho}$.
     Then,
     \begin{align}
         \label{e:convenergies}
         F_\rho(\gamma_\rho) \to \inf_{\gamma \in X_{a_0}^{b_0} }F_0(\gamma), \ \ \text{ as $\rho \to 0^+$.}
     \end{align} 
     If, in addition if the family $(\gamma_\rho)$ also satisfies
     \begin{equation}
     \label{e:stronglycen}
         \sup_{\rho} |k_\rho| < \infty, \ \ \text{ where } \ \ k_\rho:= \sup\, \{k \in \mathbb{R} \ | \ \gamma_\rho(t) \leq (a_\rho +b_\rho)/2 \text{ if }t \leq k \},
     \end{equation}
     then every $L^1_{loc}$-limit of $(\gamma_\rho)_{\rho}$ is an increasing minimizer of $F_0$ in $X_{a_0}^{b_0}$.
     \end{proposition}
\begin{proof}
For convenience of the reader, we subdivide the proof into three steps.

    \textbf{Step 1}. We first identify the map $\gamma_0$. Since every map $\gamma_{\rho} \colon \mathbb{R} \to [a_\rho,b_\rho]$ is increasing, we have $(\gamma_\rho)_{\rho} \subset BV_{loc}(\mathbb{R})$ and $|D\gamma_\rho|(\mathbb{R}) \leq |b_\rho-a_\rho|\leq C$. Hence, by classical compactness results in $BV_{loc}(\mathbb{R})$ we find an  increasing function $\gamma_0 \colon \mathbb{R} \to [a_0,b_0]$, and a subsequence $(\gamma_{\rho_j})$ such that 
    \begin{align}
    \label{e:contdep11}
        &\gamma_{\rho_j} \to \gamma_0, \ \ \text{ in } L^1_{loc}(\mathbb{R}) \ \ \text{ and } \ \ \gamma_{\rho_j}(t) \to \gamma_0(t), \ \ \text{ for a.e. } t \in \mathbb{R} \\
        \label{e:contdep12}
        & \ \ \ \ \ \ \ \ \ \ \ \ \ \ \ D\gamma_{\rho_j} \rightharpoonup D\gamma_0, \ \ \text{ in duality with $C^0_c(\mathbb{R})$}.
    \end{align}
     Being the functional $F_\rho$ translation-invariant, 
     the maps $\tilde{\gamma}_{\rho}(t):=\gamma_{\rho}(t-k_{\rho})$, where $k_{\rho}$ is defined as in \eqref{e:stronglycen}, satisfy
          $\tilde\gamma_\rho \leq (a_\rho+b_\rho)/2$ and $\tilde\gamma_\rho \geq (a_\rho+b_\rho)/2$ almost everywhere on $(-\infty,0)$ and $(0,\infty)$, respectively. In particular, arguing as in the proof of \eqref{e:contdep11},
          we identify an increasing map $\tilde\gamma_0:\R\to[a_0,b_0]$ such that, up to the extraction of a further not-relabelled subsequence, there holds
          \begin{align}
    \label{e:contdep11-new}
        &\tilde\gamma_{\rho_j} \to \tilde\gamma_0, \ \ \text{ in } L^1_{loc}(\mathbb{R}) \ \ \text{ and } \ \ \tilde\gamma_{\rho_j}(t) \to \tilde\gamma_0(t), \ \ \text{ for a.e. } t \in \mathbb{R}\\
    \label{e:bcinfty}
        &\tilde\gamma_0 \leq \frac{a_0+b_0}{2} \text{ a.e. on $(-\infty,0)$} \ \ \text{ and } \ \   \tilde\gamma_0 \geq \frac{a_0+b_0}{2} \text{ a.e. on $(0,\infty)$ }.
    \end{align}

    Note that, 
    \begin{equation}
    \label{eq:centered}
    \text{if }\eqref{e:stronglycen}\text{ holds, then }\tilde{\gamma}_0=\gamma_0(\cdot-k_0), \text{ where }k_{\rho_j}\to k_0,
    \end{equation}
    whereas no connection between $\tilde\gamma_0$ and $\gamma_0$ is established if \eqref{e:stronglycen} is not fulfilled.
    
    In order to show that $\tilde\gamma_0 \in X_{a_0}^{b_0}$ we notice first that the hypothesis $a_\rho \to a_0$ and $b_\rho \to b_0$ in combination with \eqref{e:contdep11} gives $a_0 \leq \tilde\gamma_0(t) \leq b_0$ for almost every $t\in \R$. 
    
    Defining
$\overline{W}_\rho \colon \mathbb{R}\to [0,\infty)$ as $\overline{W}_\rho(t):=W_\rho \mathbbm{1}_{(a_{\rho},b_{\rho})}(t)$ for every $t\in \R$, conditions (2)--(3) imply
    \begin{equation}
        \label{e:contdep1.1}
        \overline{W}_\rho(t) \to \overline{W}_0(t), \ \ \text{ in } L^\infty(\mathbb{R}).
    \end{equation}
    Now \eqref{e:contdep1.1} together with the minimality of the maps $\gamma_{\rho}$ (and hence of $\tilde\gamma_{\rho}$) and with Fatou Lemma implies (up to the extraction of a further, not-relabelled, subsequence)
    \[
   \infty > \liminf_{j} \int_{\mathbb{R}} W_{\rho_j}(\tilde\gamma_{\rho_j}) \, dt =  \liminf_{j} \int_{\mathbb{R}} \overline{W}_{\rho_j}(\tilde\gamma_{\rho_j}) \, dt\geq   \int_{\mathbb{R}} \overline{W}_{0}(\tilde\gamma_{0}) \, dt = \int_{\mathbb{R}} W_{0}(\tilde\gamma_{0}) \, dt.
    \]
    In particular, being $W_0 >0$ on $(a_0,b_0)$, we can make use of condition \eqref{e:bcinfty} to infer $\lim_{t \to -\infty} \tilde\gamma_0(t) = a_0$ and $\lim_{t \to \infty} \tilde\gamma_0(t) = b_0$, hence $\tilde\gamma_0 \in X_{a_0}^{b_0}$.

\textbf{Step 2}. We show that $F_{\rho_j}(\gamma_{\rho_j}) \to F_0(\tilde\gamma_0)$. First, in view of Fatou lemma, \eqref{e:contdep11} and by the translation invariance of the functionals $F_{\rho_j}$, owing to \eqref{e:contdep1.1}, we infer  
    \begin{align*}
        \liminf_{j} F_{\rho_j}(\gamma_{\rho_j})=\liminf_{j} F_{\rho_j}(\tilde\gamma_{\rho_j})\geq F_0(\tilde\gamma_0)
    \end{align*}

    In order to achieve the opposite $\limsup$ inequality, it is enough to construct a sequence of increasing functions $(u_j)$ such that $u_j \in X_{a_{\rho_j}}^{b_{\rho_j}}$ and  
    \begin{equation}
    \label{eq:limsup-contdep}
    \limsup_{j} F_{\rho_j}(u_j) \leq F_0(\tilde\gamma_0).
    \end{equation}
    Indeed, once \eqref{eq:limsup-contdep} is established, from the minimality of $\gamma_{\rho_j}$ we obtain
    \[
    \limsup_{j} F_{\rho_j}(\gamma_{\rho_j}) \leq \limsup_{j} F_{\rho_j}(u_j) \leq F_0(\tilde\gamma_0),
    \]
    thus yielding the thesis. To prove \eqref{eq:limsup-contdep}, we set 
    \[
    u_j(t) := \varphi_{\rho_j}(\tilde\gamma_0(t)), \ \ \text{ for } t \in \mathbb{R},
    \]
    where $(\varphi_\rho)$ is the family of maps provided by (4)--(7). Then, from (1), (2), and (7) and recalling that  $\tilde\gamma_0 \in X_{a_0}^{b_0}$, we deduce $u_j \in X_{a_{\rho_j}}^{b_{\rho_j}}$. Additionally,  
    \begin{align*}
       & J(h) |\varphi_{\rho_j}(\tilde\gamma_{0}(t+h)) - \varphi_{\rho_j}(\tilde\gamma_{0}(t))|^2  \leq \lfloor\varphi_{\rho_j}\rfloor_{C^{0,1/2}([a_0,b_0])}^2 J(h) |h| \frac{|\tilde\gamma_0(t+h) - \tilde\gamma_0(t)|}{|h|},
       \end{align*}
       for every $(t,h) \in \mathbb{R} \times \mathbb{R}$. In particular, since $|\dot{\tilde\gamma}_0|(\mathbb{R})= |b_0-a_0|$, we have
       \[
       \int_{\mathbb{R} \times \mathbb{R}}J(h) |h| \frac{|\tilde\gamma_0(t+h) - \tilde\gamma_0(t)|}{|h|} \, dtdh \leq |b_0-a_0| \int_{\mathbb{R}} J(h)|h| \, dh < \infty.
       \]
       By Lebesgue Dominated Convergence Theorem, from (4)--(5) we deduce that 
       \begin{align*}
           &\lim_{j} \int_{\mathbb{R} \times \mathbb{R}} J(h)  |\varphi_{\rho_j}(\tilde\gamma_{0}(t+h)) - \varphi_{\rho_j}(\tilde\gamma_{0}(t))|^2 \, dtdh= \int_{\mathbb{R} \times \mathbb{R}} J(h) |\tilde\gamma_{0}(t+h) - \tilde\gamma_{0}(t)|^2 \, dtdh.
       \end{align*}
In view of the change of variable formula \eqref{e:changevar} with $G=W \circ \varphi_{\rho_j}$, we write
\begin{align*}
    &\lim_{j}\int_\mathbb{R} W_{\rho_j}(u_j(t)) \, dt = \lim_{j}\int_{a_{0}}^{b_{0}} W_{\rho_j}(\varphi_{\rho_j}(t)) \, d\dot{\tilde\gamma}^{-1}_0(t)\\ 
    &= \lim_j \int_{[a_{0},a_{0}+\delta) \cup (b_{0}-\delta,b_{0}]} W_{\rho_j}(\varphi_{\rho_j}(t)) \, d\dot{\tilde\gamma}^{-1}_0(t)+ \lim_{j} \int_{a_{0}+\delta}^{b_{0}-\delta} W_{\rho_j}(\varphi_{\rho_j}(t))\, d\dot{\tilde\gamma}^{-1}_0(t),
    \end{align*}
    whenever both limits on the right-hand side exist. To study the first term, we notice that from (7) there holds
    \[
    W_{\rho_j}(\varphi_{\rho_j}(t)) \leq \frac{1}{c_1} W_0(t), \ \ \text{for every }t \in [a_0,a_0+\delta) \cup (b_0-\delta,b_0].
    \]
    Being $W_0 \in L^1([a_0,b_0];\dot{\gamma}_0^{-1})$, we are thus allowed to apply the Lebesgue Dominated Convergence Theorem and to infer from (3) and (6)
    \begin{equation}
        \lim_j \int_{[a_{0},a_{0}+\delta) \cup (b_{0}-\delta,b_{0}]} W_{\rho_j}(\varphi_{\rho_j}(t)) \, d\dot{\tilde\gamma}^{-1}_0(t)= \int_{[a_{0},a_{0}+\delta) \cup (b_{0}-\delta,b_{0}]} W_{0}(t) \, d\dot{\tilde\gamma}^{-1}_0(t).
    \end{equation}
    Concerning the second term, by \eqref{e:contdep1.1} and since $\dot{\tilde\gamma}^{-1}_0$ is a bounded Radon measure on $[a_{0}+\delta,b_{0}-\delta]$, we deduce
    \begin{equation}
        \lim_j \int_{[a_{0}+\delta,b_{0}-\delta]} W_{\rho_j}(\varphi_{\rho_j}(t)) \, d\dot{\tilde\gamma}^{-1}_0(t)= \int_{[a_{0}+\delta,b_{0}-\delta]} W_{0}(t) \, d\dot{\tilde\gamma}^{-1}_0(t).
    \end{equation}
    This proves \eqref{eq:limsup-contdep} and yields that $F_{\rho_j}(\gamma_{\rho_j}) \to F_0(\tilde\gamma_0)$ as $j \to \infty$.
    
    \textbf{Step 3}. We now show that $\tilde\gamma_0$ is a minimum of $F_0$. 
    Let $\gamma \in X_{a_0}^{b_0}$. The same construction as in Step 2 shows that $\limsup_{j} F_{\rho_j}(u_j) \leq F_0(\gamma)$ where $u_j:= \varphi_{\rho_j}(\gamma) \in X_{a_{\rho_j}}^{b_{\rho_j}}$. Therefore, we have
\[
F_0(\gamma) \geq \limsup_{j} F_{\rho_j}(u_j) \geq \limsup_{j} F_{\rho_j}(\gamma_{\rho_j}) =F_0(\tilde\gamma_0),
\]
for every $\gamma \in X_{a_0}^{b_0}$.
    
    Summarizing, we have obtained \eqref{e:convenergies} for a particular subsequence $\rho_j$. In order to obtain the full convergence \eqref{e:convenergies} we notice that the argument above leads to the validity of the following implication: if $\rho_j$ is any sequence converging to $0$, then we find a subsequence $\rho_{j_\ell}$ for which \eqref{e:convenergies} holds true. Therefore \eqref{e:convenergies} follows from a well known characterization of sequential convergence in metric spaces. 
    
    Eventually, the last statement of the proposition immediately follows by observing that if \eqref{e:stronglycen} is satisfied, then the identification in \eqref{eq:centered} holds true. Thus, $\gamma_0\in X_{a_0}^{b_0}$ and its optimality follows from that of 
    $\tilde\gamma_0$ and from the translation invariance of $F_0$.
    \end{proof}

\subsection{H\"older selection of optimal profiles} 
\label{subs:sel-1d}

In this subsection, we consider the non-autonomous non-local functionals
\begin{equation}
         \label{e:lipdep0}
       F_x(\gamma)=\frac{1}{4}\int_{\mathbb{R}\times \mathbb{R}} J(t'-t)|\gamma(t')-\gamma(t)|^2 \, dt'dt + \int_{\mathbb{R}} W(x,\gamma(t)) \, dt,
         \end{equation}
where $W$ satisfies \ref{H1}--\ref{H5} and $J \colon \mathbb{R} \to (0,\infty)$ is an even Kernel satisfying \ref{K2}. In the main result of this section, cf. Theorem \ref{p:lipselx} below, we show that, under either a further assumption on $W$ or further integrability condition on $J$,
 there exists a family of increasing minimizers $(\gamma_x)_x$ of the functionals $F_x$ depending on $x$ in a H\"older-continuous way.\\

In order to establish Theorem \ref{p:lipselx} we need a series of preliminary results.\\ 

We begin by recalling the (sharp) finiteness principles for Lipschitz selection obtained in \cite[Theorem 1.2]{fef}. In order to properly state this result we introduce the following notation: for a given positive integer $m$ and a given Banach space $Y$ we say that $C\subset Y$ belongs to $\mathcal{K}_m(Y)$ if $C$ is a non-empty compact and convex set such that $C \subset V$ for some $m$-dimensional subspace $V \subset Y$.

\begin{theorem}[Theorem 1.2 in \cite{fef}]
    \label{t:finitelipsel}
Let $m$ be a positive integer, let $(X,d)$ be a metric space, and let $F \colon X \to \mathcal{K}_m(Y)$ for some Banach space $Y$. Suppose that there exists $\lambda>0$ such that for every $X' \subset X$ consisting of at most $2^{m+1}$ points, the restriction $F|_{X'}$ of $F$ to $X'$ admits a Lipschitz selection $f_{X'}$ with $\lfloor f_{X'} \rfloor_{C^{0,1}(X';Y)} \leq \lambda $.

Then  $F$ admits a Lipschitz selection $f$ with $\lfloor f\rfloor_{C^{0,1}(X;Y)} \leq q \lambda$ where $q:= q(m) > 0$.
\end{theorem}

The next lemma shows that under \ref{H1}--\ref{H3}, for every compact set $K\in \R^m$ the family of double-well potentials $(W(x,\cdot)_{x\in K}$ satisfies assumptions (1)--(7) of Proposition \ref{p:contdep}.

\begin{lemma}
\label{lem:H3-trans}
Let $W$ satisfy \ref{H1}--\ref{H3}. Then, for every compact set $K \subset \mathbb{R}^m\EEE \times \mathbb{R}^m\EEE$ there exist (transition) maps $\varphi^K_\rho \colon \mathbb{R}^m\EEE \times \mathbb{R}^m\EEE \times \mathbb{R} \to \mathbb{R}$ for $0 < \rho \leq 1$, such that for every $(x,x') \in K$ the following properties holds true 
    \begin{enumerate} [label=(X.\arabic*)]
    \vspace{1mm}
        \item \label{X1} There holds
        \begin{align*}
        & \ \ \ \ \ \ \ \ \ \varphi^K_\rho(x,x',\cdot ) \colon [z_1(x),z_2(x)] \to [z_1(x+\rho x'),z_2(x + \rho x')], \\
        &\varphi^K_\rho(x,x',z_1(x))=z_1(x+\rho x') \ \ \text{ and } \ \ \varphi^K_\rho(x,x',z_2(x))=z_2(x+\rho x').
        \end{align*}
        \vspace{1mm}
        \item \label{X2}The following bounds are fulfilled:
        \begin{align*}
        & \lfloor\varphi^K_\rho(\cdot,\cdot,t)\rfloor_{C^{0,\alpha}(K)} < \infty, \ \ \text{ for }t\in \mathbb{R} \text{ and }\alpha \in (1/2,1]\text{ as in }\ref{H1},\EEE \\ 
       &\lfloor\varphi^K_\rho(x,x',\cdot)\rfloor_{ C^{0,1}\EEE([z_1(x),z_2(x)])} \leq 1 + \mathrm{O}_K(\rho),
        \end{align*}
         where $|\mathrm{O}_K(\rho)|\leq C\rho$ as $\rho\to 0^+$. 
        \vspace{1mm}
         \item \label{X3} $\varphi^K_\rho(x,0,t)= t$ for every $t \in  \mathbb{R}$.
         \vspace{1mm}
         \item \label{X4} $\varphi^K_\rho(x,x',t) \to t, \text{ in }L^{\infty}_{loc}(K \times \mathbb{R})$ as $\rho \to 0^+$.
         \vspace{1mm}
        \item \label{X5} For every $t \in [z_1(x),z_1(x)+2\delta_K) \cup (z_2(x)-2\delta_K,z_2(x)]$ we have
        \[
c_1 W_{x,\rho}(x,\varphi^K_\rho(x,x',t))\leq W(x,t) \leq c_2 W_{x,\rho}(x,\varphi^K_\rho(x,x',t))
\]
for some $0 < c_1 \leq c_2 <\infty$, where $W_{x,\rho}(x',t)=W(x + \rho x',t)$  and where $\delta_K$ is as in \ref{H2}. \EEE
    \end{enumerate}
\end{lemma}
\begin{proof}
  
 We set 
 \begin{equation}
 \label{e:h3def}
 \varphi_\rho^K(x,x',t):=\left(\frac{t-z_1(x)}{z_2(x)-z_1(x)}\right)z_2(x+\rho x')+\left(1-\frac{t-z_1(x)}{z_2(x)-z_1(x)}\right)z_1(x+\rho x')
 \end{equation}
 for every $(x,x',t)\in \mathbb{R}^m\EEE \times \mathbb{R}^m\EEE \times \mathbb{R}$.
 By definition, the maps $\varphi_\rho^K$ satisfy \ref{X1}, \ref{X3} und \ref{X4}. In view of the regularity of the functions $z_i$, we further directly infer \ref{X2}.
 \EEE

 We notice that for  $t \in [z_1(x),z_2(x)]$ there \EEE holds 
 \begin{align*}
      &W_{x,\rho}(x',\varphi_\rho(x,x',t)) \leq \frac{1}{\delta} f(\min \{|\varphi_\rho(x,x',t)-z_1(x+\rho x')|, |\varphi_\rho(x,x',t)-z_2(x+\rho x')| \})\\
      &\quad= \frac{1}{\delta} f(\min \{|t-z_1(x)|, |t-z_2(x)| \}) \leq \frac{1}{\delta^2} W(x,t) \leq \frac{1}{\delta^3} f(\min\{|t-z_1(x)|,|t-z_2(x)|\}) \\
      &\quad = \frac{1}{\delta^3} f(\min \{|\varphi_\rho(x,x',t)-z_1(x+\rho x')|, |\varphi_\rho(x,x',t)-z_2(x+\rho x')| \})\leq \frac{1}{\delta^4} W_{x,\rho}(x',\varphi_\rho(x,x',t)).
 \end{align*}
  Thus, \ref{X5} is satisfied by choosing  $c_1=\delta^2$ and $c_2=1/\delta^2$.
  %\todo[inline]{Eli: in questo caso a dire il vero la X.5 vale in forma pi\'u forte, non so se ci serve}\EEE
\end{proof}

We proceed by introducing the notion of centered families.

\begin{definition}[Centered family]
\label{d:cenfam}
    Let $K \subset  \mathbb{R}^m \EEE$ be compact and let $(\gamma_x)_{x \in K}\subset X_{-1}^1$. We say that the family $(\gamma_x)_{x \in K}$ is centered on $K$ if for every $\omega \in (0,2]$ we find $R':=R'(\omega,K)>0$ such that
    \begin{align}
    \label{e:cenfam1}
   &1-\frac{\omega}{2} < \gamma_x(t) \leq 1, \ \text{ for every $t \geq R'$ and every $x \in K$} \\
   \label{e:cenfam2}
   & -1 \leq  \gamma_x(t) <-1+\frac{\omega}{2}, \ \text{ for every $t \leq -R'$ and every $x \in K$}. 
\end{align}
\end{definition}
In other words, a family of maps in $X_{-1}^1$ is centered if for $t$ big enough it becomes uniformly close to $1$ and for $t$ small enough uniformly close to $-1$. The next lemma provides an example of a centered family of optimal profiles.

\begin{lemma}
\label{r:cenfam}
    Assume that $W$ satisfies \ref{H1}--\ref{H5} with $z_1 \equiv -1$ and $z_2 \equiv 1$, let $K\subset \R^m$ be compact, and let 
    $(\gamma_x)_{x\in K} \subset X_{-1}^1$ be a family of increasing optimal profiles, each minimizing in $X_{-1}^1$ the corresponding functional $F_x$ in \eqref{e:lipdep0}.
    Then, by letting $k_x := \sup \{k \in \mathbb{R} \ | \  \gamma_x(t) \leq 0 \text{ if } t \leq k \}$, the family $(\gamma_x(\cdot-k_x))_{x\in K}\subset X_{-1}^1$ is centered on $K$, and each translated map $\gamma_x(\cdot-k_x)$ is still a minimizer of $F_x$. 
    \end{lemma}
    \begin{proof}
    By contradiction,  assume that there exist a compact set $K \subset  \mathbb{R}^m$, and $\ep >0$ such that we find sequences $R'_n \nearrow +\infty$, $t_n \geq R'_n$, and $x_n \in K$ such that
    \begin{equation}
    \label{e:cenrem2}
    1-\frac{\omega}{2} \geq  \gamma_{x_n}(t_n-k_{x_n}), \ \ \text{ for every $k=1,2,\dotsc$}.
    \end{equation}
    Lemma \ref{lem:H3-trans} and Proposition \ref{p:contdep}, together with the choice of the translation constants $k_x$, allow us to infer that, up to extracting a non-relabelled subsequence for which $x_n \to x_0 \in K$ and  $\gamma_{x_n}(\cdot-k_{x_n}) \to \gamma_{x_0}$ almost everywhere in $\mathbb{R}$, with $\gamma_{x_0}$ an increasing minimizer in $X_{-1}^1$ of $F_{x_0}$. Nevertheless, the fact that $\gamma_{x_n}$ is increasing for every $n$ implies
    \[
    1-\frac{\omega}{2} \geq \limsup_n  \gamma_{x_n}(t_n-k_{x_n}) \geq \limsup_{n} \gamma_{x_n}(t-k_{x_n}) =  \gamma_{x_0}(t), \ \ \text{ for a.e. $t \leq \liminf_n t_n =+ \infty$ ,}
    \]
    which cannot be possible since $\gamma_{x_0} \in X_{-1}^1$. The same argument shows the validity of \eqref{e:cenfam2}.
    Finally, the fact that each translated map remains a minimizer of $F_x$ follows by a direct change of variable.
\end{proof}

We are now in a position to prove  Lipschitz dependence of continuous optimal profiles. Recall that the regularity of optimal profiles for even kernels $J$ satisfying either \eqref{e:regopt6.1234} or let $J \in\mathcal{L}_1(\eta,\lambda,\rho)$ for some $\eta,\lambda,\rho\in (0,1)$ was ensured from Proposition \ref{p:regoptintk} and Corollary \ref{c:holdrreg1}, respectively.

\begin{proposition}
\label{p:lipdep}
    Let $J \colon \mathbb{R} \to (0,\infty)$ be an even function satisfying \ref{K2}.  Let $W \colon  \mathbb{R}^m \EEE \times \mathbb{R} \to [0,\infty)$ satisfy \ref{H1}--\ref{H4} with $z_1\equiv -1$ and $z_2\equiv 1$.
         In addition, let $K \in  \mathbb{R}^m \EEE$ be compact and let $(\gamma_x)_{x\in K}$ be a centered family on $K$ such that $\gamma_x$ is an increasing and continuous minimizer in $X_{-1}^1$ of $F_x$ (recall \eqref{e:lipdep0}).
         Then, for every $\alpha \in (1/2,1]$ and for every pair $x,x' \in K$ there exists $k \in [-2R'(\delta/2,K),2R'(\delta/2,K)]$ (see Definition \ref{d:cenfam}) satisfying 
         \begin{equation}
             \label{e:lipdep1}
             \int_{\mathbb{R}} |\gamma_x(t) -\gamma_{x'}(t-k)| \sigma_J(t) \, dt \leq C\, \lfloor\partial_t W \rfloor_{C^{0,\alpha}(K \times [-2,2])} |x-x'|^{\alpha^2},
         \end{equation}
         where $C= C(\delta,K,\alpha,\|\sigma_J\|_{L^1}) > 0$, $\sigma_J$ is  a strictly positive integrable function defined as   
         \begin{equation}
         \label{e:lipdep1.428}
         \sigma_J(t):=\bigg(\inf_{t_0 \in [-2R'(\delta/2,K),2R'(\delta/2,K)]} J(t-t_0)\bigg) \wedge 1, \ \ t \in \mathbb{R},
         \end{equation}
         and, with a slight abuse of notation, we have written $\delta$ in place of $\delta_K$ (see \ref{H2}).
         
        \end{proposition}

\begin{proof}
Without loss of generality, we can assume that $\lfloor\partial_t W \rfloor_{C^{0,\alpha}(K \times [-2,2])} < \infty$. In addition, it is enough to prove \eqref{e:lipdep1} for every $x,x' \in K$ whose relative distance is below a fixed threshold, otherwise \eqref{e:lipdep1} simply follows from the boundedness of $\gamma_x$ and $\gamma_{x'}$. 

Since the family $(\gamma_x)_{x\in K}$ is centered on $K$, we can consider $R':= R'(\delta,K)>0$ for which \eqref{e:cenfam1} and \eqref{e:cenfam2} holds true with $\omega=\delta$.
Now, fix $x,x' \in K$ such that \begin{equation}\label{eq:beta}\beta:=\delta^{-1}\lfloor\partial_t W \rfloor_{C^{0,\alpha}(K \times [-2,2]) } |x-x'|^\alpha \leq \delta/8.
\end{equation}
By using that $\gamma_x,\gamma_{x'}$ are increasing continuous functions in $X_{-1}^1$,
 we find that for every $\ep >0$ sufficiently small there exists $k_\ep \in \mathbb{R}$ for which
\begin{equation}
\label{e:lipdep2}
\text{if $k \leq k_\epsilon$ then $\gamma_{x'}(t) < \gamma_x(t-k) + \beta + \epsilon $ for every $t \in \mathbb{R}$.}
 \end{equation}
 
 Now we consider $\overline{k}_\epsilon \in \mathbb{R}$ to be the supremum of all possible values of $k_\epsilon$ for which \eqref{e:lipdep2} holds true (notice that $\overline{k}_\epsilon$ is finite up to possibly consider a smaller set of couples $(x,x')$ in such a way that $\beta$ is below half of the distance of the two wells $\pm1$, namely, less than $1$). By using again that $\gamma_x,\gamma_{x'}$ are continuous, we see that $\gamma_x(t) \leq \gamma_{x'}(t- \overline{k}_\epsilon)+ \beta + \epsilon$ for every $t$ and that there exists $t_\epsilon \in \mathbb{R}$ such that
 \begin{equation}
     \label{e:lipdep3}
     \gamma_x(t_\epsilon) = \gamma_{x'}(t_\epsilon- \overline{k}_\epsilon) +\beta + \epsilon.
 \end{equation}
For convenience of the reader, we subdivide the remaining part of the proof into three steps.

\textbf{Step 1}. We claim that, by letting $R^\pm_x \in \mathbb{R}$ be defined in such a way that
\[
\begin{split}
(-\infty,R^-_x] &=\left \{t \in \mathbb{R} \ | \ -1 \leq \gamma_x(t) \leq -1+\frac{\delta}{2}  \right\} \\
[R^+_x,\infty) &= \left\{t \in \mathbb{R} \ | \ 1-\frac{\delta}{2} \leq \gamma_x(t) \leq 1  \right\} 
\end{split}
\]
then we have (notice that we have in general $-R' \leq R^-_x$ and $R^+_x \leq R'$)
\begin{equation}
\label{e:claim}
t_{\epsilon} \in [R^-_x,R^+_x], \ \ \text{for every $\epsilon >0$ sufficiently small}.
\end{equation}
Indeed, suppose by contradiction that there exists a subsequence $\epsilon_j \searrow 0$ such that for example $t_{\epsilon_j} \in (R^+_x,\infty)$ for every $j\in \mathbb{N}$. 

By letting $\gamma^\epsilon_{x'}(t):=  \gamma_{x'}(t- \overline{k}_\epsilon) +\beta + \epsilon$, we know from the minimality of $\gamma_x$ and $\gamma_{x'}$ that the following system is a.e. pointwise satisfied (see also Remark \ref{r:holderreg1})
\begin{equation}
\label{eq:system-prop}
\begin{cases}
    L_J\gamma^\epsilon_{x'}  = \partial_t W (x',\gamma^\epsilon_{x'} -\beta-\epsilon ), &\text{a.e. in $\mathbb{R}$} \\
    L_J\gamma_{x}  = \partial_t W (x,\gamma_{x}), &\text{a.e. in $\mathbb{R}$} \\
    \gamma^\epsilon_{x'}\geq \gamma_x, &\text{ in $\mathbb{R}$} \\
    \gamma^\epsilon_{x'}(t_\epsilon)= \gamma_x(t_\epsilon). &\text{ \ }
\end{cases}
\end{equation}
From the assumption that $t_{\epsilon_j} \in  (R^+_x,\infty)$ we can exploit property \ref{H4} and the inequality $\beta < \delta/2$, to infer $\partial_t W (x',\gamma_{x}(t_{\epsilon_j}) -\beta-\epsilon_j ) \leq \partial_t W(x',\gamma_x(t_{\epsilon_j}))-\delta \beta $ for every $j$ big enough. {Since the point $t_{\epsilon_j}$ might not satisfy the first and second equations in \eqref{eq:system-prop}, we consider $(s_\ell)$, a sequence such that  $s_\ell \to t_{\epsilon_j}$ and the first and second equations in \eqref{eq:system-prop} are satisfied in $t_{\epsilon_j}+s_\ell$ for every $\ell=1,2,\dotsc$. The third property in \eqref{eq:system-prop} allows us to make use of Fatou's lemma to infer}
\begin{align}
\nonumber
    0 &\leq \int_{\mathbb{R}} [\gamma^{\epsilon_j}_{x'}(t_{\epsilon_j} +t)-\gamma_x(t_{\epsilon_j} +t) + \gamma^{\epsilon_j}_{x'}(t_{\epsilon_j} -t)- \gamma_x(t_{\epsilon_j} -t)]  J(t) \, dt\\
     \nonumber
     \leq\liminf_{\ell \to \infty}&\int_{\mathbb{R}} [\gamma^{\epsilon_j}_{x'}(t_{\epsilon_j} +s_\ell +t)-\gamma_x(t_{\epsilon_j}+s_\ell +t) + \gamma^{\epsilon_j}_{x'}(t_{\epsilon_j} +s_\ell -t)- \gamma_x(t_{\epsilon_j} +s_\ell -t)]  J(t) \, dt \\ 
     &= \lim_{\ell \to \infty} 2L_J\gamma^{\epsilon_j}_{x'}(t_{\epsilon_j}+s_\ell) -2L_J\gamma_{x}(t_{\epsilon_j} +s_\ell) \\
     \label{e:whoknow}
    & =2\partial_t W (x',\gamma^{\epsilon_j}_{x'}(t_{\epsilon_j}) -\beta-\epsilon )-2\partial_t W (x,\gamma_{x}(t_{\epsilon_j})) \\
    \label{e:maxprinciple}
    & \leq2 \partial_t W (x',\gamma_{x}(t_{\epsilon_j}))-2\delta\beta -2\partial_t W (x,\gamma_{x}(t_{\epsilon_j})) < 0,
\end{align}
{where \eqref{e:whoknow} comes from the continuity of $\gamma_x$ and $\gamma_{x'}$ as functions of $t \in \mathbb{R}$}. In particular $\gamma^{\epsilon_j}_{x'}(t_{\epsilon_j} +t)-\gamma_x(t_{\epsilon_j} +t) =0$ for every $t \in \mathbb{R}$, but this contradicts the fact that $\lim_{t \to \infty} \gamma_x(t_{\epsilon_j} +t) = \lim_{t \to \infty} \gamma_{x'}(t_{\epsilon_j} - \overline{k}_{\epsilon_j} +t)=1$. Analogously one can prove that the assumption $t_{\epsilon_j} \in (-\infty,R^-_x)$ for every $j=1,2,\dotsc$ gives a contradiction. The claim is thus proved. The case in which $t_{\epsilon_j} \in (-\infty,R^-_x)$ for every $j=1,2,\dotsc$ follows analogously. This completes the proof of \eqref{e:claim}.

\textbf{Step 2}. We claim that $\overline{k}_\epsilon \in [-2R'',2R'']$ for every $\epsilon$ sufficiently small, where we have set $R'' := R'(\delta/2,K)$. Notice that, in view of  Definition \ref{d:cenfam}, we can assume without loss of generality that $R'' \geq R'$. By contradiction, let $\epsilon_j$ be an infinitesimal subsequence such that, for example, $\overline{k}_{\epsilon_j} > 2R''$ for every $j\in \mathbb{N}$. By applying \eqref{e:claim}, we find a further non-relabelled subsequence $(\epsilon_j)_j$ such that $t_{\epsilon_j} \to t_0 \in [R^-_x,R^+_x]$. Now thanks to \eqref{e:lipdep3} and our choice of $R''$ we obtain (remember that being $R^+_x \leq R' \leq R''$ then $R^+_x - 2R'' \leq -R''$)
\[
-1+\frac{\delta}{4} +\beta \geq \limsup_{j \to \infty} \gamma_{x'}(t_{\epsilon_j} - \overline{k}_{\epsilon_j}) +\beta +\epsilon_j = \lim_{\epsilon \to 0^+}  \gamma^{\epsilon_j}_{x'}(t_{\epsilon_j}) =\lim_{\epsilon \to 0^+}  \gamma_{x}(t_{\epsilon_j}) = \gamma_x(t_0) \geq -1+\frac{\delta}{2},
\]
which contradicts \eqref{eq:beta}. Assuming that $\overline{k}_{\epsilon_j} \leq -2R''$ for every $j\in \mathbb{N}$ leads to a similar contradiction. The claim is thus proved.

 \textbf{Step 3}. In view of Steps 1 and 2, we find a subsequence $(\ep_j)_j$ such that $t_{\epsilon_j} \to t_0 \in [-2R'',2R'']$ and $\overline{k}_{\epsilon_j} \to \overline{k}_0 \in [-2R'',2R'']$. Passing to the limit in \eqref{eq:system-prop}, we obtain
 \[
\begin{cases}
    L_J\gamma^0_{x'}  = \partial_t W (x',\gamma^0_{x'} -\beta), &\text{a.e. in $\mathbb{R}$} \\
    L_J\gamma_{x}  = \partial_t W (x,\gamma_{x}), &\text{a.e. in $\mathbb{R}$} \\
    \gamma^0_{x'}\geq \gamma_x, &\text{ in $\mathbb{R}$} \\
    \gamma^0_{x'}(t_0)= \gamma_x(t_0), &\text{ \ }
\end{cases}
\]
where we set $\gamma^0_{x'}(t):= \gamma_{x'}(t-\overline{k}_0) + \beta$. Therefore, by using the same strategy as for \eqref{eq:system-prop}, we may assume with no loss of generality that $t_0$ does satisfy both the first and second equation in the above system to write 
\begin{align}
\nonumber
  \int_{\mathbb{R}} [\gamma^{0}_{x'}(t_{0} +t)&-\gamma_x(t_{0} +t) + \gamma^{0}_{x'}(t_{0} -t)- \gamma_x(t_{0} -t)]  J(t) \, dt =2 L_J\gamma^{0}_{x'}(t_0)-2 L_J\gamma_{x}(t_0) \\
  \nonumber
  & =2\partial_t W (x',\gamma^{0}_{x'}(t_0) -\beta)-2\partial_t W (x,\gamma_{x}(t_0)) \\
  \label{e:lipdep4}
  & =2\partial_t W (x',\gamma_{x}(t_0) -\beta)-2\partial_t W (x,\gamma_{x}(t_0)) \\
  \nonumber
    & \leq2 \lfloor\partial_t W \rfloor_{C^{0,\alpha}(K \times [-2,2])}\beta^{\alpha} + 2 \lfloor\partial_t W \rfloor_{C^{0,\alpha}(K \times [-2,2])}|x-x'|^{\alpha}  \\
    \nonumber
    &=2 \lfloor\partial_t W \rfloor_{C^{0,\alpha}(K \times [-2,2])}(\delta^{-\alpha}\lfloor\partial_t W \rfloor_{C^{0,\alpha}(K \times [-2,2])}^{\alpha}+ |x-x'|^{\alpha(1-\alpha)})|x-x'|^{\alpha^2} \\
    &\leq C' \lfloor\partial_t W \rfloor_{C^{0,\alpha}(K \times [-2,2])} |x-x'|^{\alpha^2},\nonumber
\end{align}
where $C':=C'(\delta,K,\alpha)$. Since $ \gamma^0_{x'}\geq \gamma_x$ and $t_0 \in [-2R'',2R'']$ by \eqref{e:lipdep1.428} we write
\begin{align*}
& \int_{\mathbb{R}} [\gamma^{0}_{x'}(t)-\gamma_x(t)] \sigma_J(t) \, dt \\
&\leq \int_{\mathbb{R}} [\gamma^{0}_{x'}(t)-\gamma_x(t)] J(t-t_0) \, dt \\
&=  \int_{\mathbb{R}} [\gamma^{0}_{x'}(t+t_0)-\gamma_x(t+t_0)] J(t) \, dt \\
    &\leq \int_{\mathbb{R}} [\gamma^{0}_{x'}(t_{0} +t)-\gamma_x(t_{0} +t) + \gamma^{0}_{x'}(t_{0} -t)- \gamma_x(t_{0} -t)]  J(t) \, dt.
    \end{align*}
Thus, by \eqref{e:lipdep4} and \eqref{eq:beta} we infer
\begin{align*}
    & \int_{\mathbb{R}} |\gamma_{x'}(t-\overline{k}_0)-\gamma_x(t)| \sigma_J(t) \, dt \\
    &\leq \int_{\mathbb{R}} |\gamma_{x'}(t-\overline{k}_0) +\beta-\gamma_x(t)| \sigma_J(t) \, dt + \beta \|\sigma_J\|_{L^1(\mathbb{R})} \leq C\lfloor\partial_t W \rfloor_{C^{0,\alpha}} |x-x'|^{\alpha^2},
\end{align*}
where $C:= C(\delta,K,\alpha,\|\sigma_J\|_{L^1})$. We eventually set $k:= \overline{k}_0$ and the proof is concluded.

\end{proof}

\begin{remark}
    \label{r:maxprinciple} 
    Under the hypotheses of Proposition \ref{p:lipdep}, by arguing as in \eqref{e:maxprinciple}, it follows that any increasing solution $\gamma \in X_{-1}^1$ of $L_J \gamma = \partial W(x,\gamma)$ on $\mathbb{R}$ satisfies $-1 < \gamma(t) < 1$ for every $t \in \mathbb{R}$. Indeed, suppose for example that $\gamma(t_0)=1$ for some $t_0 \in \mathbb{R}$. Then, since $\gamma$ is an increasing function in $X_{-1}^1$, it holds $\gamma(t) \leq 1$ for every $t \in \mathbb{R}$. We can thus write
    \begin{align*}
\nonumber
    0 \leq \int_{\mathbb{R}} [1&-\gamma(t_0 +t) + 1- \gamma(t_0 -t)]  J(t) \, dt = 2L_J1(t_0) -2L_J\gamma(t_0) \\
    \nonumber
    & =2\partial_t W (x,1)-2\partial_t W (x,\gamma(t_0)) =0.
\end{align*}
Therefore we deduce $\gamma \equiv 1$, which gives a contradiction since we assumed $\gamma \in X_{-1}^1$. Analogously, assuming $\gamma(t_0)=-1$ for some $t_0 \in \mathbb{R}$ implies $\gamma \equiv -1$, which gives the same contradiction.
\end{remark}
\begin{remark} 
We point out that the results of Proposition \ref{p:lipdep} still hold if instead of a centered family of minimizers of $F_x$ we consider instead solutions $\gamma_x\in X_{-1}^1\cap BV(\R)$ to $$L_J\gamma_{x}(t)  = \partial_t W (x,\gamma_{x}(t)), \text{ in $\mathbb{R}$},$$
cf. \eqref{e:regopt2}. Also in this case, under suitable integrability assumptions on $J$, the continuity of the family $(\gamma_x)_x$ would be ensured by Proposition \ref{p:regoptintk} and Corollary \ref{c:holdrreg1}. Whereas for optimal profiles, though, also the existence of a centered family is guaranteed a priori up to translations, cf. Lemma \ref{r:cenfam}, this is not the case for stationary points, for which this further property would need to be assumed. We also point out that the $BV$-regularity in this latter case would be needed to ensure that the $L_I$ has the representation in \eqref{e:regopt3} and does not involve principal-value operators.
\end{remark}

For every $x \in K$ we introduce the notation
   \[
   O_x := \{ \gamma \in X_{-1}^1 \ | \ \emph{$\gamma$ is increasing, continuous, and minimizes \eqref{e:lipdep0}} \}.
   \]
The next step consists in combining the finiteness principle for Lipschitz selection in Theorem \ref{t:finitelipsel} with the Lipschitz dependence of optimal profiles in Proposition \ref{p:lipdep}.

\begin{proposition}[First H\"older-selection of optimal profiles]
\label{p:lipsel}
   Let $K \subset  \mathbb{R}^m \EEE$ be a compact set. Let $J$ and $W$ satisfy the hypotheses of Proposition \ref{p:lipdep}.
   Assume that one of the following conditions is satisfied
   \begin{enumerate}
       \item  $Q_x(t):= \partial_t W(x,t)+t$ has a continuous inverse for every $(x,t) \in K \times [-1,1]$ and
       \begin{equation}
       \label{e:lipselint}
           \int_{\mathbb{R} \setminus (-\rho,\rho)} J(h)|h| \, dh < \infty  \ \ \text{ and } \ \  \|J\|_{L^1(\mathbb{R})}=1
       \end{equation}
       \item $J \in \mathcal{L}_1(\eta,\lambda,\rho)$ for some $\eta,\lambda,\rho \in (0,1)$.
   \end{enumerate}
   Then, there exists a weight function $\overline{\sigma}_J \colon \mathbb{R} \to (0,\infty)$ (depending also on $K$) such that for every $\alpha \in (1/2,1]$ we find a family $(\gamma_x)_{x \in K}$ with $\gamma_x \in O_x$ and
   \begin{equation}
             \label{e:lipsel1}
             \|\gamma_x -\gamma_{x'}\|_{L^1(\mathbb{R};\overline{\sigma}_J)} \leq L\, \lfloor \partial_t W \rfloor_{C^{0,\alpha}(K \times [-2,2])} |x-x'|^{\alpha^2}, \ \ \text{ for $x,x' \in K$},
         \end{equation}
         where $L:= L(\delta,K,\alpha,\|\sigma_J\|_{L^1}) \geq 0$, $\delta$ is the parameter satisfying (ii)-(iii) of Proposition \ref{p:lipdep}, and $\sigma_J$ is the weight function given in \eqref{e:lipdep1.428}.
\end{proposition}
\begin{proof}
    Let $(\gamma_x)_{x \in K}$ be the centered family given by Lemma \ref{r:cenfam} satisfying $\gamma_x \in O_x$ and let $\sigma_J$ be the weight function introduced in \eqref{e:lipdep1.428}.
For the rest of the proof we can assume without loss of generality that $\lfloor \partial_t W \rfloor_{C^{0,\alpha}(K \times [-2,2])} < \infty$.   Let $ R'(\omega,K)>0$ ($\omega >0$) the quantity introduced in Definition \ref{d:cenfam} and relative to the centered family $(\gamma_x)_{x \in K}$ and set 
$$\pm \tilde{R}(\omega) := \pm (R'(\omega,K)+2R'(\delta/2,K)).$$ Then, for every $\omega>0$ we find
    \begin{align}
    \label{e:cenfam1.37}
   &1-\frac{\omega}{2} < \gamma_x(t-k) \leq 1, \text{ for $t \geq \tilde{R}(\omega)$, $x \in K$, $k \in[-2R'(\delta/2,K),2R'(\delta/2,K)]$} \\
   \label{e:cenfam2.37}
   & -1 \leq  \gamma_x(t-k) <-1+\frac{\omega}{2}, \text{ for $t \leq -\tilde{R}(\omega)$, $x \in K$, $k \in [-2R'(\delta/2,K),2R'(\delta/2,K)]$}. 
\end{align}
    Now, define $\sigma'_J \colon (-1,1) \to (0,\infty)$ as $$\sigma'_J(s):= \inf_{t \in [-\tilde{R}(2-2|s|),\tilde{R}(2-2|s|)]} \sigma_J(t) >0.$$

 In view of \eqref{e:cenfam1.37}--\eqref{e:cenfam2.37}, we can write for every $k \in [-2R'(\delta/2,K),2R'(\delta/2,K)]$, $x,x' \in K$, and $s \in (-1,1)$
\begin{equation}
    \label{e:containing}
    \gamma^{-1}_x(s) \in [-\tilde{R}(2-2|s|),\tilde{R}(2-2|s|)] \ \ \text{and} \ \ \gamma^{-1}_x(s) +k \in [-\tilde{R}(2-2|s|),\tilde{R}(2-2|s|)].
\end{equation}
For simplicity, we subdivide the remaining part of the proof into three steps.

\textbf{Step 1}. We first identify the framework to apply Theorem \ref{t:finitelipsel}. In view of the definition of $\sigma'_J$ and by \eqref{e:containing} we infer the following estimate for every $k \in [-R'(\delta/2,K),R'(\delta/2,K)]$ and $x,x' \in K$ 
\begin{align}
\nonumber
   & \int_{-1}^1 |\gamma^{-1}_x(s) - \gamma^{-1}_{x'}(s) +k|\sigma'_J(s) \, ds =  \int_{-1}^1 \bigg| \int_{\gamma^{-1}_x(s)}^{\gamma^{-1}_{x'}(s) +k} \sigma'_J(s) \, dt \bigg|ds \\
    \nonumber
    &\qquad\leq \int_{-1}^1 \bigg| \int_{\gamma^{-1}_x(s)}^{\gamma^{-1}_{x'}(s) +k} \sigma_J(t) \, dt \bigg|ds 
    = \int_{\mathbb{R}} \bigg| \int_{\gamma_x(t)}^{\gamma_{x'}(t-k)}  \, ds \bigg| \sigma_J(t)dt \\
    \label{e:equinormY}
    &\quad= \int_{\mathbb{R}} |\gamma_x(t) - \gamma_{x'}(t-k)| \sigma_J(t)\,dt.
    \end{align}

    Let $Y$ be the Banach space of all (equivalence classes of) Lebesgue measurable real-functions of $\mathbb{R}$ endowed with the norm $\|\cdot\|_Y$ defined as
    \begin{equation}
        \label{e:defnormY}
        \|y\|_Y := \int_{-1}^1 |y(s)| \sigma'_J(s) \, ds. 
    \end{equation}
     For every $x \in K$ consider the convex and compact set $C_x \subset Y$ defined as
    \[
    C_x:= \{ \gamma^{-1}_x +k \ | \ k \in [-2R'(\delta/2,K),2R'(\delta/2,K)] \subset \mathbb{R}\}.
    \]
     By virtue of our assumption (1)--(2) we can make use of Proposition \ref{t:eximin} and Proposition \ref{p:lipdep} to infer that $O_x \neq \emptyset$ for every $x \in K$ so that we have $C_x \in \mathcal{K}_1(Y)$ for every $x \in K$. 
     
     Consider the metric space $(K,d)$ where $d(x,x'):= |x-x'|^{\alpha^2}$ and consider the map $F \colon K \to \mathcal{K}_1(Y)$ defined as $F(x):= C_x$.  
     By virtue of Theorem \ref{t:finitelipsel}, in order to deduce the existence of a Lipschitz selection of $F$, it is enough to prove that there exists $\lambda>0$ such that given a collection of points $K' \subset K$ with $2 \leq \# K' \leq 4$, the restriction $F|_{K'}$ has a Lipschitz selection $f_{K'}$ with Lipschitz semi-norm less than or equal to $\lambda$. To this aim let $\lambda := 3C \lfloor \partial_t W\rfloor_{C^{0,\alpha}(K\times [-2,2])}$ where $C$ is the constant satisfying \eqref{e:lipdep1}. Clearly it is enough to consider the case $\#K'=4$ being the other two cases simpler. 
     
     Let us assume that $K' := \{x_1,x_2,x_3,x_4\}$ contains exactly four different points. We claim that, up to reorder the set $K'$, we end up with a graph $G$ consisting of two possible configurations
     \begin{enumerate}[(a)]
     \item nodes $\{x_1,x_2,x_3,x_4\}$ and links $\{(x_1,x_2),(x_2,x_3),(x_3,x_4)\}$
     \item nodes $\{x_1,x_2,x_3,x_4\}$ and links $\{(x_1,x_2),(x_2,x_3),(x_2,x_4)\}$
     \end{enumerate}
   satisfying in both cases the following property: given any couple of nodes $\{x_i,x_j\}$ ($i \neq j$) we can connect $x_i$ and $x_j$ with a path $p$ (made of links of $G$) such that for every $\ell:=(x_k,x_h) \in p$ it holds true $|\ell|:=|x_k-x_h| \leq |x_i-x_j|$. 
     
    \textbf{Step 2}. Here we describe how to conclude the proof of the proposition assuming that the claim in Step 1 holds true. Once the the claim is proved, the proof goes as follows: we can make use of Proposition \ref{p:lipdep} and of \eqref{e:equinormY} to find a selection $f_{K'}$ of $F|_{K'}$ satisfying in both cases (a) and (b) 
     \begin{align*}
    \text{$\|f(x_k)-f(x_{h})\|_Y \leq C \lfloor \partial_t W \rfloor_{C^{0,\alpha}} d(x_k,x_{h})$ whenever $(x_k,x_h)$ is a link of $G$}.
     \end{align*}
     If we consider two arbitrary indices $i,j =1,2,3,,4$, we can make use of the above stated property of $G$ to find a path $p$ connecting $x_i$ and $x_j$ and estimate
     \[
     \begin{split}
     \|f_{K'}(x_i)- f_{K'}(x_j)\|_Y &\leq \sum_{\ell \in p} C \lfloor \partial_t W \rfloor_{C^{0,\alpha}(K \times [-2,2])} |\ell|^{\alpha^2}\\
     &\leq \lambda |x_i-x_j|^{\alpha^2} = \lambda d(x_i,x_j).
     \end{split}
     \]
     Therefore we deduce the existence of a Lipschitz selection $f$ of $F$ with Lipschitz semi-norm not exceeding $q \lambda$ where $q >0$ is the dimensional constant given in Theorem \ref{t:finitelipsel}. In order to show \eqref{e:lipsel1} we proceed as follows. With a slight abuse of notation, for every $x \in K$ we denote by $\gamma_x$ the element in $O_x$ defined as the inverse of the function given by the selection $f$ evaluated at $x$. Define $m_K \colon [0,\infty) \to [0,1]$ as
     \begin{equation}
     m_K(t) :=\max\big\{\sup_{x \in K} |\gamma_x(t)|, \sup_{x \in K} |\gamma_x(-t)| \big\}.
     \end{equation}
     Since by Remark \ref{r:maxprinciple} we know that $-1<\gamma_x(t)<1$ for every $t \in \mathbb{R}$ and every $x \in K$, and since every sequence $(x_j) \subset K$ is such that the family $(\gamma_{x_j})_{j}$ satisfies condition \eqref{e:stronglycen} of Proposition \ref{p:contdep}, by a contradiction argument we obtain that $0 \leq m_K(t) <1$ for every $t \in \mathbb{R}$. Hence, by letting $\overline{\sigma}_J \colon \mathbb{R} \to (0,\infty)$ be defined as 
     \begin{equation}
     \label{e:defoversigma}
     \overline{\sigma}_J(t):= \inf_{s \in [-m_K(t),m_K(t)]} \sigma'_J(s),
     \end{equation}
     we see that $\overline{\sigma}_J(t) >0$ for every $t \in \mathbb{R}$.
     
     By applying Fubini's Theorem we deduce
     \begin{align*}
         &\int_{\mathbb{R}} |\gamma_x(t) - \gamma_{x'}(t)|\, \overline{\sigma}_J(t) \, dt = \int_{\mathbb{R}} \bigg| \int_{\gamma_x(t)}^{\gamma_{x'}(t)} \overline{\sigma}_J(t)\, ds \bigg|dt\\
         &\quad= \int_{\mathbb{R}} \bigg| \int_{\gamma_x(t)}^{\gamma_{x'}(t)} \inf_{s \in [-m(t),m(t)]} \sigma'_J(s) \, ds \bigg|dt
         \leq \int_{\mathbb{R}} \bigg| \int_{\gamma_x(t)}^{\gamma_{x'}(t)}  \sigma'_J(s) \, ds \bigg|dt\\
         &\quad= \int_{-1}^1 \bigg| \int_{\gamma_x^{-1}(s)}^{\gamma_{x'}^{-1}(s)} dt \bigg| \sigma'_J(s)\, ds \\
         &= \|f(x)-f(x')\|_Y.
     \end{align*}
     Property \eqref{e:lipsel1} immediately follows from the fact that $f$ is a Lipschitz selection of $F$ with semi-norm not exceeding $q \lambda$.
     
     \textbf{Step 3}. We are left to prove the claim in Step 1. We start with a set $K' \subset  \mathbb{R}^m \EEE$ ($m \geq 2$) containing four different points $\{x,y,z,w\}$. Up to reordering the set $K'$ we may suppose that $|x-y|$ minimizes the distance among all possible couples made of elements of $K'$. If $|z-w|$ minimizes the distance among all the possible couples different from $(x,y)$ then we choose an order of the form
      \[
      \{x,y,z,w\} \ \  \{x,y,w,z\} \ \ \{y,x,z,w\} \ \ \{y,x,w,z\}
      \]
      which minimizes all the possible distances between the second and the third elements.
Then by letting $x_1$ be the first element in our order, $x_2$ the second etc., we end up exactly in configuration (a). If instead, after removing the couple $(x,y)$ one among $(x,z),(x,w), (y,z), (y,w)$, is the couple minimizing the distance, we have two possible cases: let us suppose that $(x,z)$ is the minimizing couple (exactly the same argument can be applied in the other three cases) then, either one of the following two subcases occurs: 
\begin{enumerate}[(i)]
    \item $|z-y|$ minimizes the distance among all couples after removing $(x,y)$ and $(x,z)$ ($x,y,z$ are the vertices of an equilateral triangle)
    \item one among $(z,w),(x,w),(y,w)$ minimizes the distance among all couples after removing $(x,y)$ and $(x,z)$.
\end{enumerate}
     In case (ii) we have two possible subcases: if $(z,w)$ is minimizing the distance then we set $x_1:= y$, $x_2:= x$, $x_3:=z$, and $x_4:= w$ while if $(y,w)$ is minimizing the distance then we set $x_1:= w$, $x_2:= y$, $x_3:=x$, and $x_4:= z$, and in both cases we are in the (a) configuration; if $(x,w)$ is minimizing the distance then we set $x_1=y$, $x_2=x$, $x_3=z$, $x_4=w$ and we are in the (b) configuration. 

     In case (i) we can proceed exactly as in case (ii) but by replacing the condition \emph{if $(\cdot,\cdot)$ is minimizing the distance} by the condition \emph{if $(\cdot,\cdot)$ is minimizing the distance after further removing the couple $(z,y)$}. Therefore, the only fact that remains to verify, is that in both configurations (a) and (b) (as in case (ii)) the couple $(z,y)$ can be connected with a path $p$ of $G$ such that every $\ell \in p$ satisfies $|\ell| \leq |z-y|$. But since in case (i) the points $x,y,z$ are the vertices of an equilateral triangle, we simply choose the path $p:= (y,x) \cup (x,z)$ which in each cases is a path made of links of the chosen graph $G$. This concludes the proof. 
\end{proof}

We are finally in a position to prove the main result of this section, extending the results of Proposition \ref{p:lipsel} to the case in which the two wells of $W$ are $x$-dependent and $\alpha$-H\"older continuous.

\begin{theorem}
\label{p:lipselx}
    Let $K \subset  \mathbb{R}^m \EEE$ be compact, let $J \colon  \mathbb{R}^m \EEE \to (0,\infty)$ be an even function and let $W \colon \Omega \times \mathbb{R} \to [0,\infty)$ satisfy \ref{H1}--\ref{H4}.  Let us further denote for every $x \in K$ 
   \[
   O_x := \{ \gamma \in X_{z_1(x)}^{z_2(x)} \ | \ \emph{$\gamma$ is increasing and continuous, and minimizes \eqref{e:lipdep0}} \}.
   \]
    
  Assume that one of the following conditions is satisfied
    \begin{enumerate}[(i)]
       \item  $Q_x(t):= \partial_t W(x,t)+t$ has a continuous inverse for every $x \in K$ and $t \in [z_1(x),z_2(x)]$, and
       \begin{equation}
       \label{e:lipselintx1}
          \int_{\mathbb{R} \setminus (-1,1)} J(h)|h| \, dh < \infty  \ \ \text{ and } \ \  \|J\|_{L^1(\mathbb{R})}=1
       \end{equation}
       \item $J \in \mathcal{L}_1(\eta,\lambda,\rho)$ for some $\eta,\lambda,\rho \in (0,1)$,
   \end{enumerate}
   Then there exists a weight function $\overline{\sigma}_J \colon \mathbb{R} \to (0,\infty)$ (depending also on $K$) such that for every $\alpha \in (1/2,1]$  we find a family $(\gamma_x)_{x \in K}$ such that $\gamma_x \in O_x$ and 
   \begin{equation}
             \label{e:lipselx1}
            \|\gamma_x -\gamma_{x'}\|_{L^1(\mathbb{R};\overline{\sigma}_J)} \leq L\, \lfloor\partial_t W \rfloor_{C^{0,\alpha}(K \times [-M,M])} |x-x'|^{\alpha^2}, \ \ \text{ for $x,x' \in K$},
         \end{equation}
          where   $L:= L(\delta_K,K,\|\sigma_J\|_{L^1},\|\overline{\sigma}_J\|_{L^1},\lfloor z_1 \rfloor_{C^{0,\alpha}},\lfloor z_2 \rfloor_{C^{0,\alpha}}) > 0$, $\delta_K$ is the parameter satisfying \ref{H4}, $\sigma_J$ is the weight function given in \eqref{e:lipdep1.428}, and $M:= 2\max_{x \in K} |z_1(x) \vee z_2(x)|$.
    \end{theorem}

\begin{proof}
    Let $q,p \colon  \mathbb{R}^m \EEE \times \mathbb{R} \to \mathbb{R}$ be defined as $q(x,t):= \frac{z_2(x)-z_1(x)}{2}t + \frac{z_1(x)+z_2(x)}{2}$ and $p(x,t):=  \frac{2}{z_2(x)-z_1(x)}\big(t -\frac{z_1(x)+z_2(x)}{2}\big)$. Notice that $q(x,[-1,1])= [z_1(x),z_2(x)]$ and $q(x,p(x,t))=t$ for every $(x,t) \in  \mathbb{R}^m \EEE \times \mathbb{R}$. By defining $\tilde{W} \colon  \mathbb{R}^m \EEE \times \mathbb{R} \to [0,\infty)$ as $$\tilde{W}(x,t):= \frac{4W(x,q(x,t))}{(z_1(x)-z_2(x))^2},$$ we see that, by virtue of \ref{H1}--\ref{H4}, $\tilde{W}$ satisfies the assumptions of Proposition \ref{p:lipdep}. A change of variables shows that $\gamma \in X_{z_1(x)}^{z_2(x)}$ minimizes \eqref{e:lipdep0} if and only if $\tilde{\gamma} \in X_{-1}^1$ defined as $\tilde{\gamma}(t):= p(x,\gamma(t))$ minimizes
     \begin{equation}
         \label{e:lipselx2}
       \tilde{F}_x(\gamma)=\frac{1}{4}\int_{\mathbb{R}\times \mathbb{R}} J(t'-t)|\gamma(t')-\gamma(t)|^2 \, dt'dt + \int_{\mathbb{R}} \tilde{W}(x,\gamma(t)) \, dt.
         \end{equation}
         Hence, we are in position to apply Proposition \ref{p:lipsel} and to find a family $(\tilde{\gamma}_x)_{x \in K}$ of increasing minimizers for \eqref{e:lipselx2} such that $\tilde{\gamma}_x \in X_{-1}^1$ and \eqref{e:lipsel1} is satisfied. In view of \eqref{e:lipsel1} and by the identity $q(x,p(x,t))=t$ we deduce
         \begin{align*}
            &\|\gamma_x -\gamma_{x'}\|_{L^1(\mathbb{R};\overline{\sigma}_J)} =  \int_{\mathbb{R}} |q(x,\tilde{\gamma}_x(t)) -q(x',\tilde{\gamma}_{x'}(t))| \overline{\sigma}_J(t)\, dt \\
             &\quad\leq \int_{\mathbb{R}} |q(x,\tilde{\gamma}_x(t)) -q(x,\tilde{\gamma}_{x'}(t))| \overline{\sigma}_J(t)\, dt + \int_{\mathbb{R}} |q(x,\tilde{\gamma}_{x'}(t)) -q(x',\tilde{\gamma}_{x'}(t))| \overline{\sigma}_J(t)\, dt \\
             &\quad\leq \lfloor q\rfloor_{C^{0,\alpha}(K \times [-2,2])} \int_{\mathbb{R}} |\tilde{\gamma}_x(t) -\tilde{\gamma}_{x'}(t)| \overline{\sigma}_{J}(t)\, dt
             +  \lfloor q\rfloor_{C^{0,\alpha}(K \times [-2,2])}  \|\overline{\sigma}_J\|_{L^1} |x-x'|^{\alpha} \\
             &\quad\leq  \lfloor q\rfloor_{C^{0,\alpha}(K \times [-2,2])}  (L \,\lfloor \partial_t W \rfloor_{C^{0,\alpha}(K \times [-M,M])} +  \|\overline{\sigma}_J\|_{L^1}) |x-x'|^{\alpha^2}.
         \end{align*}
         Since $\lfloor q\rfloor_{C^{0,\alpha}(K \times [-2,2])} < \infty$ by \ref{H1}--\ref{H2}, we obtain the thesis.
\end{proof}

\begin{remark}
    \label{r:L^2locestimate}
For every $R >0$ condition \eqref{e:lipselx1} implies that the family $(\gamma_x)$ provided by Theorem \ref{p:lipselx} satisfies for every $x,x' \in K$
\begin{equation}
\label{e:truncatedlipsel}
    \int_{-R}^R |\gamma_x(t) -\gamma_{x'}(t)|^2 \, dt \leq \frac{ML}{\inf_{t \in (-R,R)} \overline{\sigma}_J(t)} \, \lfloor \partial_t W \rfloor_{C^{0,\alpha}(K \times [-M,M])} |x-x'|^{\alpha^2}. 
\end{equation}
\end{remark}

\begin{remark}
\label{r:bormeas}
    Let $\xi \in  \mathbb{S}^{m-1} \EEE$, let $(\gamma_x)$ be the family provided by Theorem \ref{p:lipselx}, and let $V$ denote any $(n-1)$-dimensional affine plane of $ \mathbb{R}^m \EEE$ which is orthogonal to $\xi$. Define $f \colon (V \cap \Omega) \times \mathbb{R} \to \mathbb{R}$ as $f(x):= \gamma_y(t)$ for $x=y+t\xi$ and $(y,t) \in (V \cap \Omega) \times \mathbb{R}$. Then, the family $(\gamma_y)_{y \in V \cap \Omega}$ is Borel measurable (see \cite[Section 5.3]{ags}) and $ D_\xi f= \dot{\gamma}_y \otimes \mathcal{H}^{n-1} \restr (V \cap \Omega)$ as measures in $(V \cap \Omega) \times \mathbb{R}$, meaning that
    \[
  D_\xi f(B) = \int_{V \cap \Omega}  \bigg(\int_{\{t \ | \ \gamma_y(t) \in B \}} \dot{\gamma}_y(t)\,dt \bigg) d\mathcal{H}^{n-1}(y), 
    \]
     for every $B \subset \Omega$ Borel measurable.
\end{remark}

\section{Locality defect and $\ep$-traces} 
\label{s:locdef}
In this section we always assume $J \colon \mathbb{R}^m\EEE \to [0,\infty)$ to be a function satisfying \ref{K2}.  We recall here the notion of \EEE defect functional  introduced in \cite{alb-bel2}.  Given measurable sets $A,B \subset \mathbb{R}^m\EEE$ and a measurable function $u \colon \mathbb{R}^m\EEE \to \mathbb{R}$, for every $\ep>0$ \EEE the locality defect between $A$ and $B$ is defined as \EEE
\begin{equation}
\label{e:deficit}
    \Lambda_\ep(u;A,B):= \frac{1}{4\ep}\int_{A \times B} J_\ep(x-x')(u(x)-u(x'))^2 \, dxdx'.
\end{equation}
We denote $\Lambda_\ep(u;A):= \Lambda_\ep(u;A,A)$. 

 Consider \EEE the auxiliary potential $\hat{J}$, given by
\begin{equation}
    \hat{J}(h) := \int_0^1 J\bigg(\frac{h}{t} \bigg) \bigg|\frac{h}{t}\bigg| \frac{dt}{t^n}, \ \ \text{ for } h \in \mathbb{R}^m\EEE.
\end{equation}
 By definition, $\hat{J}$ is a non-negative kernel. In view of \ref{K2}, it \EEE  satisfies
\[
\|\hat{J}\|_{L^1(\mathbb{R})} = \int_{\mathbb{R}^m\EEE} J(h)|h| \, dh < \infty.
\]

 As shown in \cite{alb-bel2}, given a sequence of maps $(u_{\ep})$, the asymptotic value of $\Lambda_{\ep}(u_{\ep},A,B)$ is, roughly speaking, only determined by the behavior of the sequence $(u_{\ep})_{\ep}$ close to the intersection of the boundaries of $A$ and $B$. In order to make this statement rigorous, in in \cite{alb-bel2} the authors introduce a notion of ``asymptotic trace" for sequences $(u_{\ell})_{\ell}$ which belong to the domain of $\Lambda_{\ep}(u_{\ep},A,B)$ but enjoy, a priori, no $BV$-regularity, so that a proper notion of trace would not be well defined. We recall \cite[Definition 2.1]{alb-bel2} below. \EEE

\begin{definition}[$\epsilon$-traces convergence]
    Let $\Sigma \subset  \mathbb{R}^m \EEE$ denote a subset of a Lipschitz hypersurface and let $A \subset  \mathbb{R}^m \EEE$ denote a measurable set. Given a sequence $u_\ell \colon A \to \mathbb{R}$ of measurable functions and a sequence of numbers $\epsilon \searrow 0$, we say that the $\epsilon_\ell$-traces of $u_\ell$ relative to $A$ converge on $\Sigma$ to $v \colon \Sigma \to \mathbb{R}$ if
    \begin{equation}
    \label{e:epsilontrace}
    \lim_{\ell\to \infty} \int_\Sigma \bigg( \int_{\{h \ | \ y + \epsilon_\ell h \in A\}} \hat{J}(h) |u_\ell(y+\epsilon_\ell h) -v(y)| \, dh \bigg)dy =0.
    \end{equation}
\end{definition}

 The next proposition ensures that $L^1$-convergence of a sequence $(u_{\ell})$ to $u$ is enough to obtain convergence of the $\ep$-traces for most Lipschitz hypersurfaces.
For its proofs we refer to \cite[Proposition 2.5]{alb-bel2}. \EEE

\begin{proposition}
\label{p:etrace_convergence}
Let $A \subset  \mathbb{R}^m \EEE$ be a set of positive measure, let $u_\ell \colon A \to [-1,1]$, let $(\epsilon_\ell)$ be a sequence of positive numbers converging to $0$, and let $f \colon A \to \mathbb{R}$ be Lipschitz. If $u_\ell \to u$ in $L^1(A)$ then  the $\epsilon_\ell$-traces of $u_\ell$ relative to $A$ converge to $u$ on $f^{-1}(t)$ for a.e. $t \in \mathbb{R}$. 
\end{proposition}

We conclude this section by recalling a result (see \cite[Theorem 2.8]{alb-bel2}) showing how the notion of locality defect is related to that of convergence of $\ep$-traces. \EEE
\begin{definition}
\label{d:divided}
    Let $A,A' \subset  \mathbb{R}^m \EEE$ and let $\Sigma$ be a subset of some Lipschitz hypersurface of $ \mathbb{R}^m \EEE$. We say that $A,A'$ are divided by $\Sigma$ if for every $x\in A$ and $x' \in A'$ the segment joining $x$ and $x'$ intersects $\Sigma$
\end{definition}

 For the statement and proof of the next proposition we refer to \cite[Theorem 2.8]{alb-bel2}. \EEE

\begin{proposition}
\label{p:deficit}
Let $A,A' \subset  \mathbb{R}^m \EEE$ be measurable sets divided by $\Sigma$. Then given a sequence of positive numbers $\epsilon_\ell \searrow 0$ and a sequence of functions $u_\ell \colon A \cup A' \to \mathbb{R}$ we have
\begin{equation}
    \label{e:def_estimate1}
    \limsup_{\ell \to \infty} \Lambda_{\epsilon_\ell}(u_\ell;A,A') \leq \|\hat{J}\|_{L^1( \mathbb{R}^m \EEE)} \limsup_{\ell\to \infty}\|u_\ell\|_{L^\infty(A \cup A')} \, \mathcal{H}^{n-1}(\Sigma),
\end{equation}
for some constant $C>0$ depending only on the kernel $J$. Moreover, if the $\epsilon_\ell$-trace of $u_\ell$ relative to $A$ and $A'$ converges on $\Sigma$ to $v$ and $v'$, respectively, then
\begin{equation}
    \label{e:def_estimate2}
    \limsup_{\ell \to \infty} \Lambda_{\epsilon_\ell}(u_\ell;A,A') \leq \frac{1}{2}\|\hat{J}\|_{L^1( \mathbb{R}^m \EEE)}  \, \int_{\Sigma}|v -v'|d\mathcal{H}^{n-1}.
\end{equation}
\end{proposition}
We conclude this section with a technical result about convergence of $(\ep)$-traces which will be instrumental in the construction of the recovery sequence for our $\Gamma$-convergence result.
\begin{proposition}
\label{r:defuniform}
Let $\Sigma \subset  \mathbb{R}^m \EEE$ be a compact $(n-1)$-dimensional Lipschitz manifold, and let $A\subset \R^m$ be a measurable set. Suppose that $v \colon A \cup \Sigma \to \mathbb{R}$ is a bounded function such that for every $y \in \Sigma$ there holds 
\begin{equation}
\label{e:contvdef}
\lim_{\rho \to 0} \sup_{y \in \Sigma} \sup_{\substack{x+y \in A \\ x \in B_\rho(0)} } |v(y+x)-v(y)|=0,
\end{equation}
where the second supremum above is set to be $0$ whenever $\{x \in B_\rho(0) \ | \ x+y \in A \}=\emptyset$. Suppose furthermore that the maps $u_\ell \colon A \to \mathbb{R}$ are equi-bounded functions such that $u_\ell \to v$ uniformly on every compact subset of $A \cap U$ for some open neighborhood $U$ of $\Sigma$. Then the $\epsilon_\ell$-traces of $u_\ell$ relative to $A$ converge on $\Sigma$ to $v$. 
\end{proposition}
\begin{proof}
We first observe that
\begin{align}
     &\nonumber\limsup_{\ell\to \infty} \int_\Sigma \bigg( \int_{\{h \ | \ y + \epsilon_\ell h \in A\}} \hat{J}(h) |u_\ell(y+\epsilon_\ell h) -v(y)| \, dh \bigg)dy \\
     &\nonumber\leq \limsup_{\ell\to \infty} \int_\Sigma \bigg( \int_{\{h \ | \ y + \epsilon_\ell h \in A\}} \hat{J}(h) |u_\ell(y+\epsilon_\ell h) -v(y+\epsilon_\ell h)| \, dh \bigg)dy \\
     &\label{eq:basic-dec}+\limsup_{\ell\to \infty} \int_\Sigma \bigg( \int_{\{h \ | \ y + \epsilon_\ell h \in A\}} \hat{J}(h) |v(y+\epsilon_\ell h) -v(y)| \, dh \bigg)dy
\end{align}
Now choose a compact set $K \subset A$ such that $\int_{A \setminus K} \hat{J}(h)dh \leq \delta$ for some arbitrary and positive $\delta$. Then, we can exploit the uniform convergence of $(u_{\ell})_{\ell}$ on $K\cap U$ to write
\begin{align*}
&\limsup_{\ell\to \infty} \int_\Sigma \bigg( \int_{\{h \ | \ y + \epsilon_\ell h \in A\}} \hat{J}(h) |u_\ell(y+\epsilon_\ell h) -v(y+\epsilon_\ell h)| \, dh \bigg)dy \\
&\leq \limsup_{\ell\to \infty} \int_\Sigma \bigg( \int_{\{h \ | \ y + \epsilon_\ell h \in K \cap U\}} \hat{J}(h) |u_\ell(y+\epsilon_\ell h) -v(y+\epsilon_\ell h)| \, dh \bigg)dy \\
&+ \limsup_{\ell\to \infty} \int_\Sigma \bigg( \int_{\{h \ | \ y + \epsilon_\ell h \in K \setminus U\}} \hat{J}(h) |u_\ell(y+\epsilon_\ell h) -v(y+\epsilon_\ell h)| \, dh \bigg)dy \\
&+ \limsup_{\ell\to \infty} \int_\Sigma \bigg( \int_{\{h \ | \ y + \epsilon_\ell h \in A \setminus K\}} \hat{J}(h) |u_\ell(y+\epsilon_\ell h) -v(y+\epsilon_\ell h)| \, dh \bigg)dy \\
&\leq \mathcal{H}^{n-1}(\Sigma)\limsup_{\ell \to \infty}(\|u_\ell\|_{L^\infty}+ \|v\|_{L^\infty}) \int_{h \in \epsilon_\ell^{-1}((K\setminus U) -y) } \hat{J}(h) \, dh \\
&+  \delta\mathcal{H}^{n-1}(\Sigma) \limsup_{\ell \to \infty} (\|u_\ell\|_{L^\infty}+ \|v\|_{L^\infty}).
\end{align*}
Notice that, since $y \in \Sigma$ and $\Sigma \Subset U$, then there exists a sufficiently small radius $\rho >0$ such that $(K \setminus U) -y \subset  \mathbb{R}^m \EEE \setminus B_\rho(0)$. Therefore, since $\hat{J} \in L^1( \mathbb{R}^m \EEE)$ we deduce that $\int_{h \in \epsilon_\ell^{-1}((K\setminus U) -y) } \hat{J}(h) \, dh \to 0$ as $\ell \to \infty$. Hence, the arbitrariness of $\delta$ yields 
\[
\lim_{\ell \to \infty}\int_\Sigma \bigg( \int_{\{h \ | \ y + \epsilon_\ell h \in A\}} \hat{J}(h) |u_\ell(y+\epsilon_\ell h) -v(y+\epsilon_\ell h)| \, dh \bigg)dy = 0.
\]
In order to estimate the second limsup in \eqref{eq:basic-dec}, we consider a sufficiently large $R>0$ such that $\int_{ \mathbb{R}^m \EEE \setminus B_R(0)} \hat{J}(h)dh \leq \delta$ for some arbitrary and positive $\delta$. With this choice, we infer
\begin{align*}
    &\limsup_{\ell\to \infty} \int_\Sigma \bigg( \int_{\{h \ | \ y + \epsilon_\ell h \in A\}} \hat{J}(h) |v(y+\epsilon_\ell h) -v(y)| \, dh \bigg)dy \\
&\leq \limsup_{\ell\to \infty} \int_\Sigma \bigg( \int_{\{h \ | \ y + \epsilon_\ell h \in A\} \cap B_R(0)} \hat{J}(h) |v(y+\epsilon_\ell h) -v(y)| \, dh \bigg)dy \\
&+ \limsup_{\ell\to \infty} \int_\Sigma \bigg( \int_{\{h \ | \ y + \epsilon_\ell h \in A\} \setminus B_R(0)} \hat{J}(h) |v(y+\epsilon_\ell h) -v(y)| \, dh \bigg)dy \\
&\leq \limsup_{\ell\to \infty} \int_\Sigma \bigg( \int_{\{h \ | \ y + \epsilon_\ell h \in A\} \cap B_R(0)} \hat{J}(h) |v(y+\epsilon_\ell h) -v(y)| \, dh \bigg)dy \\
&+\delta \mathcal{H}^{n-1}(\Sigma) 2 \|v\|_{L^\infty} 
\end{align*}
 In view of \eqref{e:contvdef}, there exists $\overline{\ell}=\overline{\ell}(\delta)$ such that for every $\ell \geq \overline{\ell}(\delta)$ the following holds true: $|v(y+\epsilon_\ell h) -v(y)| \leq \delta$ for every $y \in \Sigma$ and every $h$ satisfying $y+\epsilon_{\ell}h \in A$ and $|h| < R$. Therefore we have
 \[
 \lim_{\ell\to \infty} \int_\Sigma \bigg( \int_{\{h \ | \ y + \epsilon_\ell h \in A\} \cap B_R(0)} \hat{J}(h) |v(y+\epsilon_\ell h) -v(y)| \, dh \bigg)dy =0.
 \]
Thanks to the arbitrariness of $\delta>0$ we conclude that
 \[
 \lim_{\ell \to \infty}\int_\Sigma \bigg( \int_{\{h \ | \ y + \epsilon_\ell h \in A\}} \hat{J}(h) |v(y+\epsilon_\ell h)-v(y)| \, dh \bigg)dy = 0.
 \]
 This completes the proof.
\end{proof}
\EEE

\section{Compactness and slicing preliminaries}
\label{s:sec5}
By virtue of condition \ref{H5} (see Remark \ref{r:convimpl}) we fix for the entire section an integer $M$ satisfying 
\begin{equation}
    \label{e:boundwells123}
      \max_{x \in \overline{\Omega}} \max\{|z_1(x)|,|z_2(x)|\} \EEE \leq M  \ \ \text{ and } \ \ W(x,\pm M) \leq W(x,t), \ \ \text{ $|t| \geq M$, $x \in \overline{\Omega}$}.
\end{equation}
This section is devoted to establish compactness of sequences of maps with equibounded energies, as well as to state and prove some basic results which will be instrumental in the proof of our two $\Gamma$-convergence results.

In particular, we provide a  characterizations of kernels and potentials arising by restricting our functionals to $m-1$ hyperplanes.  

\subsection{Compactness} We will show that any equibounded sequences of the energy \eqref{e:functional} admits a subsequence $L^1$-converging to an element in $CP(\Omega;\mathbf{z}(x))$. 

\begin{proposition}
\label{p:compactness}
    Let $\Omega \subset  \mathbb{R}^m \EEE$ be open, let $J \colon  \mathbb{R}^m \EEE \to [0,\infty)$ be a function satisfying \ref{K3}, and let $W \colon  \mathbb{R}^m \EEE \times \mathbb{R} \to [0,\infty)$ satisfy \ref{H1}--\ref{H2} and \ref{H5}. Consider a family of functions $u_\epsilon \colon \Omega \to \mathbb{R}$ such that  $\limsup_{\epsilon} F_\epsilon(u_\epsilon) < \infty$. Then, there exists $u \in CP(\Omega;\mathbf{z}(x))$ such that, up to the extraction of a non-relabelled subsequence, there holds $u_\epsilon \to u$ in $L^1(\Omega)$.
\end{proposition}
\begin{proof}
    Without loss of generality we can assume that $J$ is not identically zero. Let $N$ be a positive integer. By letting $J^N \colon  \mathbb{R}^m \EEE \to [0,\infty)$ be defined as $J^N(h):= J(h) \wedge N$, and by letting $F^N_\epsilon$ be the energy in \eqref{e:functional} with $J$ replaced by $J^N$, then we have that $J^N$ satisfies \ref{K1} and $\sup_{\epsilon} F^N_\epsilon(u_\epsilon) < \infty$. Therefore, it is enough to consider the setting in which $J \in L^1( \mathbb{R}^m \EEE) \cap L^\infty( \mathbb{R}^m \EEE)$.

The proof strategy consists in modifying the sequence $(u_\epsilon)_{\ep}$ in order to end up with functions $c_\epsilon \colon \Omega \to \{-1,1\}$ for which we can apply Step 1 of the compactness argument in \cite[Theorem 3.1]{alb-bel2}.
For convenience of the reader, we subdivide the  proof into three steps. In order to lighten the notation, throughout the proof we will replace the quantities $\delta_{\bar\Omega}$ and $M_{\bar\Omega}$ from \ref{H2} and Remark \ref{r:convimpl}, with $\delta$ and $M$, respectively.
    
    \textbf{Step 1}.  Denote by $\Pi_x$ the projection of $\mathbb{R}$ onto the interval $[z_1(x),z_2(x)]$. We claim that the family given by $u'_\epsilon := \Pi_x(u_\epsilon)$ is equi-bounded in energy and satisfies 
    \begin{equation}
    \label{eq:smallL1}
        \|u'_{\ep}-u_{\ep}\|_{L^1(\Omega)}\to 0
    \end{equation} as $\ep\to 0$. 
    
    Indeed, we first notice that for every $x,x' \in \Omega$ one and only one of the following cases occurs
    \begin{enumerate}
    \item $|\Pi_{x'}(t) - \Pi_x(t)| = |z_1(x)-z_2(x')|$ if $t \leq z_1(x)$ and $t\geq z_2(x')$ 
        \item $|\Pi_{x'}(t) - \Pi_x(t)| = |z_1(x')-z_2(x)|$ if $t \leq z_1(x')$ and $t \geq z_2(x)$ 
        \item $|\Pi_{x'}(t) - \Pi_x(t)| = 0$ if $t \in [z_1(x),z_2(x)] \cap [z_1(x'),z_2(x')]$
        \item $|\Pi_{x'}(t) - \Pi_x(t)| \leq |z_i(x)-z_i(x')|$ otherwise for some $i=1,2$.
    \end{enumerate}
    We further notice that, if (1) or (2) hold true, then we make use of \ref{H2} to infer $|x-x'| \geq 64 \delta^2 \min_{i=1,2}\lfloor z_i \rfloor_{C^{(1+\alpha)/2}}^{-2}$. We thus write
    \begin{align*}
       & \int_{\Omega \times \Omega} J_\epsilon(x'-x) |\Pi_{x'}(u(x')) - \Pi_x(u(x))|^2 \, dx'dx \leq 2 \int_{\Omega \times \Omega} J_\epsilon(x'-x) |\Pi_{x'}(u(x')) - \Pi_{x'}(u(x))|^2 \, dx'dx \\
       &\quad+2\int_{\Omega \times \Omega} J_\epsilon(x'-x) |\Pi_{x'}(u(x)) - \Pi_{x}(u(x))|^2 \, dx'dx 
        \leq 2 \lfloor \Pi_{x'}(\cdot) \rfloor_{C^1} \int_{\Omega \times \Omega} J_\epsilon(x'-x) |u(x') - u(x)|^2 \, dx'dx \\
       &\quad+2 \max_{i=1,2}\lfloor z_i \rfloor_{C^{(1+\alpha)/2}}^2  \int_{\Omega \times \Omega} J_\epsilon(x'-x) |x'-x|^{1+\alpha} \, dx'dx 
       +8 M^2 \int_{\{|x-x'| \geq \beta \}} J_\epsilon(x'-x)  \, dx'dx,
        \end{align*}
        where we set $\beta := 64 \delta^2 \min_{i=1,2}\lfloor z_i \rfloor_{C^{(1+\alpha)/2}}^{-2}$ (notice also that $\lfloor \Pi_{x'}(\cdot) \rfloor_{C^1} \leq 1$).
        Now we use \ref{K3} and argue similarly to Lemma \ref{l:decaykernel} to infer
        \[
        \lim_{\epsilon}\frac{1}{\epsilon} \int_{\Omega \times \Omega} J_\epsilon(x'-x) |x'-x|^{1+\alpha} \, dx'dx = 0. 
        \]
        Additionally,
        \begin{align*}
        \limsup_{\epsilon} &\frac{1}{\epsilon} \int_{\{|x-x'| \geq \beta\}} J_\epsilon(x'-x)  \, dx'dx 
        = \limsup_{\epsilon} \epsilon^{n-1} \int_{\{|h| \geq \epsilon^{-1}\beta \} \times \epsilon^{-1} \Omega} J(h)  \, dhdx \\
        &\leq \limsup_{\epsilon} \frac{\epsilon^{n}}{\beta} \int_{\{|h| \geq \epsilon^{-1}\beta \} \times \epsilon^{-1} \Omega} J(h) |h| \, dhdx = \limsup_{\epsilon} \frac{|\Omega|}{\beta} \int_{\{|h| \geq \epsilon^{-1}\beta \}} J(h) |h| \, dh=0,
        \end{align*}
        where in the last equality we used \ref{K3}. In addition, property \ref{H1} of the potential $W$ entails
        \[
        \int_\Omega W(u'_\epsilon(x)) \, dx \leq  \int_\Omega W(u_\epsilon(x)) \, dx.
        \]
        Therefore we can summarize by saying that $\limsup_{\epsilon} F_\epsilon(u'_\epsilon) < \infty$.
        
        To prove that the $L^1$-difference between $u'_{\ep}$ and $u_{\ep}$ converges to zero, fix $\eta>0$ arbitrarily small, and let
        \begin{equation}
        \label{eq:def-Seta}
            S^{\ep}_{\eta}:=\{x\in \Omega:\,\text{dist}(u_{\ep}(x);\mathbf{z}(x))>\eta\}.
        \end{equation}
        By Remark \ref{r:convimpl}, we find
        \begin{align}
            &\nonumber\int_{\Omega}|u_{\ep}(x)-u_{\ep}'(x)|dx=\int_{ S^{\ep}_{\eta}}|u_{\ep}(x)-u_{\ep}'(x)|dx+\int_{\Omega\setminus S^{\ep}_{\eta}}|u_{\ep}(x)-u_{\ep}'(x)|dx\\
            &\label{eq:l1diff}\quad\leq \int_{S^{\ep}_{\eta}}|u_{\ep}(x)-u_{\ep}'(x)|dx+\eta|\Omega|.  
        \end{align}
         By \ref{H3}, Remark \ref{r:convimpl}, and the uniform bound on the energies $(F_{\ep}(u_{\ep}))_{\ep}$ we infer the bound
         \begin{equation}
         \label{eq:boundS}
             f(\eta)|S^{\ep}_{\eta}|\leq \int_{S^{\ep}_{\eta}}f(\text{dist}(u_{\ep}(x);\mathbf{z}(x)))dx\leq C\ep,
         \end{equation}
         so that, by \ref{H5} and Remark \ref{r:convimpl}, the first term on the right-hand side of \eqref{eq:l1diff} can be estimated as
         \begin{align} 
             &\nonumber\int_{ S^{\ep}_{\eta}}|u_{\ep}(x)-u_{\ep}'(x)|dx\leq \left(\frac1\delta+ M \right)|S^{\ep}_{\eta}|+\int_{\{x\in \Omega:|u_{\ep}(x)|>\frac1\delta\}}|u_{\ep}(x)|dx\\
             &\quad \leq \frac{C\ep\left(\frac1\delta+ M \right)}{f(\eta)}+\frac{C\ep}{\delta}.
             \label{eq:equiint}
         \end{align}
         By combining \eqref{eq:l1diff} with \eqref{eq:equiint}, we deduce
         $$\limsup_{\ep\to 0}\int_{\Omega}|u_{\ep}(x)-u_{\ep}'(x)|dx\leq \eta|\Omega|.$$
         The claim follows then from the arbitrariness of $\eta$.
         
        \textbf{Step 2}. Fix again $\eta\in (0,1)$, and define $c_\epsilon \colon \Omega \to \{-1,1\}$ as follows
        \[
        c_\epsilon(x):=
        \begin{cases}
            1 &\text{ if $|u'_\epsilon(x) -z_2(x)| \leq  \delta$} \\
            -1 &\text{ otherwise}.
        \end{cases}
        \]
        We observe that, by denoting $L:= \inf_{x \in \Omega, t \in [z_1(x)+\delta,z_2(x)-\delta]} W(x,t) > 0$
        \begin{align}
        \label{e:compactness1}
       &\limsup_{\epsilon} \frac{|\{x \in \Omega \ | \ u'_\epsilon(x) \in [z_1(x)+\delta,z_2(x)-\delta] \}|}{\epsilon} \\
       \nonumber
       &\leq \limsup_{\epsilon} \frac{1}{\epsilon L}\int_{\{u'_\epsilon(x) \in [z_1(x)+\delta,z_2(x)-\delta] \}} W(x,u'_\epsilon(x)) \, dx\\
       \nonumber
       &\leq \limsup_{\epsilon} \frac{F_\epsilon(u'_\epsilon)}{L} < \infty.   
        \end{align}
                
        Therefore, 
        \begin{align*}
            &\frac{1}{\epsilon}\int_{\Omega \times \Omega} J_\epsilon(x-x') |c_\epsilon(x') - c_\epsilon(x)|^2\, dx'dx\\
            & \leq \frac{2}{\epsilon}\int_{\Omega \times \{u'_\epsilon(x) \in [z_1(x)+\delta,z_2(x) -\delta]\}} J_\epsilon(x-x') |c_\epsilon(x') - c_\epsilon(x)|^2\, dx'dx \\
            &+ \frac{2}{\epsilon}\int_{\{u'_\epsilon(x) \leq z_1(x)+\delta\} \times \{u'_\epsilon(x') \geq z_2(x')-\delta\}} J_\epsilon(x-x') |c_\epsilon(x') - c_\epsilon(x)|^2\, dx'dx \\
            & \leq \frac{8M^2 \|J\|_{L^1}}{\epsilon}|\{u'_\epsilon(x) \in [z_1(x)+\delta,z_2(x) -\delta]\}| \\
            &+ \frac{1}{\epsilon 4\delta^2 }\int_{\Omega \times \Omega} J_\epsilon(x-x') |u'_\epsilon(x') - u'_\epsilon(x)|^2\, dx'dx,
        \end{align*}
        where we used \ref{H2} to obtain the last inequality. From \eqref{e:compactness1} we can thus infer that
        \begin{equation}
        \label{e:compactness2}
        \limsup_{\epsilon} \frac{1}{\epsilon}\int_{\Omega \times \Omega} J_\epsilon(x-x') |c_\epsilon(x') - c_\epsilon(x)|^2\, dx'dx < \infty.
        \end{equation}
Property \eqref{e:compactness2} and the fact that $J$ is a function which is not identically zero and belongs to $L^1( \mathbb{R}^m \EEE) \cap L^\infty( \mathbb{R}^m \EEE)$, allows us to argue as in Step 1 of the proof of \cite[Theorem 3.1]{alb-bel2} and to infer that $c_\epsilon \to c$ strongly in $L^1(\Omega)$, for a map $c \in BV(\Omega; \{-1,1\})$.

\textbf{Step 3}. Define $\tilde{u}_\epsilon \colon \Omega \to [-M,M]$ as 
 \[
 2 \tilde{u}_\epsilon(x):= (1+c_\epsilon(x))z_2(x) + (1-c_\epsilon(x))z_1(x).
 \]
  By Step 2, we obtain that $\tilde{u}_\epsilon \to u$ in $L^1(\Omega)$, where $2u= (1+c)z_2 +(1-c)z_1$, so that $u \in CP(\Omega; \mathbf{z}(x))$. To conclude the proof, it remains to show that $u_{\ep}-\tilde{u}_{\ep}\to 0$ strongly in $L^1(\Omega)$. Owing to Step 1, it actually suffices to prove that $u'_{\ep}-\tilde{u}_{\ep}\to 0$ strongly in $L^1(\Omega)$. To this aim, we observe that, since $\tilde{u}_{\ep}(x)\in \textbf{z}(x)$ for every $x\in \Omega$, for $0<\eta<\delta$ there holds
  \begin{align}
      &\nonumber\int_{\Omega}|u_{\ep}(x)-\tilde{u}_{\ep}(x)|dx\leq \int_{S^{\ep}_{\eta}}|u_{\ep}(x)-u'_{\ep}(x)|dx+\int_{S^{\ep}_{\eta}}|u'_{\ep}(x)-\tilde u_{\ep}(x)|dx+\int_{\Omega\setminus S^{\ep}_{\eta}}|u_{\ep}(x)-\tilde{u}_{\ep}(x)|dx\\
      &\label{eq:finaldec}\quad\leq \int_{\Omega}|u_{\ep}(x)-u'_{\ep}(x)|dx+2M|S^{\ep}_{\eta}|+\eta|\Omega|,
      \end{align}
      where $S^{\ep}_{\eta}$ is the set defined in \eqref{eq:def-Seta}, and where to estimate the latter term we have used that on $\Omega\setminus S^{\ep}_{\eta}$ there holds $\text{dist}\,(u_{\ep}(x);\mathbf{z}(x))<\eta$. Thus, since $\eta<\delta$, either $u_{\ep}(x)=u_{\ep}'(x)\in (z_1(x),z_2(x))$ and $|u_{\ep}(x)-\tilde{u}_{\ep}(x)|<\eta$, or $u_{\ep}'(x)\in \mathbf{z}(x)$ and $\tilde{u}_{\ep}(x)=u_{\ep}(x)$. In particular, by \eqref{eq:smallL1} and \eqref{eq:boundS}, we infer
      $$\limsup_{\ep\to 0}\int_{\Omega}|u_{\ep}(x)-\tilde{u}_{\ep}(x)|dx\leq C\eta.$$
      The thesis follows then owing to the arbitrariness of $\eta$.
\end{proof}
\subsection{Slicing preliminaries}
In this subsection we consider restrictions of our interaction kernels and double-well potentials to suitable $k$-dimensional subspaces, with $1 \leq k < m$. We collect here the main definitions and properties which will be instrumental for the proof of our Gamma-convergence result.

\begin{definition}[The class $\mathcal{L}_k^V(\sigma,\lambda,\rho)$]
\label{def:classJV}
Let $1 \leq k < m$ and let $V \subset \mathbb{R}^m$ be a $k$-dimensional subspace. Let also $J \colon V \to (0,\infty)$. \EEE Let $\sigma,\lambda,\rho \in (0,1)$. We say that $J$ belongs to the class $\mathcal{L}_k^V(\sigma,\lambda,\rho)$ if there exists $\tilde{J} \in L^1(V)$ such that
\begin{align}
    \label{e:classlm1-bis}
    &\frac{\lambda}{|h|^{k+\sigma}} \leq  J(h) \leq \frac{1}{\lambda |h|^{k+\sigma}}+\tilde{J}(h), \ \ \text{ for every $h \in (B_\rho(0) \setminus \{0\})\cap V $}\\
    \label{e:classm2-bis}
    & \int_{V \setminus B_{\rho}(0)} J(h)|h| \, dh < \infty.
\end{align}
    
\end{definition}

The following lemma provides a stability results under integration for the class introduced in Definition \ref{def:classJ}.

\begin{lemma}
\label{lemma:stability}
Let $1 \leq k < m$ and let $V \subset \mathbb{R}^m$ be a $(m-k)$-dimensional subspace. Let $J \in \mathcal{L}_m(\sigma,\lambda,\rho)$.  Then, there exists $\lambda'>0$ such that the restriction $J^V \colon V \to (0,\infty)$, defined as $J^V(x):= \int_{V^\bot} J(y+x) \, dy$ satisfies $J^V\in \mathcal{L}_{m-k}^{V}(\sigma,\lambda',\rho/2)$.
\end{lemma}
\EEE
\begin{proof}

Let us denote $B^k_\rho(0):= B_\rho(0) \cap V^\bot$ and $B^{m-k}_\rho(0):= B_\rho(0) \cap V$ for every $\rho >0$. For $x \in B^{m-k}_{\rho/2}(0) \setminus \{0\}$ we make use of \eqref{e:classlm1} to write
\[
\begin{split}
J^V(x) &\geq \int_{B^k_{\rho/2}(0)} \frac{\lambda}{|y+x|^{m+\sigma}} \, dy \\
&= |x|^k\int_{B^k_{\rho/2|x|}(0)} \frac{\lambda}{||x|y+x|^{m+\sigma}} \, dy \\
&= \frac{1}{|x|^{m-k+\sigma}} \int_{ B^k_{\rho/2|x|}(0)} \frac{\lambda}{|y+x/|x||^{m+\sigma}} \, dy 
\end{split}
\]
Notice that for every $x \in B^{m-k}_{\rho/2}(0) \setminus \{0\}$ we have
\[
\int_{ B^k_{\rho/2|x|}(0)} \frac{\lambda}{|y+x/|x||^{m+\sigma}} \, dy \geq \min\bigg\{1,\min_{\eta \in \mathbb{S}^{k-1}} \int_{ B^k_{1}(0)} \frac{\lambda}{|y+\eta|^{m+\sigma}} \, dy\bigg\} =: \tilde{\lambda} >0.
\]
This means that $J^V(x) \geq \tilde{\lambda}/|x|^{m-k+\sigma}$ for every $x \in B^{m-k}_{\rho/2}(0) \setminus \{0\}$. On the other hand, for every $x \in B^{m-k}_{\rho/2}(0) \setminus \{0\}$ we can make use again of \eqref{e:classlm1} to write
\[
\begin{split}
J^V(x) \leq \int_{V} &\bigg(\frac{1}{\lambda|y+x|^{m+\sigma}}\mathbbm{1}_{B^k_{\rho/2}(0)}(y) + J(x+y)\mathbbm{1}_{V \setminus B^k_{\rho/2}(0)}(y)+ \tilde{J}(y+x) \bigg)dy \\
&= \frac{1}{|x|^{m-k+\sigma}} \int_{ B^k_{\rho/2|x|}(0)} \frac{1}{\lambda|y+x/|x||^{m+\sigma}} \, dy + \tilde{J}^V(x) \\
& \ \ \ \ \ \ \leq  \frac{\tilde{\lambda'}}{|x|^{m-k+\sigma}} + \tilde{J}^V(x)
\end{split}
\]
where we have set  
\[
\begin{split}
\tilde{J}^V(x) &:= \int_{V \setminus B^k_{\rho/2}(0)} J(y+x) \, dy + \int_V \tilde{J}(y+x) \, dy \in L^1(V) \ \ \text{(by \eqref{e:classlm1})} \\
\tilde{\lambda'} &:= \max \bigg\{1, \max_{\eta \in \mathbb{S}^{k-1}}\int_{V} \frac{1}{\lambda|y+\eta|^{m+\sigma}}\, dy\bigg\} < \infty.
\end{split}
\]
 Therefore $J^V \in \mathcal{L}^{V^\bot}_{m-k}(\sigma,\lambda',\rho/2)$ with $\lambda':= \min\{\tilde{\lambda},\tilde{\lambda'}^{-1}\}$.

Eventually, one easily verifies condition \eqref{e:classm2-bis} since
\begin{align*}
\int_{V^\bot \setminus B_{\rho/2}(0)} J^V(x)|x| \, dx &=\int_{(V^\bot \setminus B_{\rho/2}(0)) \times V}  J(y+x)|x| \,dxdy \\
&\leq \int_{ \mathbb{R}^m \EEE \setminus B_{\rho/2}(0)} J(h)|h| < \infty,
\end{align*}
where we used both \eqref{e:classlm1} and \eqref{e:classm2} to infer the finiteness of the last integral in the above inequalities. 
\EEE
\end{proof}

A direct application of Fubini's Theorem yields that also the second class of kernels which will be relevant for our analysis is stable with respect to integration.
\begin{lemma}\label{lemma:stable-l1}
Let $1 \leq k < m$ and let $V \subset \mathbb{R}^m$ be a $(m-k)$-dimensional subspace. Let $J:\R^m\to (0,+\infty)$ satisfy \ref{K1} with $\|J\|_{L^1(\R^m)}=1$. Then, the restriction $J^V \colon V \to (0,\infty)$, defined as $J^V(x):= \int_{V^\bot} J(y+x) \, dy$ satisfies $\int_V J^V(x)\UUU(1+|x|)\EEE\,dx<+\infty$, with $\|J\|_{L^1(V)}=1$.
\end{lemma}
\begin{proof}
By Fubini's Theorem,
$$\int_V J^V(x)\UUU (1+|x|) \EEE\,dx=\int_V \int_{V^{\top}}J(y+x)\UUU (1+|x|) \EEE\,dy\,dx\leq \int_{\R^m}J(h)\UUU (1+|h|) \EEE\,dh<+\infty,$$
and
$$\int_V J^V(x)\,dx=\int_V \int_{V^{\top}}J(y+x)\,dy\,dx=\int_{\R^m}J(h)\,dh=1.$$
\end{proof}

Analogously to what we did for the interaction kernels, we also need to introduce a notion of restrictions of double-well potentials to suitable subspaces of $\mathbb R^m$.

\begin{definition}\label{def:W-restr}
Let $x_0 \in \mathbb{R}^m$ and $\xi \in  \mathbb{S}^{m-1}$. We define the $(m-1)$-dimensional affine plane $V_{\xi}$ as
\[
V_{\xi} := \{y +\xi \ | \ y \in \xi^\bot \}.
\]
We denote by $C_\rho(x_0)$ any $(m-1)$-dimensional cube oriented by $\xi$ of side-length $\rho$ and centered at $x_0$, namely, $C_\rho(x_0) := \{ x_0 +\sum_{i=1}^{m-1} \alpha_i \xi_i \ | \ -\rho/2 \leq \alpha_i \leq \rho/2, \ i=1,\dotsc,m-1\}$ where $\{\xi_1,\dotsc,\xi_{m-1}\}$ is any orthonormal basis of $\xi^\bot$. We further denote by $ Q_\rho(x_0) \subset \mathbb{R}^m$ the $m$-dimensional closed cube centered at $x_0$ satisfying $Q_\rho(x_0) \cap \xi^\bot = C_\rho(x_0)$. In order to ease the notation we will write  $Q_\rho:=Q_\rho(0)$. Notice that both in the definition of $C_\rho(x_0)$ and $Q_\rho(x_0)$ we drop the dependence from $\xi$ and $\{\xi_1,\dotsc,\xi_{m-1}\}$ since it will be always clear from the context.

For every $\rho >0$,  $x\in \mathbb{R}^m\EEE$,  and $\xi \in  \mathbb{S}^{m-1} \EEE$  we  define the potential $\mathcal{W}^{\rho\xi}_{x} \colon \xi^\bot \times \mathbb{R}\to [0,\infty)$ as 
\begin{equation}
\label{e:infpotential}
\mathcal{W}^{\rho\xi}_{x}(y,t):= \inf_{s \in [-\rho/2,\rho/2]} W(x+ y +s\xi,t).
\end{equation}
\end{definition}

 Define 
\begin{equation}
\label{eq:well-r}
\mathbf{z}_{i,x}(y,\rho):=\{z_i(x+ y+s\xi) \ | \ s \in [-\rho/2,\rho/2]\}\quad\text{for every } i=1,2.
\end{equation}
We notice that for $y \in \xi^\bot$ there holds $\mathcal{W}^{\rho\xi}_{x}(y,t)=0$ if and only if $t \in \bigcup_{i=1}^2 \mathbf{z}_{i,x}(y,\rho)$. Thanks to the continuity of the wells  $z_i$, cf \ref{H1},  we  find 
\[
\bigcap_{\rho >0} \mathbf{z}_{i,x}(y,\rho)= z_i(x+y), \ \ \text{ for $i=1,2$}, \text{ for every }y \in \xi^\bot \text{ and }x\in \mathbb{R}^m\EEE. \EEE
\]
In addition,  owing to the $\alpha$-H\"older regularity of the moving wells, cf. again \ref{H1}, there holds 

\begin{equation}
\label{eq:well-r-loc-un}
\mathbf{z}_{i,x}(y,\rho) \searrow z_i(x+y)\text{ as }\rho\to 0\quad\text{ locally uniformly in }y.
\end{equation}

As a consequence, for every compact set $K\subset \xi^{\bot}$ there exists $\rho^K>0$ small enough, so that $\mathbf{z}_{1,x}(y,\rho) \cap \mathbf{z}_{2,x}(y,\rho)= \emptyset$  for every $y\in K$ and for every $\rho<\rho^K$. 

Being $\mathbf{z}_{i,x}(y,\rho)$ closed sets, we can define  
\begin{equation}
\label{e:defaibi}
a_{x}(y,\rho) := \max \{t \in \mathbf{z}_{1,x}(y,\rho)\} < b_{x}(y,\rho) := \min \{t \in \mathbf{z}_{2,x}(y,\rho)\},
\end{equation}
and assume with no loss of generality that,  for every sufficiently small $\rho >0$, depending only on the set $C_1(0)$, there holds 
\[
b_{x}(y,\rho) - a_{x}(y,\rho) \geq  4\delta_{C_1(0)}, \ \ \text{ for every $y \in C_1(0)$},
\]
where $\delta_{C_1(0)}$ is defined as in \ref{H2}--\ref{H5}.
 
 In particular,  for every sufficiently small $\rho >0$ depending only on the set $C_1(0)$   
 \begin{equation}
 \mathcal{W}^{\rho\xi}_{x}(y,t)>0, \ \ \text{for every $y \in C_1(0)$ and $t \in (a_{x}(y,\rho),b_{x}(y,\rho))$}.
 \end{equation}

 The results in the next sections will rely on showing that the potentials $\mathcal{W}^{\rho\xi}_{x}$ still fulfills assumptions (1)--(7) of Proposition \ref{p:contdep} (compare with Lemma \ref{lem:H3-trans}).

    \begin{lemma}
    \label{lem:trans-rest}
         Assume that $W$ satisfies \ref{H1}--\ref{H5}. Let $\mathcal{W}^{\rho\xi}_{x}$ be as in Definition \ref{def:W-restr}. Then, there exists functions $\phi_{x,\rho} \colon C_1(0) \times \mathbb{R} \to \mathbb{R}$ such that, for every $y \in C_1(0)$ and for every sufficiently small $\rho >0$, there holds
    \begin{enumerate}[label=(Y.\arabic*)]
     \vspace{1mm}
        \item \label{Y1} $\phi_{x,\rho}(y,\cdot ) \colon \mathbb{R} \to \mathbb{R}$ and
          \begin{align*}
              &\phi_{x,\rho}(y,\cdot ) \colon [a_{x}(0,\rho),b_{x}(0,\rho)] \to [a_{x}(y,\rho),b_{x}(y,\rho)] \\
              \phi_{x,\rho}(&y,a_{x}(0,\rho)) =a_{x}(y,\rho) \ \ \text{ and } \ \ \phi_{x,\rho}(y,b_{x}(0,\rho)) =b_{x}(y,\rho).
          \end{align*}
        \item \label{Y2}
        $\sup_{\rho>0 , y \in C_1(0)} \lfloor\phi_{x,\rho}(y,\cdot)\rfloor_{C^{0,\alpha}([a_{x}(0,\rho),b_{x}(0,\rho)])} < \infty$.
         \vspace{1mm}
         \item \label{Y3}       
        $\phi_{x,\rho}(y,t) \to t, \text{ in }L^{\infty}_{}(C_1(0) \times \mathbb{R})$.
         \vspace{1mm}
        \item \label{Y4} There exist $0 < c_1 \leq c_2 < \infty$ such that, for every $t \in [a_{x}(y,\rho)+\delta_{C_1(0)},b_{x}(y,\rho)-\delta_{C_1(0)}]$, we have
        \[
        c_1\mathcal{W}^{\rho\xi}_{x}(y, \phi_{x,\rho}(y,t)) \leq \mathcal{W}^{\rho\xi}_{x}(0,t) \leq c_2 \mathcal{W}^{\rho\xi}_{x}(y, \phi_{x,\rho}(y,t)).
        \]
    \end{enumerate}
     \end{lemma}
     \color{black}
\begin{proof}
    Since from Remark \ref{r:convimpl} we know that $f$ is increasing, for every $(x_0,y,t) \in \mathbb{R}^m\EEE \times  C_1(0)^{\xi} \times \mathbb{R}$,  by \ref{H3} there holds \EEE
\begin{align*}
     \bar\delta^{\xi} f(\inf_{s \in [-\rho/2,\rho/2]}\text{dist}(t,\mathbf{z}(x+y+s\xi))) &= \bar\delta^{\xi} \inf_{s \in [-\rho/2,\rho/2]} f(\text{dist}(t,\mathbf{z}(x+y+s\xi))) \\
    &\leq \mathcal{W}^{\rho\xi}_{x}(y,t)\\
    &\leq \frac{1}{\bar\delta^{\xi}}\inf_{s \in [-\rho/2,\rho/2]} f(\text{dist}(t,\mathbf{z}(x+y+s\xi)))) \\
    &= \frac{1}{\bar\delta^{\xi}}f(\inf_{s \in [-\rho/2,\rho/2]}\text{dist}(t,\mathbf{z}(x+y+s\xi))).
\end{align*}
     We note \EEE that for every $\rho>0$ we have the identity
    \[
    \inf_{s \in [-\rho/2,\rho/2]}\text{dist}(t,\mathbf{z}(x+y+s\xi))= d_{x,\rho}(t,y),\]
    where,  recalling \eqref{eq:well-r},\EEE
    \[
    d_{x,\rho}(t,y) := 
    \begin{cases}
       0, &\text{ if $t \in \bigcup_{i=1}^2 \mathbf{z}_{i,x}(y,\rho)$} \\
       \min\{|t-a_{x}(y,\rho)|, |t-b_{x}(y,\rho)|\} &\text{ if $t \in (a_{x}(y,\rho),b_{x}(y,\rho))$},\\
     |t-a_{x}(y,\rho)| &\text{ if $t \in (-\infty,a_{x}(y,\rho)$}\EEE \\
     |t-b_{x}(y,\rho)| &\text{ if $t \in (b_{x}(y,\rho),\infty)$}.\EEE
    \end{cases}
    \]
  In particular, for $(x,y,t) \in \mathbb{R}^m\EEE \times C_1(0) \times \mathbb{R}$ we  infer \EEE
    \begin{equation}
    \label{e:infdis}
    \delta f(d_{x,\rho}(t,y)) \leq \mathcal{W}^{\rho\xi}_{x}(y,t) \leq \frac{1}{\delta} f(d_{x,\rho}(t,y)).
    \end{equation}
    For every $y\in \xi^{\bot}$, we now define  $\phi_{x,\rho}(y,\cdot) \colon [a_{x}(0,\rho),b_{x}(0,\rho)] \to [a_{x}(y,\rho),b_{x}(y,\rho)]$ as
   \begin{align*}
       \phi_{x,\rho}(y,t):=\left(\frac{t-a_x(0,\rho)}{b_x(0,\rho)-a_x(0,\rho)}\right)a_x(y,\rho)+\left(1-\frac{t-a_x(0,\rho)}{b_x(0,\rho)-a_x(0,\rho)}\right)b_x(y,\rho).
   \end{align*}
   Then, \ref{Y1} and \ref{Y2} follow by definition. Property \ref{Y3} is a consequence of the fact that $a_x(y,\rho)\to z_1(y,\rho)$ and $b_x(y,\rho)\to z_2(y,\rho)$ uniformly in $C_1(0)$ as $\rho\to 0$ (see \eqref{eq:well-r-loc-un} and \eqref{e:defaibi}). Eventually, property \ref{Y4} follows by combining \eqref{e:infdis} and \ref{H3}.
\end{proof}

\section{$\Gamma$-limsup upper-bound}
\label{s:sec6}
This section is devoted to the construction of a recovery sequence. We start with a technical lemma, whose proof is postponed to the end of this subsection. Recall Definitions \ref{def:surf} and \ref{def:classJ}, as well as the functionals in \eqref{e:functionalxi148}.

\begin{lemma}
    \label{l:decaykernel}
    Let $m\in \mathbb{N}$, and let $A \subset \mathbb{R}^m$ be a bounded measurable set. Assuming that $J \in \mathcal{L}_m(\eta,\lambda,\rho)$ for some $\eta,\lambda \in (0,1)$ and some $\rho >0$, there holds
    \begin{equation}
        \label{e:decaykernel1}
         \lim_{\epsilon \to 0^+} \epsilon^{m} \int_{\epsilon^{-1}(A \times A)} J(x-x') \epsilon^{\beta} |x-x'|^{1+\beta} \, dx'dx=0,
    \end{equation}
    for every $\beta >0$.
\end{lemma}

The following definition of polyhedral sets and polyhedral functions was provided in \cite{alb-bel2}.

\begin{definition}[Polyhedral set and polyhedral function]
A polyhedral set in $ \mathbb{R}^m \EEE$ is an open set $E$ whose boundary is a Lipschitz manifold contained in the union of finitely many affine hyperplanes; the faces of $E$ are the intersections of the boundary of $E$ with each one of these hyperplanes, and an edge point of $E$ is a point which belongs to at least two different faces (that is, a point where $\partial E$ is not smooth). We denote by $\nu_E$ the inner normal to $\partial E$ (defined for all points in the boundary which are not edge points).

A polyhedral set in $\Omega$ is the intersection of a polyhedral set in $ \mathbb{R}^m \EEE$ with $\Omega$. We say that $u \in CP(\Omega;\mathbf{z}(x))$ is polyhedral if there exists a polyhedral set $A$ in $ \mathbb{R}^m \EEE$ with $\partial A$ transversal to $\partial \Omega$ and such that $u=z_1$ on $A \cap \Omega$ while $u=z_2$ on $\Omega \setminus A$.
\end{definition}

In order to construct recovery sequences for general functions in $CP(\Omega;\mathbf{z}(x))$ we will proceed as follows: we will first establish such a result for polyhedral functions and then argue by approximation to extend the construction to the whole class. An intermediate step in this direction, following the approach in \cite{alb-bel2}, will be to show that $\Omega$ satisfies the additional property below.

\begin{definition}[Condition (P)]
    Let $A$ be a polyhedral set in $\Omega$. We say that $A$ satisfies condition (P) if for every polyhedral function $u$ in $CP(\Omega;\mathbf{z}(x))$ there exists a sequence of measurable functions $(u_\epsilon)$ defined on $\overline{A}$, such that $z_1(x) \leq u_\epsilon(x) \leq z_2(x)$ for every $x \in \Omega$, and
        \begin{align}
        \label{e:limsuppol1}
            &u_\epsilon \to u \text{ uniformly on every compact set } K \subset \overline{A} \setminus S_u \\
            \label{e:limsuppol2}
            &\limsup_{\epsilon \to 0^+} F_\epsilon(u_\epsilon; A) \leq \int_{A \cap S_u} \sigma(y,\xi) \, d\mathcal{H}^{n-1}(y)
        \end{align}
\end{definition}

We are now in a position to construct a recovery sequence for polyhedral functions. \UUU In the proposition, we leverage the core result of Hölder continuous dependence in Theorem \ref{p:lipselx} for gluing together optimal profiles defined on parallel lines that intersect orthogonally with the faces of the polyhedral set. \EEE

\begin{proposition}[Recovery sequence for polyhedral functions.]
\label{t:recpol}
    Let $\Omega \subset  \mathbb{R}^m \EEE$ be open, let $J \colon  \mathbb{R}^m \EEE \to (0,\infty)$ be an even function, and let $W \colon  \mathbb{R}^m \EEE \times \mathbb{R} \to [0,\infty)$ satisfy \ref{H1}--\ref{H4}. Let further $\alpha \in (1/2,1]$ be given by \ref{H1}. Assume that either one of the following conditions is satisfied:
    \begin{enumerate}[(i)]
       \item  $Q_x(t):= \partial_t W(x,t)+t$ has a continuous inverse for every $x \in \Omega$ and $t \in [z_1(x),z_2(x)]$, and $J$ satisfies \ref{K1} with $\|J\|_{L^1( \mathbb{R}^m \EEE)}=1$
       \item $J \in \mathcal{L}_m(\eta,\lambda,\rho)$ for some $\eta \in (0,\alpha^2)$ and $\lambda,\rho \in (0,1)$.
   \end{enumerate}
     
     Then, for every polyhedral function $u\in CP(\Omega;\mathbf{z}(x))$ there exists a sequence $(u_\epsilon) \subset L^1(\Omega)$ such that $u_\epsilon \to u$ in $L^1(\Omega)$ and
    \begin{equation}
        \label{e:gli1111}
        \limsup_{\epsilon \to 0^+} F_{\epsilon}(u_\epsilon;\Omega) \leq F_0(u).
    \end{equation}
\end{proposition}
\begin{proof}
We prove the theorem assuming (ii). The proof under condition (i) follows exactlz in the same way.
With no loss of generality we assume that $F_0(u)< \infty$. We will achieve the desired statement by showing that $\Omega$ satisfies (P). To this purpose consider the following three statements:
\begin{enumerate}[(a)]
    \item If $A$ is a polyhedral set in $\Omega$ with $\mathcal{H}^{n-1}(\overline{A} \cap S_u)=0$, then $A$ satisfies (P).
    \item Let $\Sigma$ be a face of $S_u$ and let $\pi$ be the projection map on the affine hyperplane which contains $\Sigma$: if $A$ is a polyhedral set in $\Omega$ such that $S_u\cap A= \Sigma$ and $\pi(A) = \Sigma$, then $A$ satisfies (P).
    \item If $A_1$ and $A_2$ satisfy (P) and are disjoint, then the interior of $\overline{A}_1 \cup \overline{A}_2$ satisfies (P).
\end{enumerate}

Since $\Omega$ can be expressed as $\Omega = \bigcup_\ell \overline{A}_\ell$ for finitely many polyhedral sets $A_\ell$ in $\Omega$ satisfying either (a) or (b), proving the validity of (a),(b), and (c) will immediately allow us to infer that $\Omega$ satisfies property (P). Statement (a) and (c) follow by applying the decay estimate of the defect (due to the regularity of the wells, condition \eqref{e:contvdef} is satisfied for $z_1$ and $z_2$, and, in particular, Proposition \ref{r:defuniform} applies) exactly as in \cite[Theorem 5.2]{alb-bel2}. Therefore, to conclude the proof of the proposition it only remains to prove (b). For convenience of the reader, we subdivide the remaining part of the proof into five Steps.\\

\noindent\textbf{Step 1: Recalling the H\"older continuous dependence of optimal profiles}. Since $u \in CP(\Omega;\mathbf{z}(x))$, we have $u^\pm(y) \in \{z_1(y),z_2(y)\}$ for every $y \in \Sigma$. {\color{black} Let $\xi\in \mathbb{S}^{m-1}$ be orthogonal to $\Sigma$, and} denote by $V$ the affine hyperplane containing $\Sigma$. \UUU Furthermore, in order to simplify the notation, we may assume with no loss of generality that $0 \in V$, namely, $V = \xi^\bot$ \EEE. 

By Theorem \ref{p:lipselx} with $K:= V \cap \overline{\Omega}$ and $J=J^\xi$ (Recall \eqref{e:kerneljxi1}, \eqref{e:functionalxi148}, as well as Lemmas \ref{lemma:stability} and \ref{lemma:stable-l1}), we find a family $(\gamma_y)_{y \in \Sigma}$ made of increasing functions in $ X_{z_1(y)}^{z_2(y)}$ such that for every $R >0$ 
    \begin{align}
        \label{e:limsuppol3}
        F_{y}^\xi(\gamma_y) &= \sigma(y,\xi) \\
       \label{e:limsuppol3.1.1.1}
        \int_{-R}^R |\gamma_y(t) -\gamma_{y'}(t)|^2 \, dt &\leq L \,  \lfloor \partial_t W \rfloor_{C^{0,\alpha}(K\times[-M,M])} |y-y'|^{\alpha^2},
    \end{align}
    where $L:=L(\delta_K,K,\|\sigma_{J^\xi}\|_{L^1},\|\overline{\sigma}_{J^\xi}\|_{L^1},\lfloor z_1\rfloor_{C^{0,\alpha}},\lfloor z_2\rfloor_{C^{0,\alpha}})>0$ (see \ref{H4} for the role of the parameter $\delta_K$, as well as \eqref{e:lipdep1.428} and \eqref{e:defoversigma} for the definitions of $\sigma_{J^\xi}$ and $\overline{\sigma}_{J^\xi}$, respectively).\\ 
    
    \noindent\textbf{Step 2: Gluing together optimal profiles in the direction orthogonal to $\Sigma$}. We show here that there exists $R':= R'(\omega,K) >0$ such that
    \begin{align}
        \label{e:cenfam1.1.1}
        z_2(y)  \leq  a_\omega \gamma_y(t) + b_\omega(y), \ \ \text{ for every $t \geq R'$ and every $y \in \Sigma$}\\
        \label{e:cenfam1.1.2}
       a_\omega \gamma_y(t) + b_\omega(y) \leq z_1(y) , \ \ \text{ for every $t \leq R'$ and every $y \in \Sigma$},
    \end{align}
    where $a_\epsilon:= (1-\omega/2)^{-1} >1$ and $b_\omega(y):= \frac{z_1(y)+z_2(y)}{2}(1-a_\omega)$. To this aim, let $\tilde{\gamma}_y:= p(y,\gamma_y) \in X_{-1}^1$, where $p(y,t):=  \frac{2}{z_2(y)-z_1(y)}\big(t -\frac{z_1(y)+z_2(y)}{2}\big)$ for every $y\in \Omega$ and $t\in \mathbb{R}$ (see the beginning of the proof of Theorem \ref{p:lipselx}). We claim that the family $(\tilde{\gamma}_y)$ is centered on $K$ according to Definition \ref{d:cenfam}.

    Indeed, assume by contradiction that there exists $\omega > 0$ and two sequences $R'_k \nearrow \infty$ and $y_k \to y_0 \in K$, such that for instance $\tilde{\gamma}_{y_k} \leq 1-\frac{\omega}{2}$ for some $t_k \geq R'_k$ for every $k\in \mathbb{N}$. Since each $\tilde{\gamma}_{y_k}$ is increasing and $-1 \leq \gamma_{y_k} \leq 1$, we infer that any $L^1_{loc}$-limit of $\tilde{\gamma}_{y_k}$ must be bounded away from $1$ for every positive $t$. Nevertheless, we know from \eqref{e:limsuppol3.1.1.1} and the Lipschitzianity of $p$ that the $L^1_{loc}$-limit of the sequence $\tilde{\gamma}_{y_k}$ is exactly $\tilde{\gamma}_{y_0}$. This is in contradiction with the fact that $\tilde{\gamma}_{y_0} \in X_{-1}^1$. Analogously, one can prove that the assumption $\tilde{\gamma}_{y_k} \geq -1+\frac{\omega}{2}$ for some $t_k \leq -R'_k$ leads to a contradiction. Therefore the claim is proved.
    
    The property of being centered immediately implies that for every $\omega > 0$ we find $R':= R'(\omega,K) >0$ such that
    \begin{align}
        \label{e:beingcent1}
        &1 \leq a_\epsilon \tilde{\gamma}_y(t), \ \ \text{ for every $t \geq R'$ and every $y \in K$} \\
        \label{e:beingcent2}
        & a_\epsilon \tilde{\gamma}_y(t) \leq -1, \ \ \text{ for every $t \leq -R'$ and every $y \in K$}
    \end{align}
    Consider now the map $q(y,t):= \frac{z_2(y)-z_1(y)}{2}t + \frac{z_1(x)+z_2(x)}{2}$ for every $y\in \Omega$ and $t\in \mathbb{R}$ (see the beginning of the proof of Theorem \ref{p:lipselx}). Properties \eqref{e:cenfam1.1.1}--\eqref{e:cenfam1.1.2} follow then by applying $q(y,\cdot)$ to \eqref{e:beingcent1}--\eqref{e:beingcent2}. 
    
Consider the maps $\psi_\omega(y,t):= a_\omega t + b_\omega(y)$ and $\varphi_\omega(y,t):= \frac{t-b_\omega(y)}{a_\omega}$. Then, for every $y \in \Sigma$ we have
 \begin{align}
 \label{e:z1z2psi}
 [z_1(y),z_2(y)] &\subset \psi_\omega(y,[z_1(y),z_2(y)]) \\
 \label{e:z1z2phi}
 \varphi_\omega(y,[z_1(y),z_2(y)]) &\subset [z_1(y),z_2(y)] \\
 \label{e:z1z2inv}
 \psi_\omega(y,\varphi_\omega(y,t))&=t\\
 \label{e:z1z2sup}
 \psi_\omega(y,t) &\geq t, \ \ \text{ for $t \in [(z_1(y)+z_2(y))/2,z_2(y)]$} \\
  \label{e:z1z2sub}
  \psi_\omega(y,t) &\leq t, \ \ \text{ for $t \in [z_1(y),(z_1(y)+z_2(y))/2]$}.
 \end{align}
 We claim that for every $\omega >0$ sufficiently small there holds
\begin{equation}
\label{e:claimlimsup}
    W(y,\psi_\omega(y,t)) \leq c\, W(y,t), \ \ \text{ for every $t \in \varphi_\omega(y,[z_1(y),z_2(y)])$ },
\end{equation}
for some constant $c:=c(K)\geq1$. To prove the claim, we notice that, thanks to \eqref{e:z1z2phi}, as well as to the uniform convergenes $a_\omega \searrow 1$, and $b_\omega \to 0$ in $K$, we have $ [\varphi_\omega(y,z_1(y)),\varphi_\omega(y,z_1(y) + \frac{\delta}{2})] \subset [z_1(y),z_1(y)+\delta]$ for every $\omega >0$ small enough, independently of $y \in K$. By using also \eqref{e:z1z2inv} and \eqref{e:z1z2sub}, assumption \ref{H4} entails that for every $t \in [\varphi_\omega(y,z_1(y)),\varphi_\omega(y,z_1(y) + \frac{\delta}{2})]$ (notice that $z_1(y)+\delta \leq (z_2(y)+z_1(y))/2$ thanks to condition \ref{H2})
\[
W(y,t)-W(y,\psi_\omega(y,t)) = \int^t_{\psi_\omega(y,t)} \partial_t W(y,t) \, dt \geq 0.
\]
In the very same way, the uniform convergence of $a_\omega$ and $b_\omega$ in $K$, as well as properties \eqref{e:z1z2phi}, \eqref{e:z1z2sup}, and \ref{H4} imply that for every $t \in [\varphi_\omega(y,z_2(y) - \frac{\delta}{2}),\varphi_\omega(y,z_2(y))]$
\[
W(y,\psi_\omega(y,t)) -W(y,t) = \int^{\psi_\omega(y,t)}_t \partial_t W(y,t) \, dt \leq 0.
\]
Summarizing, for every $y \in K$ and $t \in [\varphi_\omega(y,z_1(y)),\varphi_\omega(y,z_1(y) + \frac{\delta}{2})] \cup [\varphi_\omega(y,z_2(y) - \frac{\delta}{2}),\varphi_\omega(y,z_2(y))]$, for every $\omega >0$ small enough, we have
\begin{equation}
    \label{e:relw1}
    W(y,t)\geq W(y,\psi_\omega(t)).
\end{equation}
Since both $\varphi_{\omega}(y,t) \to t$ and $\psi_\omega(y,t) \to t $ locally uniformly in $y$ and $t$ as $\omega \to 0^+$, by \ref{H1} we infer the existence of a constant $c:=c(K)\geq 1$ such that for every $y \in K$, $t \in [\varphi_\omega(y,z_1(y) + \frac{\delta}{2}),\varphi_\omega(y,z_2(y) - \frac{\delta}{2})]$, and for every $\omega >0$ small enough, there holds
\begin{equation}
    \label{e:relw2}
    c \,W(y,t)\geq W(y,\psi_\omega(t)).
\end{equation}
By putting together \eqref{e:relw1} and \eqref{e:relw2} we obtain \eqref{e:claimlimsup}.

We eventually define for every $y \in K$ and every $0 < \omega < \frac{1}{2}$ the function $\sigma_y \colon \mathbb{R} \to [z_1(y),z_2(y)]$ as
\[
(\sigma_\omega)_{y}(t):= [(a_\omega \gamma_y(t) + b_\omega(y)) \wedge z_1(y)] \vee z_2(y),
  \]  
  and we notice that for every $y \in \Sigma$ we have from \eqref{e:cenfam1.1.1}--\eqref{e:cenfam1.1.2} (see also Remark \ref{r:bormeas})
  \begin{align}
  \label{e:cenfam12345}
  &(\sigma_\omega)_y(t) = z_1(y), \ \ \text{ for $t \leq -R'$} \ \ \text{ and } \ \  (\sigma_\omega)_y(t)=z_2(y) \ \ \text{ for $t \geq R'$} \\
  \label{e:cenfam123456}
  & \ \ (\dot{\sigma}_\omega)^{-1}_y \otimes \mathcal{H}^{n-1} \restr \Sigma, \ \ \text{ is a finite Radon measure on $\mathbb{R} \times \Sigma$.}
 \end{align}

  \noindent\textbf{Step 3: Energy estimate for the non-local contribution}.
  For $\rho\in (0,1)$, let now $\varphi_\rho \colon  \mathbb{R}^m \EEE \times  \mathbb{R}^m \EEE \times \mathbb{R} \to \mathbb{R}$ be the transition maps provided by Lemma \ref{lem:H3-trans}, and satisfying \ref{X1}--\ref{X4} (To ease the notation, we have omitted the apex with the dependence on $K$). Without loss of generality we assume that $A^\xi_{y} := \{t \in \mathbb{R} \ | \ y+t\xi \in A \}$ is such that $A^\xi_{y} \subset (-\frac{1}{2},\frac{1}{2})$ for every $y \in \Sigma$. {\color{black}Recalling that every $x\in A$ can be uniquely written as $x=y+t\xi$, for $y\in \Sigma$ and $t\in A^\xi_{y}$, with a slight abuse of notation we}
  define $u_{\epsilon,\omega} \colon \Sigma \times (-1,1) \to \mathbb{R}$ as 
    \begin{equation}
        u_{\epsilon,\omega}(y+t\xi):=
        \varphi_1(y,t\xi,(\sigma_\omega)_{y}(\epsilon^{-1}t)).
    \end{equation}
    We further set $J^{\xi^\bot}(y):= \int_{\mathbb{R}} J(y+t\xi) \, dt$ for every $y \in \xi^\bot \setminus \{0\}$.\\
  
  By \ref{X2} and Young's inequality,  for every $d>1$ we infer
\begin{align*}
    &\frac{1}{\epsilon}\int_{A \times A} J_\epsilon(x'-x) |u_{\epsilon,\omega}(x')-u_{\epsilon,\omega}(x)|^2 \, dx'dx\\
    \nonumber
    &\quad\leq \frac{d}{(d-1)\epsilon}\int_{\Sigma \times \Sigma} \bigg(\int_{{\color{black}(-\frac{1}{2},\frac{1}{2})^2}} J_\epsilon(y'-y +(t' -t)\xi) |u_{\epsilon,\omega}^{y'\xi}(t')-u_{\epsilon,\omega}^{y\xi}(t')|^2 \,dt'dt\bigg) dy'dy\\ 
    \nonumber
    &\qquad+ \frac{d}{\epsilon}\int_{\Sigma \times \Sigma} \bigg(\int_{{\color{black}(-\frac{1}{2},\frac{1}{2})^2}} J_\epsilon(y'-y +(t' -t)\xi) |u_{\epsilon,\omega}^{y\xi}(t')-u_{\epsilon,\omega}^{y\xi}(t)|^2 \,dt'dt\bigg) dy'dy.
\end{align*}
Now we treat separately the last two integrals in the above inequality. In the first integral, we simply consider the change of variables given by the mupltiplication by $\ep$ in all the variables $(y,y',t,t')$ (the jacobian is exactly $\ep^{2n}$) to obtain
\begin{align*}
   \frac{1}{\ep} &\int_{\Sigma \times \Sigma} \bigg(\int_{{\color{black}(-\frac{1}{2},\frac{1}{2})^2}} J_\epsilon(y'-y +(t' -t)\xi) |u_{\epsilon,\omega}^{y'\xi}(t')-u_{\epsilon,\omega}^{y\xi}(t')|^2 \,dt'dt\bigg) dy'dy \\
   &= \frac{\epsilon^{2n}}{\epsilon^{n+1}}\int_{\ep^{-1}(\Sigma \times \Sigma)} \bigg(\int_{{\color{black}(-\frac{1}{2\ep},\frac{1}{2\ep})^2}} J(y'-y +(t' -t)\xi) |u_{\epsilon,\omega}^{\ep y'\xi}(\ep t')-u_{\epsilon,\omega}^{\ep y\xi}(\ep t')|^2 \,dt'dt\bigg) dy'dy \\
   & \leq \epsilon^{n-1}\int_{\ep^{-1}(\Sigma \times \Sigma)} \bigg(\int_{{\color{black}(-\frac{1}{2\ep},\frac{1}{2\ep})}} \bigg( \int_{\mathbb{R}} J(y'-y +(t' -t)\xi) \, dt\bigg) |u_{\epsilon,\omega}^{\ep y'\xi}(\ep t')-u_{\epsilon,\omega}^{\ep y\xi}(\ep t')|^2 \,dt'\bigg) dy'dy \\
   & = \epsilon^{n-1}\int_{\ep^{-1}(\Sigma \times \Sigma)} \bigg(\int_{{\color{black}(-\frac{1}{2\ep},\frac{1}{2\ep})}} \bigg( \int_{\mathbb{R}} J(y'-y +t\xi) \, dt\bigg) |u_{\epsilon,\omega}^{\ep y'\xi}(\ep t')-u_{\epsilon,\omega}^{\ep y\xi}(\ep t')|^2 \,dt'\bigg) dy'dy \\
   & = \epsilon^{n-1}\int_{\ep^{-1}(\Sigma \times \Sigma)} \bigg(\int_{{\color{black}(-\frac{1}{2\ep},\frac{1}{2\ep})}} J^{\xi^\bot}(y-y')|u_{\epsilon,\omega}^{\ep y'\xi}(\ep t')-u_{\epsilon,\omega}^{\ep y\xi}(\ep t')|^2 \,dt'\bigg) dy'dy.
\end{align*}
For the second integral, we first notice that 
\begin{align*}
    \int_{\Sigma} J_{\epsilon}(y'-y + (t'-t)\xi) \, dy' &= \frac{1}{\ep^n} \int_{\Sigma} J\bigg(\frac{y'-y}{\ep} + \frac{t'-t}{\ep}\xi  \bigg) \, dy'\\
    &= \frac{1}{\ep}  \int_{\ep^{-1}\Sigma} J\bigg(y'-\frac{y}{\ep} + \frac{t'-t}{\ep}\xi  \bigg) \, dy'\\
    &\leq \frac{1}{\ep}  \int_{V} J\bigg(y'-\frac{y}{\ep} + \frac{t'-t}{\ep}\xi  \bigg) \, dy' \\
    &= \frac{1}{\ep}  \int_{V} J\bigg(y'+ \frac{t'-t}{\ep}\xi  \bigg) \, dy' \\
    &= \frac{1}{\ep} J^\xi_\ep\bigg(\frac{t'-t}{\ep}\bigg) = J^\xi_\ep(t'-t).
    \end{align*}
Therefore, we can write
\begin{align*}
   \frac{1}{\epsilon}&\int_{\Sigma \times \Sigma} \bigg(\int_{{\color{black}(-\frac{1}{2},\frac{1}{2})^2}} J_\epsilon(y'-y +(t' -t)\xi) |u_{\epsilon,\omega}^{y\xi}(t')-u_{\epsilon,\omega}^{y\xi}(t)|^2 \,dt'dt\bigg) dy'dy \\ 
   &\leq \frac{1}{\epsilon}\int_{\Sigma} \bigg(\int_{{\color{black}(-\frac{1}{2},\frac{1}{2})^2}} J^\xi_\epsilon(t' -t) |u_{\epsilon,\omega}^{y\xi}(t')-u_{\epsilon,\omega}^{y\xi}(t)|^2 \,dt'dt\bigg)dy \\ 
   &= \int_{\Sigma} \bigg(\int_{{\color{black}(-\frac{1}{2\ep},\frac{1}{2\ep})^2}} J^\xi(t' -t) |u_{\epsilon,\omega}^{y\xi}(\ep t')-u_{\epsilon,\omega}^{y\xi}(\ep t)|^2 \,dt'dt\bigg)dy,
\end{align*}
where in the last equality we considered the change of variables given by the multiplication in the variables $(t,t')$ by $\ep$ (the jacobian is exactly $\ep^2$). 

Summarizing we have obatined
\begin{align}
\label{e:limsupol10}
    &\frac{1}{\epsilon}\int_{A \times A} J_\epsilon(x'-x) |u_{\epsilon,\omega}(x')-u_{\epsilon,\omega}(x)|^2 \, dx'dx\\
    \nonumber
    &\quad\leq \frac{d\epsilon^{n-1}}{d-1} \int_{\epsilon^{-1}(\Sigma \times \Sigma) } \bigg(\int_{(-\frac{1}{2\epsilon},\frac{1}{2\epsilon})}  J^{\xi^\bot}(y'-y) |u_{\epsilon,\omega}^{\epsilon y'\xi}( \epsilon t')-u_{\epsilon,\omega}^{\epsilon y\xi}(\epsilon t')|^2  \, dt'\bigg) dy'dy\\ 
    \nonumber
    & \qquad + d\int_{\Sigma} \bigg(\int_{(-\frac{1}{2\epsilon},\frac{1}{2\epsilon})^2 } J^\xi(t' -t) |u_{\epsilon,\omega}^{y\xi}(\epsilon t')-u_{\epsilon,\omega}^{y\xi}(\epsilon t)|^2 \,dt'dt\bigg) dy =: \frac{d\epsilon^{n-1}}{d-1}I_\epsilon  + d I_\epsilon'.
\end{align}

 By taking into account \ref{X2} we continue the above inequalities as follows
\begin{align*}
    I_\epsilon
    &\leq c \int_{\epsilon^{-1}(\Sigma \times \Sigma) }  J^{\xi^\bot}(y'-y) \epsilon^{1 +\alpha} |y'-y|^{1+\alpha}  \,  dy'dy \\ 
    &+ c  \int_{\epsilon^{-1}(\Sigma \times \Sigma) }  \bigg(\int_{(-\frac{1}{2\epsilon},\frac{1}{2\epsilon})} J^{\xi^\bot}(y'-y)|\sigma_{\epsilon y'}(t') -\sigma_{\epsilon y}(t')|^2  \, dt'\bigg) dy'dy \\
    &=: cI_\epsilon^1 +c I^2_\epsilon,
\end{align*}
for some constant $c >0$ depending only on $A$, $\sup_{x \in A} |z_1(x)|$, and $\sup_{x \in A} |z_2(x)|$. Since $J$ belongs to the class $\mathcal{L}_{n}(\eta,\lambda,\rho)$, the stability property proved in Lemma \ref{lemma:stability} tells us that $J^{\xi^\bot} \in \mathcal{L}_{n-1}^{\xi^\bot}(\eta,\lambda',\rho/2)$ for some $\lambda' \in (0,1)$. Hence, by applying Lemma \ref{l:decaykernel} we find
\begin{equation}
\label{e:limsup369}
\lim_{\epsilon \to 0^+}\epsilon^{n-1} I^1_\epsilon = 0.
\end{equation}
For the term $I^{2}_\epsilon$, we use \eqref{e:cenfam12345} to infer
\begin{align*}
I^{2}_\epsilon &\leq  \int_{\epsilon^{-1}(\Sigma \times \Sigma) } J^{\xi^\bot}(y'-y)  \bigg(\int_{(-\frac{1}{2\epsilon},\frac{1}{2\epsilon})} |\sigma_{\epsilon y'}(t')-\sigma_{\epsilon y}(t')|^2  \, dt'\bigg) dy'dy  \\
&\leq  a_\omega^2\int_{\epsilon^{-1}(\Sigma \times \Sigma) } J^{\xi^\bot}(y'-y)  \bigg(\int_{(-R',R')} |\gamma_{\epsilon y'}(t') + a_\omega^{-1}b_\omega(\epsilon y')-\gamma_{\epsilon y}(t') - a_\omega^{-1}b_\omega(\epsilon y)|^2  \, dt'\bigg) dy'dy  \\
&+  \int_{\epsilon^{-1}(\Sigma \times \Sigma) } J^{\xi^\bot}(y'-y)  \bigg(\int_{\{-\frac{1}{2\epsilon} \leq t' \leq -R'\}} |z_1(\epsilon y) -z_1(\epsilon y')|^2  \, dt'\bigg)dy'dy \\
&+  \int_{\epsilon^{-1}(\Sigma \times \Sigma) } J^{\xi^\bot}(y'-y) \bigg(\int_{\{R' \leq t' \leq \frac{1}{2\epsilon}\}}|z_2(\epsilon y) -z_2(\epsilon y')|^2 \, dt'\bigg) dy'dy.
\end{align*}
From \eqref{e:limsuppol3.1.1.1} and the regularity assumption on the wells $z_1,z_2$ we estimate 
\begin{align*}
&\int_{\epsilon^{-1}(\Sigma \times \Sigma) } J^{\xi^\bot}(y'-y)  \bigg(\int_{(-R',R')} |\gamma_{\epsilon y'}(t') + a_\omega^{-1}b_\omega(\epsilon y')-\gamma_{\epsilon y}(t') - a_\omega^{-1}b_\omega(\epsilon y)|^2  \, dt'\bigg) dy'dy  \\
& \leq L' \lfloor \partial_t W \rfloor_{C^{0,\alpha}} \int_{\epsilon^{-1}(\Sigma \times \Sigma) } J^{\xi^\bot}(y'-y) \epsilon^{\alpha^2} |y-y'|^{\alpha^2} \, dy'dy =:  L' I^{21}_\epsilon
\end{align*}
where $L':= L'(\delta,K,\|\sigma_{J^\xi}\|_{L^1},\|\overline{\sigma}_{J^\xi}\|_{L^1}, \lfloor z_1 \rfloor_{C^{0,\alpha}},\lfloor z_2 \rfloor_{C^{0,\alpha}}) >0$, and
\begin{align*}
&\int_{\epsilon^{-1}(\Sigma \times \Sigma) } J^{\xi^\bot}(y'-y)  \bigg(\int_{\{R' \leq |t'| \leq \frac{1}{2\epsilon}\}} |z_1(\epsilon y) -z_1(\epsilon y')|^2 + |z_2(\epsilon y) -z_2(\epsilon y')|^2 \, dt'\bigg) dy'dy \\
&\leq  2\int_{\epsilon^{-1}(\Sigma \times \Sigma) } J^{\xi^\bot}(y'-y)  \bigg(\int_{(-\frac{1}{2\epsilon},\frac{1}{2\epsilon})} \epsilon^{2\alpha} |y-y'|^{2\alpha} \, dt'\bigg) dy'dy =:  2I^{22}_\epsilon.
\end{align*}

Since $\eta \leq \alpha^2$ and $J \in \mathcal{L}_n(\eta,\lambda,\rho)$, we have 
% \[
% \int_{\mathbb{R}^{n-1}} \tilde{J}^\xi(y)|y|^{\frac{(1+\alpha)^2}{4}} \, dy \leq \int_{\mathbb{R}^{n}} J(h)(|h|\vee|h|^{\frac{(1+\alpha)^2}{4}}) \, dh < \infty
% \]
\[
\int_{\xi^\bot} J^{\xi^\bot}(y)|y|^{\alpha^2} \, dy \leq \int_{\mathbb{R}^{n}} J(h)(|h| \vee |h|^{\alpha^2} )\, dh < \infty,
\]
which yields $\epsilon^{n-1} I^{21}_\epsilon \to 0$. In addition, since 
\begin{align*}
    \int_{\epsilon^{-1}(\Sigma \times \Sigma) } &J^{\xi^\bot}(y'-y)  \bigg(\int_{(-\frac{1}{2\epsilon},\frac{1}{2\epsilon})} \epsilon^{2\alpha} |y-y'|^{2\alpha} \, dt'\bigg) dy'dy \\
    % &\leq \limsup_{\epsilon \to 0^+} \epsilon^{n-1} \int_{\epsilon^{-1}(\Sigma \times \Sigma) \cap \{|y-y'| \leq N\} } \tilde{J}^\xi(y'-y)   \epsilon^{\alpha} |y-y'|^{1+\alpha} \, dy'dy \\
    % &+ \limsup_{\epsilon \to 0^+} \epsilon^{n-1} \int_{\epsilon^{-1}(\Sigma \times \Sigma) \cap \{|y-y'| > N \}} \tilde{J}^\xi(y'-y)   \epsilon^{\alpha} |y-y'|^{1+\alpha} \, dy'dy \\
    &= \int_{\epsilon^{-1}(\Sigma \times \Sigma) } J^{\xi^\bot}(y'-y)   \epsilon^{2\alpha -1} |y-y'|^{2\alpha} \,  dy'dy,
\end{align*}
we can apply again Lemma \ref{l:decaykernel} with $\beta = 2\alpha -1 >0$ to infer $\epsilon^{n-1}I^{22}_\epsilon \to 0$. Summarizing we obtained
\begin{equation}
\label{e:i21e0}
\lim_{\epsilon \to 0^+} \epsilon^{n-1} I^{2}_\epsilon =0.
\end{equation}

Now we consider the term $I'_\epsilon$. We first notice that, under condition \ref{H3}, we can take the transition maps $\varphi_\rho$ in such a way that $\varphi_\rho(x_0,x,t)= \varphi_1(x_0,\rho x,t)$. Thus, \ref{X2} entails  
\begin{align}
\nonumber
\int_{(-\frac{1}{2\epsilon},\frac{1}{2\epsilon})^2} &J^\xi(t' -t) |\varphi_1(y,\epsilon t\xi,\sigma_y(t)) - \varphi_1(y,\epsilon t'\xi,\sigma_y(t'))|^2 \,dt'dt\\
   %  &\leq c \int_{(-\frac{1}{2\epsilon},\frac{1}{2\epsilon})^2 } J^\xi(t' -t) \epsilon^{1+\alpha}|t-t'|^{1+\alpha} \, dt'dt \\
   % &+ (1+O(\rho))\int_{(-\frac{1}{2\epsilon},\frac{1}{2\epsilon})^2} J^\xi(t' -t) |\sigma_y(t)-\sigma_y(t')|^2 \, dt'dt \\
   \nonumber
   &\leq c \int_{(-\frac{1}{2\epsilon},\frac{1}{2\epsilon})^2 } J^\xi(t' -t) \epsilon^{2\alpha}|t-t'|^{2\alpha} \, dt'dt \\
   \label{e:dom119}
   & \ \ \ + (1+O(\epsilon))a_\omega^2\int_{(-\frac{1}{2\epsilon},\frac{1}{2\epsilon})^2} J^\xi(t' -t) |\gamma_y(t)-\gamma_y(t')|^2 \, dt'dt,
\end{align}
where in the last inequality we used the fact that projections on convex sets (in this case truncation maps) are $1$-Lipschitz maps, and where $c$ is a constant depending on the parameters involved in \ref{X2}. By using again the stability property of the class $\mathcal{L}_n(\eta,\lambda,\rho)$, we know that $J^\xi \in \mathcal{L}_1^{\langle \xi \rangle}(\sigma,\lambda',\rho)$ for some $\lambda' \in (0,1]$. We can thus make use of Lemma \ref{l:decaykernel} with $\beta = 2\alpha -1 >0$ in \eqref{e:dom119} to infer
\begin{align}
\label{e:limsuppol4}
    \limsup_{\epsilon \to 0^+}& \int_{(-\frac{1}{2\epsilon},\frac{1}{2\epsilon})^2} J^\xi(t' -t) |u_{\epsilon,\omega}^{y\xi}(\epsilon t')-u_{\epsilon,\omega}^{y\xi}(\epsilon t)|^2 \,dt'dt \\
    \nonumber
    &\leq a_\omega^2\int_{\mathbb{R} \times \mathbb{R}} J^\xi(t' -t) |\gamma_y(t)-\gamma_y(t')|^2 \, dt'dt . 
\end{align}
Thanks to the assumption $F_0(u) < \infty$, we have $\int_\Sigma \sigma(y,\xi) \, d\mathcal{H}^{n-1} < \infty$. Therefore, by \eqref{e:dom119} the sequence appearing on the left-hand side of \eqref{e:limsuppol4} is dominated for $y \in \Sigma$ as $\epsilon \to 0^+$. We can thus make use of the limsup version of Fatou's lemma to write 
   \begin{align}
\label{e:limsuppoly4}
    \limsup_{\epsilon \to 0^+}& \int_\Sigma\bigg(\int_{(-\frac{1}{2\epsilon},\frac{1}{2\epsilon})^2} J^\xi(t' -t) |u_{\epsilon,\omega}^{y\xi}(\epsilon t')-u_{\epsilon,\omega}^{y\xi}(\epsilon t)|^2 \,dt'dt\bigg)dy \\
    \nonumber
    &\leq a_\omega^2\int_\Sigma\bigg(\int_{\mathbb{R} \times \mathbb{R}} J^\xi(t' -t) |\gamma_y(t)-\gamma_y(t')|^2 \, dt'dt\bigg)dy . 
\end{align}
Eventually, \eqref{e:limsupol10}, \eqref{e:limsup369},\eqref{e:i21e0}, and \eqref{e:limsuppoly4}, together with the fact that the parameter $d$ in \eqref{e:limsupol10} can be chosen arbitrary close to $1$, allow us to deduce
    \begin{align}
        \label{e:limsuppol9}
        \limsup_{\epsilon \to 0^+} & \frac{1}{4\epsilon}\int_{A \times A} J_\epsilon(x'-x) |u_{\epsilon,\omega}(x')-u_{\epsilon,\omega}(x)|^2 \, dx'dx\\
        \nonumber
        &\leq  \frac{a^2_\omega}{4} \int_{\Sigma} \bigg(\int_{\mathbb{R} \times \mathbb{R}} J^\xi(t' -t) |\gamma_y(t)-\gamma_y(t')|^2 \, dt'dt\bigg) dy. 
    \end{align}
\\

 \noindent\textbf{Step 4: Energy estimate for the potential}.
  Now we examine the energy contribution due to the double-well potential. We first observe that by \ref{X1} and \eqref{e:cenfam12345} there holds $\varphi_1(y,\epsilon t\xi,(\sigma_\omega)_y(t))=z_1(y+\epsilon t \xi)$ and $\varphi_1(y,\epsilon t\xi,(\sigma_\omega)_y(t))=z_2(y+\epsilon t \xi)$ for every $t \in (-\infty,-R')$ and $t \in (R',\infty)$, respectively, for all $y\in \Sigma$.
  
  Therefore,
  \begin{align}
      \frac{1}{\epsilon} \int_A W(x,u_{\epsilon,\omega}(x)) \, dx &= \frac{1}{\epsilon} \int_\Sigma \bigg(\int_{A^\xi_y} W(y+t\xi,u_{\epsilon,\omega}(y+t\xi)) \, dt\bigg)dy \\
      &\leq \int_\Sigma \bigg(\int_{\mathbb{R}} W(y+ \epsilon t\xi,u_{\epsilon,\omega}(y+\epsilon t\xi)) \, dt\bigg)dy \\
      \label{e:limsuppol13}
      &= \int_\Sigma \bigg(\int_{(-R',R')} W(y+ \epsilon t\xi,\varphi_1(y,\epsilon t\xi,(\sigma_\omega)_y(t))) \, dt\bigg)dy ,
  \end{align}
  where in the last equality we used \ref{H1}. 
% In order to simplify the next computation we compactly denote $(-\rho/2\epsilon,\rho/2\epsilon)$ by $I_{\rho,\epsilon }$. 
By performing the change of variables induced by the increasing function $\sigma_y$ (see Proposition \ref{p:changevar}), we write  
  \begin{align}
    & \int_\Sigma\bigg(\int_{(-R',R')} W(y+ \epsilon t\xi,\varphi_1(y,\epsilon t\xi,(\sigma_\omega)_y(t))) \, dt\bigg)dy \\
    \label{e:limsuppol12}
    &=\int_\Sigma\bigg( \int_{(z_1(y),z_2(y))} W(y+ \epsilon (\sigma_\omega)_y^{-1}(t)\xi,\varphi_1(y,\epsilon (\sigma_\omega)_y^{-1}(t)\xi,t)) \, d(\dot{\sigma}_\omega)_y^{-1}(t) \bigg)dy.
  \end{align}
 In addition, for every $\epsilon < 1$ the integrand in \eqref{e:limsuppol12} is dominated by 
  \[
\sup_{y \in \Sigma}\sup_{x \in B_{R'}(0)} \sup_{t \in (-M,M)} W(y+x,t) < \infty,
  \]
  where $M >0$ is such that $[z_1(y+x),z_2(y+x)] \subset [-M,M]$ for every $y \in \Sigma$ and $x \in B_{ R'}(0)$.
 This together with \eqref{e:cenfam123456} allow us to apply again the limsup version of Fatou's lemma and write
  \[
  \begin{split}
      &\limsup_{\epsilon \to 0^+} \int_\Sigma \bigg(\int_{(z_1(y),z_{2}(y))} W(y+ \epsilon (\sigma_\omega)_y^{-1}(t)\xi,\varphi_1(y,\epsilon (\sigma_\omega)_y^{-1}(t)\xi,t)) \, d(\dot{\sigma}_\omega)_y^{-1}(t)\bigg)dy \\
      &= \int_\Sigma \bigg(\int_{(z_1(y),z_{2}(y)) } W(y,\varphi_1(y,0,t)) \, d(\dot{\sigma}_\omega)_y^{-1}(t)\bigg)dy, \\
      &= \int_\Sigma \bigg(\int_{(z_1(y),z_{2}(y))} W(y,t) \, d(\dot{\sigma}_\omega)_y^{-1}(t)\bigg)dy,
  \end{split}
  \]
where in the last inequality we used \ref{X3}. Summarizing, we have obtained
   \begin{equation}
      \label{e:limsuppol14}
      \limsup_{\epsilon \to 0^+} \frac{1}{\epsilon} \int_A W(x,u_{\epsilon,\omega}(x)) \, dx \leq \int_\Sigma \bigg(\int_{(z_1(y),z_{2}(y))} W(y,t) \, d(\dot{\sigma}_\omega)_y^{-1}(t)\bigg)dy.
  \end{equation}

Putting together \eqref{e:functional}, \eqref{e:limsuppol9}, and \eqref{e:limsuppol14}, we have
  \begin{align}
  \nonumber
  \limsup_{\epsilon \to 0^+} F_\epsilon(u_{\epsilon,\omega};A) &\leq \frac{a^2_\omega}{4} \int_{\Sigma} \bigg(\int_{\mathbb{R} \times \mathbb{R}} J^\xi(t' -t) |\gamma_y(t)-\gamma_y(t')|^2 \, dt'dt\bigg) dy \\
  \label{e:limsuppol15}
  &+\int_\Sigma \bigg(\int_{(z_1(y),z_{2}(y))} W(y,t) \, d(\dot{\sigma}_\omega)_y^{-1}(t)\bigg)dy.
  \end{align}

Next, we study the last term of the above inequality as $\omega \to 0^+$. By using that $(\sigma_\omega)^{-1}_y(t)  = \gamma^{-1}_y(\varphi_\omega(y,t))$ for $t \in (z_1(y),z_2(y))$, an integration by parts in time shows that 
\begin{align*}
    &\int_\Sigma \bigg(\int_{(z_1(y),z_{2}(y))} W(y,t) \, d\dot{\sigma}_y^{-1}(t)\bigg)dy \\
    &= a_\omega\int_\Sigma \bigg(\int_{\varphi_\omega(y,(z_1(y),z_2(y)))} W(y,\psi_\omega(y,t)) \, d\dot{\gamma}_y^{-1}(t)\bigg)dy.
\end{align*}
Therefore, by \eqref{e:z1z2phi}, \eqref{e:claimlimsup}, and the fact that $W \in L^1(\mathbb{R} \times \Sigma ; \dot{\gamma}_y^{-1} \otimes \mathcal{H}^{n-1})$ (here we are using the assumption $F_0(u) < \infty$), we apply Lebesgue's dominated convergence theorem, and we infer (recall that $a_\omega \searrow 1$ and $b_\omega \to 0$ uniformly in $y \in \Sigma$)
\begin{align}
\nonumber
    \lim_{\omega \to 0^+}\int_\Sigma \bigg(\int_{(z_1(y),z_{2}(y))} W(y,t) \, d\dot{\sigma}_y^{-1}(t)\bigg)dy &= \int_\Sigma \bigg(\int_{(z_1(y),z_{2}(y))} W(y,t) \, d\dot{\gamma}_y^{-1}(t)\bigg)dy \\
    \label{e:limsuppol29}
    &= \int_\Sigma \bigg(\int_{\mathbb{R}} W(y,\gamma_y(t)) \, dt\bigg)dy.
\end{align}
\\

   \noindent\textbf{Step 5: Construction of a recovery sequence near $\Sigma$}.
In this last part of the proof we construct the desired sequence. Notice that, for every fixed $\omega >0$, by \ref{X2}, for $(y,t) \in \Sigma \times [-1,1]$ we have
\[
\begin{split}
|\varphi_1(y,t\xi,(\sigma_\omega)_y(\epsilon^{-1}t)) -z_1(y+t\xi)| &=|\varphi_1(y,t\xi,(\sigma_\omega)_y(\epsilon^{-1}t)) -\varphi_1(y,t\xi,z_1(y))|\\
&\leq  \lfloor \varphi_1(y,t\xi,\cdot)  \rfloor_{C^{0,1}([z_1(y),z_2(y)])} |(\sigma_\omega)_y(\epsilon^{-1}t)-z_1(y)|,
\end{split}
\]
and, analogously, 
\[
\begin{split}
|\varphi_1(y,t\xi,(\sigma_\omega)_y(\epsilon^{-1}t)) -z_2(y+t\xi)| \leq  \lfloor \varphi_1(y,t\xi,\cdot)  \rfloor_{C^{0,1}([z_1(y),z_2(y)])} |(\sigma_\omega)_y(\epsilon^{-1}t)-z_2(y)|,
\end{split}
\]
Thanks to \eqref{e:cenfam12345}, for every fixed $\omega>0$, there holds $|(\sigma_\omega)_y(\epsilon^{-1}t)-z_1(y)| \to 0$ and $|(\sigma_\omega)_y(\epsilon^{-1}t)-z_2(y)| \to 0$ as $\epsilon \to 0^+$ locally uniformly for $(y,t) \in \Sigma \times (0,1)$ and $(y,t) \in \Sigma \times (-1,0)$, respectively. Now fix two compact sets $K^+ \subset \Sigma \times (0,1)$ and $K^- \subset \Sigma \times (-1,0)$, and consider a sequence $\omega_n \searrow 0$. For every fixed $n$ we already know that 
\begin{align}
\label{e:uniconvsup1}
    &|\varphi_1(y,t\xi,(\sigma_{\omega_n})_y(\epsilon^{-1}t)) -z_1(y+t\xi)| \to 0, \ \ \text{ uniformly on $K^-$ as $\epsilon \to 0^+$}\\
    \label{e:uniconvsup1.1}
    &|\varphi_1(y,t\xi,(\sigma_{\omega_n})_y(\epsilon^{-1}t)) -z_2(y+t\xi)| \to 0, \ \ \text{ uniformly on $K^+$ as $\epsilon \to 0^+$}.
\end{align}
In particular, by \eqref{e:uniconvsup1}--\eqref{e:uniconvsup1.1} there exists a sequence $\epsilon_n \searrow 0$ such that
\begin{align}
\label{e:uniconvsup2}
    &|\varphi_1(y,t\xi,(\sigma_{\omega_n})_y(\epsilon^{-1}t)) -z_1(y+t\xi)| \leq \frac{1}{n}, \ \ \text{ for $(y,t)\in K^-$ and $0<\epsilon \leq \epsilon_n$} \\
    \label{e:uniconvsup2.1}
    &|\varphi_1(y,t\xi,(\sigma_{\omega_n})_y(\epsilon^{-1}t)) -z_2(y+t\xi)| \leq \frac{1}{n}, \ \ \text{ for $(y,t)\in K^+$ and $0<\epsilon \leq \epsilon_n$}.
\end{align}
Thus, defining $u_{\epsilon}(y+t\xi):= \varphi_1(y,t\xi,(\sigma_{\omega_n})_y(\epsilon^{-1}t))$ for $\epsilon_{n+1} < \epsilon \leq \epsilon_n$ and every $n \in \mathbb{N}$, then 
\begin{align}
\label{e:uniconvsup3}
    &u_{\epsilon} \to z_1 \ \ \text{ uniformly on $K^-$ as $\epsilon \to 0^+$} \\
    \label{e:uniconvsup3.1}
    &u_{\epsilon} \to z_2 \ \ \text{ uniformly on $K^+$ as $\epsilon \to 0^+$}.
\end{align}
Eventually, by considering sequences of compact sets $K^+_m\nearrow \Sigma \times (0,1)$ and $K^-_m \nearrow \Sigma \times (-1,0)$ we can can exploit \eqref{e:uniconvsup3}--\eqref{e:uniconvsup3.1} and again a diagonal argument, to find a not relabelled subsequence $(u_\epsilon)$ such that 
\begin{align}
\label{e:uniconvsup4}
    &u_{\epsilon} \to z_1 \ \ \text{ locally uniformly on $\Sigma \times (-1,0)$ as $\epsilon \to 0^+$} \\
    \label{e:uniconvsup4.1}
    &u_{\epsilon} \to z_2 \ \ \text{ locally uniformly on $\Sigma \times (0,1)$ as $\epsilon \to 0^+$}.
\end{align}
 To conclude, we observe that, since $a_\omega \to 1$ as $\omega \to 0^+$ by construction, then \eqref{e:limsuppol15} and \eqref{e:limsuppol29} yield also \eqref{e:gli1111}.
\end{proof}

 Our next result ensures that the surface tension $\sigma$ introduced in \eqref{e:surftension} is an upper semi-continuous function.
\begin{proposition}
\label{p:surftension}
    The function $\sigma$ defined in \eqref{e:surftension} is upper semi-continuous. 
\end{proposition}
\begin{proof}
     From Proposition \ref{t:eximin} we know that there exists an increasing minimizer $\gamma_0$ of \eqref{e:surftension} such that $\sigma(x_0,\xi_0)=F_{x_0}^{\xi_0}(\gamma_0)$. Since we want to prove the upper semi-continuity of $\sigma$ at $(x_0,\xi_0)$, we may assume with no loss of generality that $\sigma(x_0,\xi_0)< \infty$.

     Consider sequences $\xi_n \to \xi_0$, $\rho_n \to 0$, and $(x_n) \subset \overline{B}_1(0)$, and define $\gamma_n \colon \mathbb{R} \to \mathbb{R}$ as $\gamma_n(t):= \varphi_{\rho_n}(x_0,x_n,\gamma_0(t))$, where for every $n\in \mathbb{N}$, the function $\varphi_{\rho_n} \colon \mathbb{R}^m \times \mathbb{R}^m \times \mathbb{R} \to \mathbb{R}$ is the transition map given in \ref{X1}--\ref{X4} (in this case we have set $K:= \overline{B}_1(0)$). In particular, we immediately deduce $\gamma_n \in X_{z_1(x_0 + \rho_n x_n)}^{z_2(x_0 + \rho_n x_n)}$.
    
    We first notice the pointwise convergence
    \begin{equation}
    \label{e:uppsem112}
        |\gamma_n(t+h) -\gamma_n(t)|^2 \to |\gamma_0(t+h) -\gamma_0(t)|^2, \ \ \text{for $\mathcal{L}^{2}$-a.e. $(t,h) \in \mathbb{R}\times \mathbb{R}$},
    \end{equation}
    and the uniform bound with respect to the $h$-variable of the following quantity
    \begin{align}
         \nonumber
         \int_{\mathbb{R}}\frac{|\gamma_n(t+h) -\gamma_n(t)|^2}{|h|}\, dt &\leq \sup_n \lfloor  \varphi_{\rho_n}(x_0,x_n,\cdot) \EEE\rfloor^2_{C^{0,1/2}} \int_\mathbb{R} \frac{|\gamma_0(t+h) -\gamma_0(t)|}{|h|} \, dt \\
         \label{e:uppsem113}
         &\leq \sup_n \lfloor  \varphi_{\rho_n}(x_0,x_n,\cdot) \EEE \rfloor^2_{C^{0,1/2}} |z_1(x_0),z_2(x_0)|.     
    \end{align}
    We further notice that, since $\int_{ \mathbb{R}^m \EEE} J(h) |h| \, dx < \infty$ by assumption, then thanks to the continuity of the $L^1$-norm with respect to the action of the group of rotations of $ \mathbb{R}^m \EEE$ we infer
    \begin{equation}
    \label{e:strongl1}
    \lim_{n \to \infty}\int_{\mathbb{R}}|J^{\xi_n}(t)|t| - J^{\xi_0}(t)|t|| \, dt =0.
    \end{equation}
    We thus estimate
    \begin{align}
    \label{e:line1}
        &\frac{1}{4}\int_{\mathbb{R} \times \mathbb{R}} |J^{\xi_n}(h)(\gamma_n(t+h)-\gamma_n(t))^2-J^{\xi_0}(h)(\gamma_0(t+h)-\gamma_0(t))^2| \, dtdh \\
        \label{e:line2}
        &\leq \frac{1}{4}\int_{\mathbb{R} \times \mathbb{R}} |J^{\xi_n}(h)|h| - J^{\xi_0}(h)|h||  \frac{(\gamma_n(t+h)-\gamma_n(t))^2}{|h|} \, dtdh \\
        \label{e:line3}
        &+ \frac{1}{4} \int_{\mathbb{R} \times \mathbb{R}} |J^{\xi_0}(h)  |(\gamma_n(t+h)-\gamma_n(t))^2-(\gamma_0(t+h)-\gamma_0(t))^2| \, dtdh.
    \end{align}
By combining \eqref{e:uppsem113}--\eqref{e:strongl1} we immediately infer that the integral \eqref{e:line2} goes to $0$ as $m \to \infty$. By arguing as in \eqref{e:uppsem113}, the integrand in \eqref{e:line3} is dominated by the function $f \in L^1(\mathbb{R}^2)$ defined as
\begin{align*}
f(t,h):= \sup_n (\lfloor  \varphi_{\rho_n}(x_0,x_n,\cdot) \EEE \rfloor^2_{C^{0,1/2}}+1)  &\frac{|\gamma_0(t+h) -\gamma_0(t)|}{|h|} J^{\xi_0}(h)|h|.
\end{align*}
Hence, we can apply Lebesgue's dominated convergence theorem in combination with \eqref{e:uppsem112} to conclude that \eqref{e:line3} goes to $0$ as $n\to \infty$.

    In addition, we exploit \ref{X5} and argue as in the proof of Proposition \ref{p:contdep} to infer 
    \begin{align*}
    \lim_{n \to \infty}\int_{\mathbb{R}} W(x_n,\gamma_n(t)) \, dt =\int_{\mathbb{R}} W(x_0,\gamma_0(t)) \, dt. 
    \end{align*}
     Therefore we have proved the convergence
    \begin{equation}
        \lim_{n\to +\infty} F^{\xi_n}_{0,x_n}(\gamma_n) = F^{\xi_0}_{0,x_0}(\gamma_0).
    \end{equation}
    But this means that
    \begin{align*}
    \limsup_{n\to +\infty}\sigma(x_n,\xi_n) \leq \lim_{n\to \infty} F^{\xi_n}_{0,x_n}(\gamma_n)  =F_{0,x_0}^{\xi_0}(\gamma_0) = \sigma(x_0,\xi_0).
    \end{align*}
    The arbitrariness of the sequences $(x_n),(\xi_n),(\rho_n)$ allows us to infer that $\sigma$ is upper semi-continuous at $(x_0,\xi_0)$.
\end{proof}

We are now in position to conclude the proof of the $\Gamma$-limsup upper-bound.

\begin{theorem}[Recovery sequence]
\label{t:recgen}
    Let $\Omega \subset  \mathbb{R}^m \EEE$ be open, let $J \colon  \mathbb{R}^m \EEE \to (0,\infty)$ be an even function, and let $W \colon  \mathbb{R}^m \EEE \times \mathbb{R} \to [0,\infty)$ satisfy \ref{H1}--\ref{H4}. Let further $\alpha \in (1/2,1]$ be given by \ref{H1}. Assume that either one of the following conditions is satisfied:
    \begin{enumerate}[(i)]
       \item  $Q_x(t):= \partial_t W(x,t)+t$ has a continuous inverse for every $x \in \Omega$ and $t \in [z_1(x),z_2(x)]$, and $J$ satisfies \ref{K1} with $\|J\|_{L^1( \mathbb{R}^m \EEE)}=1$
       \item $J \in \mathcal{L}_m(\eta,\lambda,\rho)$ for some $\eta \in (0,\alpha^2)$ and $\lambda,\rho \in (0,1)$.
   \end{enumerate}
     Then, for every $u \in CP(\Omega;\mathbf{z}(x))$  there exists a sequence $(u_\epsilon) \subset L^1(\Omega)$ such that $u_\epsilon \to u$ in $L^1(\Omega)$ and
    \begin{equation}
        \label{e:gli11111}
        \limsup_{\epsilon \to 0^+} F_{\epsilon}(u_\epsilon;\Omega) \leq F_0(u).
    \end{equation}
\end{theorem}

\begin{proof}
    Since $u \in CP(\Omega;\mathbf{z}(x))$ we know that there exists a set $E$ of finite perimeter in $\Omega$ such that $u=z_1$ on $E$ and $u=z_2$ on $\Omega \setminus E$. By the approximation result in \cite[Theorem 1.24]{giusti} we find a sequence $(E_k)$ of polyhedral sets, such that 
    \begin{align}
    \label{e:recseq3}
    &\mathbbm{1}_{E_k} \to \mathbbm{1}_E, \ \ \text{ in $L^1$ \ as $k \to \infty$}, \\
    \label{e:recseq4}
    &|D\mathbbm{1}_{E_k}|(\Omega) \to |D\mathbbm{1}_{E}|(\Omega), \ \ \text{ as $k \to \infty$}.
    \end{align}
    Define $v_k \colon \Omega \to \mathbb{R}$ as $v_k(x):= z_1(x) \mathbbm{1}_{E_k}(x)+ z_2(x) \mathbbm{1}_{\Omega \setminus E_k}(x)$. Then, $v_k \in CP(\Omega;\mathbf{z}(x))$ is a polyhedral function for every $k\in \mathbb{N}$. Thus, owing to Proposition \ref{t:recpol} we find a sequence $(v_{k,\epsilon})_\epsilon \subset L^1(\Omega)$ such that
    \begin{align}
    \label{e:recseq1}
    &v_{k,\epsilon} \to v_k, \  \ \text{ in $L^1(\Omega)$, \ as $\epsilon \to 0^+$\ for every $k\in \mathbb{N}$},\\
    \label{e:recseq2}
    &\limsup_{\epsilon \to 0^+} F_\epsilon(v_{k,\epsilon};\Omega) \leq F_0(v_k), \ \ \text{ for every $k\in \mathbb{N}$}.
    \end{align}
    Hence, thanks to a diagonal argument, to conclude the proof of the theorem, it remains only to show that $\limsup_{k \to \infty} F_0(v_k) \leq F_0(u)$. To this purpose, we notice that for every $u \in CP(\Omega;\mathbf{z}(x))$ we have $F_0(u)=\tilde{F}_0(E)$ where
    \[
    \tilde{F}_0(E):= \int_{\Omega} \sigma\left(x,\frac{D\mathbbm{1}_E}{|D\mathbbm{1}_E|}(x)\right) \, d|D\mathbbm{1}_E|(x), \ \ \text{ for $E:= \{u = z_1\}$}.
    \]
     As a result, it is enough to show that $\limsup_{k \to \infty} \tilde{F}_0(E_k) \leq \tilde{F}_0(E)$. In view of \eqref{e:recseq3}--\eqref{e:recseq4} and the upper semi-continuity of $\sigma$ given by Proposition \ref{p:surftension}, this latter inequality follows from a result due to Reshetnyak (see \cite[Appendix]{modica}).
\end{proof}

We conclude this section with the proof of Lemma \ref{l:decaykernel}.
\begin{proof}[Proof of Lemma \ref{l:decaykernel}]
By symmetrizing we may assume with no loss of generality that $-A=A$. Given two sets $B,B' \subset \mathbb{R}^m$ we denote by $B+B' := \{x \in \mathbb{R}^m \ | \ x= z+z', \ z \in B,\ z' \in B'\}$. By changing variables we have 
\begin{align*}
    \int_{\epsilon^{-1}(A \times A)} J(x'-x)  \epsilon^{\beta} |x-x'|^{1+\beta} \,dx'dx \leq \int_{\epsilon^{-1}A  \times \epsilon^{-1}(A+A) } J(h)  \epsilon^{\beta} |h|^{1+\beta} \,dxdh. 
\end{align*}

 Now,
\begin{align*}
    &\epsilon^{m} \int_{\epsilon^{-1}A  \times \epsilon^{-1}(A+A) } J(h)  \epsilon^{\beta} |h|^{1+\beta} \,dxdh 
    \leq \mathcal{H}^{m}(A ) \int_{\epsilon^{-1}(A+A) \cap \{|h| \leq N\}} J(h) \epsilon^\beta |h|^{1+\beta} \, dh \\
    &\qquad+\mathcal{H}^{m}(A ) \int_{\epsilon^{-1}(A+A) \cap \{|h| > N\}} J(h) \epsilon^\beta |h|^{1+\beta} \, dh 
    \leq \mathcal{H}^{m}(A ) \int_{ \{|h| \leq N\}} J(h) \epsilon^\beta |h|^{1+\beta} \, dh \\
    &\qquad+  \mathcal{H}^{m}(A) \text{diam}(A+A)^{\beta} \int_{\{|h|>N\}} J(h)  |h| \, dh,
\end{align*}
where we used that, from the simmetry assumption on $A$, we have $A+A \subset \overline{B}_{\text{diam}(A+A)}(0)$. Since $J \in \mathcal{L}(\sigma,\lambda,\rho)$ we have $\int_{\{|h| \leq N\}} J(h)|h|^{1+\beta} <\infty$ and $\int_{\{|h| > N\}} J(h)|h|<\infty$. Taking the limsup on both sides of the above chain of inequalities, we conclude that 
\[
\begin{split}
\limsup_{\epsilon \to 0^+}\,&\epsilon^{m} \int_{\epsilon^{-1}A  \times \epsilon^{-1}(A+A) } J(h)  \epsilon^{\beta} |h|^{1+\beta} \,dxdh \\
&\leq \mathcal{H}^{m}(A) \text{diam}(A+A)^{\beta} \int_{\{|h|>N\}} J(h)  |h| \, dh \to 0 \text{ as $N \to \infty$}.
\end{split}
\]
Hence, \eqref{e:decaykernel1} follows from the arbitrariness of $N>0$.
\end{proof}

\section{The $\Gamma$-liminf inequality}
\label{s:gammalimif}

In this section we establish a lower bound for the asymptotic behavior of our sequence of functionals. 
Along the whole section we fix $\xi \in  \mathbb{S}^{m-1} \EEE$ and an orthonormal basis $\{\xi_1,\dotsc,\xi_{n-1}\}$ of $\xi^\bot$. In the following we will make use of the notation $C_\rho(0)$, $Q_\rho(x_0)$, and $V_\xi$, introduced in Definition \ref{def:W-restr}. We denote by $M$ any integer such that \eqref{e:boundwells123} holds true.
In order to simplify the notation, we further set 
\[
S^\epsilon_\rho:= \{y+t\xi \ | \ y \in C_\rho(0), \ t\in (-\rho/2\epsilon,\rho/2\epsilon) \}.
\] 
 Given a function $u \colon  \mathbb{R}^m \EEE \to \mathbb{R}$ and an $(n-1)$-dimensional cube $C_\rho(0)$ oriented by $\xi$, we define $u^{yz} \colon \mathbb{R} \to \mathbb{R}$ as 
\begin{equation}
\begin{split}
u^{yz}(t) &:= u(y+tz), \ \ \text{ for every } y \in \xi^\bot, \ z \in V_\xi,
\end{split}
\end{equation}
and we denote by $\overline{u} \colon  \mathbb{R}^m \EEE \to \mathbb{R}$ the $C_\rho(0)$-periodic extension of $u$, namely,
\begin{equation}
\overline{u}(y+t\xi) = u\left(y- \sum_{i=1}^{n-1} m_i \xi_i +t \xi\right), \ \ \text{ for every $(y,t) \in \xi^\bot \times \mathbb{R}$ and $m \in \rho \mathbb{Z}^{n-1}$}.
\end{equation}
In the proof of the liminf inequality, we will often need to modify $C_\rho(0)$-periodic extensions on a scale $\epsilon$. We will argue according to the following construction.
\begin{definition}
    Given a family $(u_\epsilon)_\epsilon$ such that $u_\epsilon \colon S_\rho^\epsilon \to \mathbb{R}$, we define $\mathring{u}_{\epsilon,\xi} \colon  \mathbb{R}^m \EEE \to \mathbb{R}$ to be its $C_\rho(0)$-periodic modification at scale $\epsilon$ and in direction $\xi \in  \mathbb{S}^{m-1} \EEE$  as
\begin{equation}
\label{e:modification1}
    \mathring{u}_{\epsilon,\xi}(x) :=
    \begin{cases}
        \overline{u}_\epsilon(x), &\text{ if }  x=y+t\xi \in S^\xi_\rho \EEE \\
        \overline{z}_1(x_0+\frac{y}{t}) &\text{ if }  x = y+t\xi, \ (y,t) \in \xi^\bot \times (-\infty,-\rho/2\epsilon) \EEE \\
        \overline{z}_2(x_0+\frac{y}{t}) &\text{ if }  x = y+t\xi, \ (y,t) \in \xi^\bot \times (\rho/2\epsilon,\infty) \EEE.
    \end{cases}
\end{equation}
\end{definition}
By construction $\mathring{u}_{\epsilon,\xi}$ is $C_\rho(0)$-periodic. For simplicity of notation, we have omitted the dependence of the above sequence from $\rho$.

We begin by stating a technical result, concerning a double liminf inequality. Its proof, relying on a directional blow-up argument, is postponed to Section \ref{s:dirblowup}.

To a family $(\psi_\epsilon)_{\epsilon >0}$ of $C^1$-diffeomorphisms of $\xi^\bot$ we associate the kernels $$J_{\psi_\epsilon}(y+t\xi):= J(\psi_\epsilon(y)+t\xi) |\det(\nabla \psi_\epsilon(y))|$$ for $(y,t) \in \xi \times \mathbb{R}$, and we define
\begin{equation}
\label{e:defjpsi}
F_{\psi_\epsilon,x_0}(u;B):= \frac{1}{4} \int_{ \mathbb{R}^m \EEE \times B} J_{\psi_\epsilon}(h)|u(x+h)-u(x)|^2 \, dhdx + \int_{B } W_{x_0}^\epsilon(x,u(x)) \, dx,
\end{equation}
for every $u \colon  \mathbb{R}^m \EEE \to \mathbb{R}$ measurable and every $B \subset  \mathbb{R}^m \EEE$ Borel measurable.  Our double liminf inequality reads as follows.

\begin{proposition}
\label{c:est_gamminf}
    Let $(\psi_\epsilon)_{\epsilon >0}$ be a family of $C^1$-diffeomorphisms of $\xi^\bot$. Let $x_0 \in \Omega$, and let $(u_\epsilon)_\epsilon$ be a family of measurable functions $u_\epsilon \colon \Omega \to [-M,M]$. Assume that
    \begin{equation}
    \label{e:l1conv2}
        u_\epsilon \to u, \ \ \text{ in $L^1(\Omega)$ as $\epsilon \to 0^+$}.
    \end{equation}
    Then, by letting $v_\epsilon \colon S_\rho^\epsilon \to [-M,M]$ be defined as $v_\epsilon(y+t\xi):=u_\epsilon(x_0+ y+\epsilon t \xi )$, there exist $\rho_{x_0}>0$ and a set of full measure $I_{x_0}\subset (0,\rho_{x_0})$ such that, for every infinitesimal sequence $(\rho_\ell) \subset I_{x_0}$ there holds 
     \begin{align}
         \label{e:est_gammainf7}
         \liminf_{\rho_\ell \to 0}\EEE \liminf_{\epsilon \to 0^+}F_{\psi_\epsilon,x_0}(\mathring{v}_{\epsilon,\xi}; &S_{\rho_\ell}^\epsilon) \rho_\ell^{1-n} \geq \inf_{\gamma \in X_{z_1(x_0)}^{z_2(x_0)}} F^\xi_{x_0}(\gamma) \\
         &- C\|\hat{J}\|_{L^1}  \limsup_{\rho_\ell\to 0} \EEE  \mint_{\partial Q^-_{\rho_\ell}(x_0)}|u(x)-z_1(x_0)| \, d\mathcal{H}^{n-1} \\
&- C  \|\hat{J}\|_{L^1}  \limsup_{\rho_\ell\to 0^+} \EEE  \mint_{\partial Q^+_{\rho_\ell}(x_0)}|u(x)-z_2(x_0)| \, d\mathcal{H}^{n-1},
     \end{align}
     where $F_{\psi_\epsilon,x_0}$ is the functional defined in \eqref{e:defjpsi}, the maps $\mathring{v}_{\epsilon,\xi}$ are as in \eqref{e:modification1},  $C:=C(M,Z,\lfloor z_1 \rfloor_{C^{0,\alpha}},\lfloor z_2 \rfloor_{C^{0,\alpha}})>0$, we set  $\partial Q^\pm_\rho(x_0):= \partial Q_\rho(x_0) \cap \{x \ |\ \pm(x-x_0) \cdot \xi >0 \}$, and $F^\xi_{x_0}$ is given by \eqref{e:functionalxi148}.
\end{proposition}

\begin{remark}
\label{r:est_gamminf}
    It is worth noticing that Proposition \ref{c:est_gamminf} does not depend on the chosen completion  $ \{\xi_1,\dotsc,\xi_{n-1}\} \EEE$ of $\{\xi\}$ to an orthonormal basis of $ \mathbb{R}^m \EEE$. This is due to the fact that Lemma \ref{l:est_gammainf}, on which the proposition is based, enjoys the same property. 
\end{remark}

We are now in a position to state and prove the first main result of this section, namely the identification of a lower bound in the case in which the admissible kernels are integrable.

\begin{theorem}[$\Gamma$-liminf with integrable kernel]
\label{t:gammaliminf1}
    Let $\Omega \subset  \mathbb{R}^m \EEE$ be open, let $J \colon  \mathbb{R}^m \EEE \to (0,\infty)$ be an even function fulfilling \eqref{K1}, and let $W \colon  \mathbb{R}^m \EEE \times \mathbb{R} \to [0,\infty)$ satisfy \ref{H1}--\ref{H5}. Then, for every sequence $(u_{\epsilon})\subset L^1(\Omega)$ such that $u_\epsilon \to u$ in $L^1(\Omega)$, there holds
    \begin{equation}
        \label{e:gli1000}
        \liminf_{\epsilon \to 0^+} F_{\epsilon}(u_\epsilon;\Omega) \geq F_0(u).
    \end{equation}
\end{theorem}
\begin{proof}
We may assume with no loss of generality that $ \liminf_{\epsilon \to 0^+} F_{\epsilon}(u_\epsilon;\Omega) < \infty$. 

By virtue of the compactness result we know that $u \in CP(\Omega;\mathbf{z}(x))$ and in particular that $\mathcal{H}^{n-1}(J_u) < \infty$. 
Since projections on convex subsets of $ \mathbb{R}^m \EEE$ are $1$-Lipschitz maps and since the constant $M$ has been chosen satisfying \eqref{e:boundwells123}, we deduce that $F_\epsilon(u^M_\epsilon;\Omega) \leq F_\epsilon(u_\epsilon;\Omega)$ for every $\epsilon>0$, where $u^M_\epsilon$ is the truncation of $u_\epsilon$ between $-M$ and $M$. Moreover, property \eqref{e:boundwells123} together with $u(x) \in \{z_1(x),z_2(x)\}$ for a.e. $x$, imply that also the truncated sequence satisfies $u^M_\epsilon \to u$ in $L^1(\Omega)$. Therefore, without loss of generality, we assume that $u_\epsilon \colon \Omega \to [-M,M]$.

  Now, let $x_0 \in J_u$, and denote by $\nu_u(x_0)$ a unit vector orthogonal to the tangent space of $J_u$ at $x_0$. Let $\{e_1,\dotsc,e_{n-1},\nu_u(x_0)\}$ be an orthonormal basis of $ \mathbb{R}^m \EEE$, and let $Q_\rho(x_0)$ be the hypercube parallel to the given basis. By Proposition \ref{p:etrace_convergence}, the $\epsilon_\ell$-traces of $u_{\epsilon_\ell}$ relative to $\Omega$ converge on $\partial Q_{\rho}(x_0)$ to $u$ for a.e. $\rho$ (sufficiently small depending on $x_0$), for every infinitesimal sequence $(\epsilon_\ell)$. Therefore, by using that the $\epsilon_\ell$-traces convergence relative to $A'$ implies the $\epsilon_\ell$-traces convergence relative to $A$ whenever $A \subset A'$, and by the arbitrariness of $\epsilon_\ell$, we infer the following 
\begin{enumerate}
    \item[(i)] for every $x_0 \in J_u$ the $\epsilon$-traces of $u_{\epsilon}$ relative to $Q_{\rho}(x_0)$ converge on $\partial Q_\rho(x_0)$ to $u$ for a.e.  $\rho \in I_{x_0}$ \EEE such that $Q_\rho(x_0) \subset \Omega$, 
\end{enumerate}
 where $I_{x_0}$ denotes the set of $\rho >0$ for which Proposition \ref{c:est_gamminf} is satisfied. \EEE In addition, since $x_0 \in J_u$, there exists a subsequence $\rho_\ell \searrow 0$ (depending on $x_0$) such that
\begin{align}
\label{e:jumpconv1}
&\lim_{\ell \to \infty} \frac{1}{(t\rho_\ell)^{n-1}} \int_{\partial Q^+_{t\rho_\ell}(x_0)} |u(x)-z_2(x_0)| \, d\mathcal{H}^{n-1}=0, \ \ \text{ for a.e. $t \in (0,1)$} \\
\label{e:jumpconv2}
&\lim_{\ell \to \infty} \frac{1}{(t\rho_\ell)^{n-1}} \int_{\partial Q^-_{t\rho_\ell}(x_0)} |u(x)-z_1(x_0)| \, d\mathcal{H}^{n-1}=0, \ \ \text{ for a.e. $t \in (0,1)$},
\end{align}
where we recall that $\partial Q^\pm_{\rho}(x_0):= \partial Q_\rho(x_0) \cap \{x \in  \mathbb{R}^m \EEE \ | \ \pm (x-x_0) \cdot \nu_u(x_0) >0\}$.  By exploiting the fact that $(\rho_\ell)$ is a countable set, we can find $t \in (0,1)$ such that, setting with abuse of notation $\rho_\ell :=t\rho_\ell$, the following condition holds:\EEE
\begin{enumerate}
    \item[(ii)] for every $x_0 \in J_u$ there exists a sequence  $(\rho_\ell) \subset I_{x_0}$ \EEE with $\rho_\ell \searrow 0$ such that the $\epsilon$-traces of $u_{\epsilon}$ relative to $Q_{\rho_\ell}(x_0)$ converge on $\partial Q_{\rho_\ell}(x_0)$ to $u$ and 
    \begin{align}
\label{e:jumpconv1.1}
&\lim_{\ell \to \infty} \frac{1}{(\rho_\ell)^{n-1}} \int_{\partial Q^+_{\rho_\ell}(x_0)} |u(x)-z_2(x_0)| \, d\mathcal{H}^{n-1}=0, \ \ \text{ for a.e. $t \in (0,1)$} \\
\label{e:jumpconv2.1}
&\lim_{\ell \to \infty} \frac{1}{(\rho_\ell)^{n-1}} \int_{\partial Q^-_{\rho_\ell}(x_0)} |u(x)-z_1(x_0)| \, d\mathcal{H}^{n-1}=0, \ \ \text{ for a.e. $t \in (0,1)$},
\end{align}
for every $\ell$ such that $Q_{\rho_\ell}(x_0) \subset \Omega$.
\end{enumerate}

Eventually, by Proposition \ref{c:est_gamminf}, properties \eqref{e:jumpconv1.1}--\eqref{e:jumpconv2.1} entail: 
\begin{enumerate}
    \item[(iii)] let $x_0 \in J_u$; then the subsequence $(\rho_\ell)$ associated to $x_0$ by condition (ii) also satisfies
    \begin{equation}
        \label{e:liminfen456}
        \liminf_{\rho_\ell \to 0}\liminf_{\epsilon \to 0^+} F_{\psi_\epsilon,x_0}(\mathring{v}_{\epsilon,\xi_0}; S_{\rho_\ell}^\epsilon) \rho_\ell^{1-n}  \geq \sigma(x_0,\xi_0), 
    \end{equation}
    where $\xi_0 = \nu_u(x_0)$, $\psi_\epsilon(y)= \epsilon^{-1}y$, $S_{\rho_\ell}^\epsilon=C_{\rho_\ell} \times (-\frac{\rho_\ell}{2\epsilon},\frac{\rho_\ell}{2\epsilon})$ where $C_{\rho_\ell}$ is any $(n-1)$-dimensional cube of size $\rho_\ell$ centered at $0$ and oriented by $\xi_0$, $v_\epsilon(y+t\xi_0) := u_\epsilon(x_\ell +y+\epsilon t \xi_0)$ for $y \in \xi_0^\bot$, and $\mathring{v}_{\epsilon,\xi_0}$ is given by \eqref{e:modification1}.
\end{enumerate}

In view of (ii) and (iii), given $\tau >0$, the fact that for every $x_0 \in J_u$ the corresponding sequence $(\rho_\ell)$ is infinitesimal allows us to make use of the Morse covering theorem \cite{fonleo} and to infer the existence of finitely many points $(x_m) \subset J_u$ and strictly positive radii $(\rho_m)$ such that (we assume $u^+(x_m)= z_2(x_m)$ and $u^-(x_m)=z_1(x_m)$ but nothing would have changed with the opposite choice, since $\sigma(x,\xi)=\sigma(x,-\xi)$):
   \begin{align}
   \label{e:gli1}
   & \ \ \ \ \ \text{the $\epsilon$-traces of $u_{\epsilon}$ relative to $Q_{\rho_m}(x_m)$ converge on $\partial Q_{\rho_m}(x_m)$ to $u$}\\
   \label{e:jumpconv1.1.1}
& \ \ \ \ \ \ \ \ \ \ \ \ \ \ \ \ \ \ \ \ \ \ \int_{\partial Q^+_{\rho_m}(x_m)} |u(x)-z_2(x_m)| \, d\mathcal{H}^{n-1} \leq \tau\rho_m^{n-1}\\
\label{e:jumpconv2.1.1}
&\ \ \ \ \ \ \ \ \ \ \ \ \ \ \ \ \ \ \ \ \ \ \int_{\partial Q^-_{\rho_m}(x_m)} |u(x)-z_1(x_m)| \, d\mathcal{H}^{n-1}\leq \tau\rho_m^{n-1} \\
   \label{e:gli2}
   & \ \ \ \ \ \ \ \ \ \ \ \ \ \ \ \ \ \ \ \ \ \ \ \ \  \sum_{m} \sigma(x_m,\xi_m) \rho_m^{n-1} \geq F_0(u) -\tau,\\
   \label{e:gli3}
   & \ \ \ \ \ \ \ \ \ \ \ \ \ \ \ \ \ \ \ \ \ \ \ \ \ \ \ \ \ \sum_\ell \rho_m^{n-1} \leq \mathcal{H}^{n-1}(J_u) + \tau\\
   \label{e:gli4}
   & \ \ \ \ \ \ \ \ \ \ \ \ \liminf_{\epsilon \to 0^+} F_{\psi_\epsilon,x_0}(\mathring{v}_{\epsilon,\xi_m}; S_{\rho_m}^\epsilon)  \geq \sigma(x_m,\xi_m) \rho_m^{n-1} - \rho_{m}^{n-1} \tau 
   \end{align}
   where $\xi_m:=\nu_u(x_m)$, $C_{\rho_m}$ is oriented accordingly to $\xi_m$, $v_\epsilon(y+t\xi_m) := u_\epsilon(x_m +y+\epsilon t \xi_m)$ for $y \in \xi_m^\bot$ (we omit the dependence of $v_\epsilon$ from $m$ in the notation).

 We hence infer 
   \begin{align}
  \nonumber
       &\liminf_{\epsilon \to 0^+} F_\epsilon(u_\epsilon;\Omega) \geq \sum_{m} \liminf_{\epsilon \to 0^+} F_\epsilon(u_\epsilon;Q_{\rho_m}(x_m)) \\
       \nonumber
       &=\sum_{m} \liminf_{\epsilon \to 0^+} \frac{1}{4}\int_{S_{\rho_m}^\epsilon \times S_{\rho_m}^\epsilon} J_{\psi_\epsilon}(x-x')|v_\epsilon(x)-v_\epsilon(x')|^2 \, dxdx' + \int_{S_{\rho_m}^\epsilon} W_{x_m}^\epsilon(x,v_\epsilon(x)) \, dx \\
       \label{e:defest1.1.6}
       &\geq \sum_{m} \liminf_{\epsilon \to 0^+} F_{\psi_\epsilon,x_m}(\mathring{v}_{\epsilon,\xi_m}; S_{\rho_m}^\epsilon)\\
       \nonumber
       &-\sum_m\limsup_{\epsilon \to 0^+} \int_{( \mathbb{R}^m \EEE \setminus S_{\rho_m}^\epsilon) \times S_{\rho_m}^\epsilon} J_{\psi_\epsilon}(x'-x)|\mathring{v}_{\epsilon,\xi_m}(x')-\mathring{v}_{\epsilon,\xi_m}(x)|^2 \, dx'dx.
   \end{align}
   Notice that, by performing the inverse change of variables considered in the above equality, we can rewrite the last term by means of the defect functional, cf. \eqref{e:deficit}, as
   \begin{align}
   \nonumber
   \int_{( \mathbb{R}^m \EEE \setminus S_{\rho_m}^\epsilon) \times S_{\rho_m}^\epsilon} &J_{\psi_\epsilon}(x'-x)|\mathring{v}_{\epsilon,\xi_m}(x')-\mathring{v}_{\epsilon,\xi_m}(x)|^2 \, dx'dx\\
   \label{e:defecto0}
   &= \Lambda_\epsilon(\tilde{u}_\epsilon ; \mathbb{R}^m \EEE \setminus Q_{\rho_m}(x_m),Q_{\rho_m}(x_m)) ,
   \end{align}
where $\tilde{u}_\epsilon \colon  \mathbb{R}^m \EEE \to [-M,M]$ is defined as
\[
\tilde{u}_\epsilon(x):=
\begin{cases}
    \overline{u}_\epsilon(x) &\text{ if } x = t+t\xi_m\EEE, \ (y,t) \in \xi_m^\bot \times (-\frac{\rho_m}{2},\frac{\rho_m}{2})\\
    \overline{z}_1(x_0+\epsilon \frac{y}{t}) &\text{ if }x=y+t\xi_m, \ (y,t)\in \xi_m^\bot \times (-\infty,-\frac{\rho_m}{2})\\
     \overline{z}_2(x_0+\epsilon \frac{y}{t}) &\text{ if }x=y+t\xi_m, \ (y,t) \in \xi_m^\bot \times (\frac{\rho_m}{2},\infty).
\end{cases}
\]

 We claim that the defect functional in \eqref{e:defecto0} goes to $0$ as $\epsilon \to 0^+$. To show this, fix $m$ and consider the orthonormal basis $\{e_1,\dotsc,e_{n-1},\xi_m\}$. Denote by $\Xi$ the family of all coordinate $(n-1)$-planes $V$ with respect to $\{e_1,\dotsc,e_{n-1},\xi_m\}$. We split $ \mathbb{R}^m \EEE\setminus Q_{\rho_m}(x_m)$ as the disjoint union of two sets $A_1$ and $A_2$, defined as
 \[
 \begin{split}
 A_1 &:= \bigcup_{V \in \Xi} \pi_{V}^{-1}(\pi_{V}(Q_{\rho_m}(x_m))) \setminus Q_{\rho_m}(x_m) \\
 A_2 &:= ( \mathbb{R}^m \EEE \setminus Q_{\rho_m}(x_m)) \setminus A_1.
 \end{split}
 \]
We further denote $\pi_{V}^{-1}(\pi_{V}(Q_{\rho_m}(x_m))) \setminus Q_{\rho_m}(x_m)$ by $A^V_1$. Let $V \in \Xi$, and suppose for example that $V \neq \xi_m^\bot$. We show that $\Lambda_\epsilon(\tilde{u}_\epsilon;A^V_1,Q_{\rho_m}(x_m)) \to 0$. To this purpose, by letting $e_V$ denote the element of $\{e_1,\dotsc,e_{n-1}\}$ which is orthogonal to $V$, we further split the set $A^V_1$ as the disjoint union of two sets $B^V_1$ and $B^V_2$ defined as
\[
\begin{split}
    B^V_1 &:= (Q_{\rho_m}(x_m) - \rho_m e_V) \cup (Q_{\rho_m}(x_m) + \rho_m e_V) \\
    B^V_2 &:= A^V_1 \setminus B^V_1.
\end{split}
\]

 In view of \eqref{e:gli1} we can make use of Proposition \ref{p:deficit} in combination with the $C_{\rho_m}$-periodicity of $\tilde{u}_\epsilon$ to infer
 \begin{align*}
     \limsup_{\epsilon \to 0^+} \Lambda_\epsilon(\tilde{u}_\epsilon;B^V_1,Q_{\rho_m}(x_m)) &\leq \|\hat{J}\|_{L^1} \int_{\partial Q_{\rho_m}(x_m) \cap (V+x_m+\rho_m e_V/2)} |u(x)-u(x-\rho_m e_v)| \, d\mathcal{H}^{n-1}\\
     &= \|\hat{J}\|_{L^1} \int_{\partial Q^+_{\rho_m}(x_m) \cap (V+x_m+\rho_m e_V/2)} |u(x)-u(x-\rho_m e_v)| \, d\mathcal{H}^{n-1} \\
     &+ \|\hat{J}\|_{L^1} \int_{\partial Q^-_{\rho_m}(x_m) \cap (V+x_m+\rho_m e_V/2)} |u(x)-u(x-\rho_m e_v)| \, d\mathcal{H}^{n-1}.
 \end{align*}
Therefore, by adding and subtracting $z_2(x_0)$ and $z_1(x_0)$ in the above integrals, we can apply conditions \eqref{e:jumpconv1.1.1} and \eqref{e:jumpconv2.1.1} to infer
\begin{equation}
\label{e:defest1.1.1}
    \limsup_{\epsilon \to 0^+} \Lambda_\epsilon(\tilde{u}_\epsilon;B^V_1,Q_{\rho_m}(x_m)) \leq 4\|\hat{J}\|_{L^1} \tau \rho_m^{n-1}.
\end{equation}
In order to estimate the defect on $B^V_2$ we first notice that 
\[
\Lambda_\epsilon(\tilde{u}_\epsilon; B^V_2,Q_{\rho_m}(x_m)) = \Lambda_\epsilon(\tilde{u}_\epsilon \mathbbm{1}_{Q_{\rho_m}(x_m) \cup B^V_2}; A^V_1,Q_{\rho_m}(x_m)).
\]
By setting 
\[
 \Sigma := [(V +x_m - \rho_m e_V) \cap A_1^V] \cup [(V + x_m+ \rho_m e_V) \cap A_1^V] \EEE, 
\]
we have by construction $\tilde{u}_\epsilon \mathbbm{1}_{Q_{\rho_m}(x_m) \cup B^V_2}= 0$ in a neighborhood of $\Sigma$ for every $\epsilon$. Hence, by Proposition \ref{r:defuniform} the $\epsilon$-traces of $\tilde{u}_\epsilon \mathbbm{1}_{Q_{\rho_m}(x_m) \cup B^V_2}$ relative to $A^V_1$ and $Q_{\rho_m}(x_m)$ converge on $\Sigma$ to $0$. Proposition \ref{p:deficit} allows thus to infer 
\begin{equation}
\label{e:defest1.1.2}
\lim_{\epsilon \to 0^+}\Lambda_\epsilon(\tilde{u}_\epsilon ; B^V_2,Q_{\rho_m}(x_m)) = 0.
\end{equation}

The case $V = \xi_m^\bot$ does not present further difficulties and can be treated analogously. In particular we obtain the validity of \eqref{e:defest1.1.1} and \eqref{e:defest1.1.2} also in this case. Therefore, we can finally infer 
\begin{equation}
\label{e:defest1.1.3}
    \limsup_{\epsilon \to 0^+} \Lambda_\epsilon(\tilde{u}_\epsilon; A_1,Q_{\rho_m}(x_m)) \leq c \|\hat{J}\|_{L^1} \tau \rho_m^{n-1},
\end{equation}
for some dimensional constant $c>0$.

From elementary geometric considerations, if we denote by $\Sigma$ the $(n-2)$-dimensional skeleton of $Q_{\rho_m}(x_m)$, we see that $\Sigma = \partial Q_{\rho_m}(x_m) \cap A_2$ and that $\mathcal{H}^{n-1}(\Sigma)=0$. Therefore by virtue of Propositions \ref{p:etrace_convergence} and \ref{p:deficit}  it is not difficult to argue similarly as above to deduce that 
\begin{equation}
\label{e:defest1.1.4}
 \limsup_{\epsilon \to 0^+} \Lambda_\epsilon(\tilde{u}_\epsilon; A_2,Q_{\rho_m}(x_m))=0.
\end{equation}

By putting together \eqref{e:defest1.1.3}--\eqref{e:defest1.1.4} and recalling that $ \mathbb{R}^m \EEE \setminus Q_{\rho_m}(x_m) = A_1 \cup A_2$, we eventually infer that for every $m$ there holds
\begin{equation}
\label{e:defest1.1.5}
    \limsup_{\epsilon \to 0^+} \Lambda_\epsilon(\tilde{u}_\epsilon;  \mathbb{R}^m \EEE \setminus Q_{\rho_m}(x_m),Q_{\rho_m}(x_m)) \leq c \|\hat{J}\|_{L^1} \tau \rho_m^{n-1},
\end{equation}
for some dimensional constant $c>0$. We now insert condition \eqref{e:defest1.1.5} in \eqref{e:defest1.1.6} and by recalling \eqref{e:gli2}--\eqref{e:gli4} we continue the estimate as
\begin{align*}
    \liminf_{\epsilon \to 0^+} F_\epsilon(u_\epsilon;\Omega) &\geq \sum_{m} \liminf_{\epsilon \to 0^+} F_{\psi_\epsilon,x_m}(\mathring{v}_{\epsilon,\xi_m}; S_{\rho_m}^\epsilon) -c\|\hat{J}\|_{L^1}\tau\sum_m \rho_m^{n-1} \\
    &\geq F_0(u) -c\mathcal{H}^{n-1}(J_u)\|\hat{J}\|_{L^1}\tau,
\end{align*}
for some dimensional constant $c>0$. Eventually, the arbitrariness of $\tau$ yields \eqref{e:gli1000}.
\end{proof}

We are now in a position to prove the $\Gamma$-liminf inequality for general kernels.

\begin{theorem}[$\Gamma$-liminf with general kernel]
\label{t:gammaliminf2}
   Let $\Omega \subset  \mathbb{R}^m \EEE$ be open, let $J \colon  \mathbb{R}^m \EEE \to (0,\infty)$ be an even function satisfying \ref{K3} and let $W \colon  \mathbb{R}^m \EEE \times \mathbb{R} \to [0,\infty)$ satisfy \ref{H1}--\ref{H5}. Then, for every sequence $(u_\epsilon)\subset L^1(\Omega)$ such that $u_\epsilon \to u$ in $L^1(\Omega)$, we have
    \begin{equation}
        \label{e:gli1.1000}
        \liminf_{\epsilon \to 0^+} F_{\epsilon}(u_\epsilon;\Omega) \geq F_0(u).
    \end{equation}
\end{theorem}

\begin{proof}
As in the proof of Theorem \ref{t:gammaliminf1} we may assume with no loss of generality that $u_\epsilon \colon \Omega \to [-M,M]$.

For $N \in \mathbb{N}$ define the kernel $J^N \colon  \mathbb{R}^m \EEE \to [0,\infty)$ as $J \wedge N$, and let $F^N_\epsilon$ be the energy \eqref{e:functional} defined by means of the kernel $J^N$. In addition, by letting $\sigma^N$ be the surface tension \eqref{e:surftension} relative to the kernel $J^N$, we consider the functional $F_0^N$ as in \eqref{e:limit} with $\sigma$ replaced by $\sigma^N$. Since $J^N$ satisfies the hypothesis of theorem \ref{t:gammaliminf1}, we already know that, if $(u_\epsilon) \subset L^1(\Omega)$ is such that $u_\epsilon \to u$ in $L^1(\Omega)$ then $\liminf_{\epsilon} F_\epsilon(u_\epsilon;\Omega)\geq \liminf_{\epsilon} F^N_\epsilon(u_\epsilon;\Omega) \geq F^N_0(u)$. Therefore, if we show that 
    \begin{equation}
        \label{e:supeq}
        \sup_{N \in \mathbb{N}} \sigma^N(x,\xi) = \sigma(x,\xi), \ \ \text{for every $(x,\xi) \in  \mathbb{R}^m \EEE \times  \mathbb{S}^{m-1} \EEE$}
    \end{equation}
    we would infer the validity of the theorem by Beppo-Levi's convergence theorem.
   
   In order to show \eqref{e:supeq}, recall that
    \[
    \begin{split}
    \sigma^N(x,\xi) &= \inf_{\gamma \in X_{z_1(x)}^{z_2(x)}} \frac{1}{4}\int_{\mathbb{R}\times \mathbb{R}} J^{N\xi}(t-t') |\gamma(t)-\gamma(t')|^2 \, dtdt' + \int_\mathbb{R} W(x,\gamma(t)) \, dt \\
    &= \frac{1}{4}\int_{\mathbb{R}\times \mathbb{R}} J^{N\xi}(t-t') |\gamma^N(t)-\gamma^N(t')|^2 \, dtdt' + \int_\mathbb{R} W(x,\gamma^N(t)) \, dt,
    \end{split}
    \]
    for some $\gamma^N \in X_{z_1(x)}^{z_2(x)}$ which is increasing and such that $\gamma^N(t) \geq (z_1(x)+z_{2}(x))/2$ for a.e. $t \geq 0$ while $\gamma^N(t) < (z_1(x)+z_{2}(x))/2$ for a.e. $t < 0$. Since we already know that $\gamma^N \in BV(\mathbb{R})$ and $|D\gamma^N|= |z_{2}(x)-z_1(x)|$ for every $N$, we deduce that there exists a subsequence $N_m$ such that 
    \[
    \gamma^{N_m}(t) \to \gamma^0(t), \ \ \text{ for a.e. $t$} \ \ \text{ and } \ \ D\gamma^{N_m} \rightharpoonup D\gamma^0, \ \ \text{ weakly* as measures}.
    \]
     Furthermore, since $W(x,t) >0$ for  $t \in (z_1(x),z_2(x))$, by arguing as in Proposition \ref{t:eximin}, we deduce that $\gamma^0$ is an increasing function in $X_{z_1(x)}^{z_2(x)}$. By Fatou's lemma we find
    \[
    \begin{split}
     &\sup_{N \in \mathbb{N}}\frac{1}{4}\int_{\mathbb{R}\times \mathbb{R}} J^{N\xi}(t-t') |\gamma^{N}(t)-\gamma^{N}(t')|^2 \, dtdt' + \int_\mathbb{R} W(x,\gamma^{N}(t)) \, dt\\
   &\geq \liminf_{m \to \infty}\frac{1}{4}\int_{\mathbb{R}\times \mathbb{R}} J^{N_m\xi}(t-t') |\gamma^{N_m}(t)-\gamma^{N_m}(t')|^2 \, dtdt' + \int_\mathbb{R} W(x,\gamma^{N_m}(t)) \, dt\\
    &\ \ \ \ \ \ \ \ \ \ \ \geq \frac{1}{4}\int_{\mathbb{R}\times \mathbb{R}} J^{\xi}(t-t') |\gamma^0(t)-\gamma^0(t')|^2 \, dtdt' + \int_\mathbb{R} W(x,\gamma^0(t)) \, dt.
    \end{split}
    \]
    On the other hand, since every $\gamma \in X_{z_1(x)}^{z_{2}(x)}$ is admissible in the minimization problem for every $N \in \mathbb{N}$, we also have
    \begin{align*}
    &\frac{1}{4}\int_{\mathbb{R}\times \mathbb{R}} J^{\xi}(t-t') |\gamma(t)-\gamma(t')|^2 \, dtdt' + \int_\mathbb{R} W(x,\gamma(t)) \, dt \\
   &= \sup_{N \in \mathbb{N}}\frac{1}{4}\int_{\mathbb{R}\times \mathbb{R}} J^{N\xi}(t-t') |\gamma(t)-\gamma(t')|^2 \, dtdt' + \int_\mathbb{R} W(x,\gamma(t)) \, dt\\
    &\geq \sup_{N \in \mathbb{N}}\frac{1}{4}\int_{\mathbb{R}\times \mathbb{R}} J^{N\xi}(t-t') |\gamma^{N}(t)-\gamma^{N}(t')|^2 \, dtdt' + \int_\mathbb{R} W(x,\gamma^{N}(t)) \, dt.
    \end{align*}
    Hence, we have proved that 
    \begin{equation}
    \label{e:supeq1}
    F^\xi_{x}(\gamma^0)= \inf_{\gamma \in X_{z_1(x)}^{z_{2}(x)}} F^\xi_{x}(\gamma) = \sigma(x,\xi), \ \ \text{ for every $(x,\xi) \in  \mathbb{R}^m \EEE \times  \mathbb{S}^{m-1} \EEE$}.
    \end{equation}
    In addition, since $J^{N\xi} \leq  J^{(N+1)\xi}$, then we have as well $\sigma^N\leq \sigma^{N+1}$ for every $N \in \mathbb{N}$. In particular $\sup_N \sigma_N(x,\xi) = \lim_N \sigma_N(x,\xi)$ for every $(x,\xi) \in  \mathbb{R}^m \EEE \times  \mathbb{S}^{m-1} \EEE$. Moreover, being $\gamma^0$ admissible in the definition of $\sigma^N$ for every $N$, we also obtain
     \begin{align*}
        &\sigma^N(x,\xi)=\frac{1}{4}\int_{\mathbb{R}\times \mathbb{R}} J^{N\xi}(t-t') |\gamma^{N}(t)-\gamma^{N}(t')|^2 \, dtdt' + \int_\mathbb{R} W(x,\gamma^{N}(t)) \, dt\\
         &\leq  \frac{1}{4}\int_{\mathbb{R}\times \mathbb{R}} J^{N\xi}(t-t') |\gamma^{0}(t)-\gamma^{0}(t')|^2 \, dtdt' + \int_\mathbb{R} W(x,\gamma^{0}(t)) \, dt\leq  F^\xi_{x}(\gamma^0).
    \end{align*}
    By combining the above estimates, we infer
    \begin{align*}
       F^\xi_{x}(\gamma^0) \geq \sup_{N \in \mathbb{N}} \sigma^N(x,\xi) = \lim_{N \to \infty} \sigma^N(x,\xi) \geq F^\xi_{x}(\gamma^0), \ \ \text{ for $(x,\xi) \in  \mathbb{R}^m \EEE \times  \mathbb{S}^{m-1} \EEE$},
    \end{align*}
    which together with \eqref{e:supeq1} yields 
    \begin{equation}
        \sup_{N \in \mathbb{N}} \sigma^N(x,\xi)=\sigma(x,\xi), \ \ \text{ for every $(x,\xi) \in  \mathbb{R}^m \EEE \times  \mathbb{S}^{m-1} \EEE$}.
    \end{equation}
    This concludes the proof of the theorem.
\end{proof}

    \section{ Directional blow-up of the energy at a point \EEE}
    \label{s:dirblowup}
Throughout this section we fix $\xi \in  \mathbb{S}^{m-1} \EEE$ and an orthonormal basis $\{\xi_1,\dotsc,\xi_{n-1}\}$ of $\xi^\bot$. We further assume that the kernel $J$ satisfies \eqref{K1}, and we denote by $M$ any integer satisfying \eqref{e:boundwells123}.
For $\epsilon >0$ we consider the functional $F_{\epsilon,x_0}$ defined as
\[
F_{\epsilon,x_0}(u; B) := \frac{1}{4}\int_{ \mathbb{R}^m \EEE \times B} J(h) |u(x+h) -u(x)|^2 \, dhdx + \int_{B} W_{x_0}^\epsilon(x, u(x)) \, dx, 
\]
for every Borel set $B \subset \Omega$ and every measurable function $u \colon  \mathbb{R}^m \EEE \to \mathbb{R}$, 
where 
\begin{equation}
\label{e:wx0epsilon}
W_{x_0}^\epsilon(y+s\xi,t):=W(x_0+y +\epsilon s\xi, t), \ \ \text{ for every }(y,s,t) \in \xi^\bot \times \mathbb{R} \times \mathbb{R}.
\end{equation}

% In the first part of this section we construct The goal of this section is to construct an asymptotic {\color{blue} calibration} of the functionals $F_{\epsilon,x_0}$ as $\epsilon \to 0^+$. 
%With this we mean a family of one-dimensional functionals depending on the parameters $\rho >0$ and $(y,z) \in \xi^\bot \times V_\xi$. %Following our method, it turns out that the family $(\mathcal{G}^{yz}_{\rho\epsilon})$ is completely determined by the kernel $J$ and by the direction $\xi$. 

 \subsection{The family of non-local operators} 
 \label{sub:8.1}
 We consider the one-dimensional kernel $J^\xi \colon \mathbb{R} \to \mathbb{R}$ defined in \eqref{e:kerneljxi1} and we observe that we can directly verify from \eqref{K1} that \EEE
\begin{equation}
\label{e:intjxi}
\int_\mathbb{R} J^\xi(t)(1+|t|) \, dt < \infty .
\end{equation}

 We first construct a family of one-dimensional kernels depending on $z \in V_\xi$ in the following way. Let $K^z \colon \mathbb{R} \to [0,\infty)$ be defined as 
\begin{equation}
\label{e:kernz}
K^z(s):= J(sz)|s|^{n-1}, \ \  \text{ for every } z \in V_\xi,
\end{equation}
and notice that
\begin{equation}
\label{e:eqkj}
   \int_{V_\xi} K^z(s) \, dz =J^\xi(s), \ \ \text{ for every } s \in \mathbb{R}.
\end{equation}
Therefore, using \eqref{e:intjxi} we have also
\begin{equation}
\label{e:aeintk}
    \int_\mathbb{R} K^z(s)(1+|s|) \, ds < \infty, \ \ \text{ for $\mathcal{H}^{n-1}$-a.e. $z \in V_\xi$}.
\end{equation}

We start with a preliminary result which tells us that the non local operator in the energy $F_{\epsilon,x_0}$ can be decomposed by means of the one-dimensional and non-local operators associated to the kernels $(K^z)$.

 \begin{proposition}
 \label{p:decoperator}
     Let $\rho >0$ and let $u \colon \{y+t\xi \ | \ (y,t) \in C_\rho(0) \times\mathbb{R}\} \to \mathbb{R}$ be a measurable function such that 
     \[
    \int_{ \mathbb{R}^m \EEE}\int_{S_\rho^\epsilon} J(h) |u(x+h) -u(x)|^2 \, dxdh < \infty,
     \]
     and let $\overline{u} \colon  \mathbb{R}^m \EEE \to \mathbb{R}$ denote its $C_\rho(0)$-periodic extension.
     Then the following identity holds true for every $\epsilon >0$ 
     \begin{align*}
     \int_{ \mathbb{R}^m \EEE}&\bigg(\int_{S_\rho^\epsilon} J(h) |u(x+h) -u(x)|^2 \, dx\bigg)dh,\\
     &= \int_{V_\xi}\bigg( \int_{C_\rho(0)} \bigg(\int_{\mathbb{R} \times (-\frac{\rho}{2\epsilon},\frac{\rho}{2\epsilon})}K^z(s)|\overline{u}^{yz}(t+s) - \overline{u}^{yz}(t)|^2 \,dsdt \bigg) dy \bigg)dz.
     \end{align*}
 \end{proposition}
\begin{proof}
    We apply the change of variables $h=sz$ with $z \in V_\xi$ and $s \in \mathbb{R}$ to obtain
    \begin{align}
    \nonumber
        \int_{ \mathbb{R}^m \EEE}\bigg(\int_{S_\rho^\epsilon} J(h) &|\overline{u}(x+h) -\overline{u}(x)|^2 \, dx\bigg)dh\\
       \label{e:decoperator}
        &= \int_{V_\xi \times \mathbb{R}} \bigg( \int_{S_\rho^\epsilon} |\overline{u}(x+sz)-\overline{u}(x)|^2 \, dx\bigg)  |s|^{n-1}J(sz)\,dzds.
    \end{align}
Now fix $(z,s) \in V_\xi \times \mathbb{R}$. By writing $z= \overline{y}+\xi$ with $\tilde{y} \in V$ we can apply the change of variables induced by the map $\psi \colon \xi^\bot \times (-\rho/2\epsilon,\rho/2\epsilon) \to \xi^\bot \times (-\rho/2\epsilon,\rho/2\epsilon)$ defined as $\psi(y,t):= y+tz$, and by observing that the determinant of its jacobian is constantly equal to $1$, we infer
\begin{align*}
\int_{S_\rho^\epsilon} |\overline{u}(x+sz)-\overline{u}(x)|^2 \, dx &= \int_{-\rho/2\epsilon}^{\rho/2\epsilon} \bigg(\int_{C_\rho(0) - t\overline{y}} |\overline{u}(y+tz+sz)-\overline{u}(y+tz)|^2\, dy  \bigg)dt \\
&=\int_{-\rho/2\epsilon}^{\rho/2\epsilon} \bigg(\int_{C_\rho(0)} |\overline{u}(y+tz+sz)-\overline{u}(y+tz)|^2\, dy  \bigg)dt \\
&= \int_{C_\rho(0)} \bigg(\int_{-\rho/2\epsilon}^{\rho/2\epsilon} |\overline{u}^{yz}(t+s) - \overline{u}^{yz}(t)|^2 \,dt \bigg) dy,
\end{align*}
where in the last but one equality we used the $C_\rho(0)$-periodicity of $\overline{u}$. By recalling the definition of $K^z$ \eqref{e:kernz} we can insert the last equality in \eqref{e:decoperator} and conclude the proof.  
\end{proof}

 \subsection{The family of potentials}
 \label{sub:8.2}
 By letting $\mathcal{W}_{x_0}^{\rho\xi}$ and $a_{x_0}(y,\rho),b_{x_0}(y,\rho)$ be defined as in \eqref{e:infpotential} and \eqref{e:defaibi}, respectively, we further simplify the notation by dropping the dependence on $x_0$, namely, we set 
\[
\mathcal{W}^{\rho\xi}:= \mathcal{W}_{x_0}^{\rho\xi} \ \ \text{ and } \ \ a_\rho(y):=a_{x_0}(y,\rho) \ \ \text{ and } \ \ b_\rho(y):= b_{x_0}(y,\rho).
\]

For every sufficiently small $\rho>0$ and every measurable function $\gamma \colon \mathbb{R} \to \mathbb{R}$ we consider the one dimensional functional $\mathcal{F}^{\rho\xi}_{y}$ defined as
\begin{equation}
\label{e:gammay}
\mathcal{F}^{\rho\xi}_{y}(\gamma) := \frac{1}{4}\int_{\mathbb{R} \times \mathbb{R}}J^\xi(s) |\gamma(t+s)-\gamma(t)|^2 \, dtds + \int_\mathbb{R} \mathcal{W}^{\rho\xi}(y,\gamma(t)) \, dt, \ \ y \in \xi^\bot.
\end{equation}

  By applying Proposition \ref{t:eximin} we recall that for every $y \in \xi^\bot$ there exists an increasing function $\gamma_{y,\rho} \in X^{b_\rho(y)}_{a_\rho(y)}$ which is a minimizer of \eqref{e:gammay}. We are now in position to define the family of potentials. \EEE
  %which will be used in the {\color{blue} calibration} procedure.
\begin{definition}
\label{d:calipotentials}
    For $z \in V_\xi$, let $H^{z}_{\rho}$ be the operator given in \eqref{e:hoperator} relative to the kernel $K^z$ and wells $\{b_\rho(y),a_\rho(y)\}$. We define the potential $P^z_{\rho} \colon \xi^\bot \times \mathbb{R} \to [0,\infty)$ as 
    \[
    P^z_{\rho}(y,t):= 
    \begin{cases}
    (H^{z}_{\rho}\gamma_{y,\rho}^{-1})(t), &\text{ if } t \in (a_\rho(y),b_\rho(y)) \text{ and } y \in C_\rho(0) \\
    0, &\text{ if } t \in \mathbb{R} \setminus (a_\rho(y),b_\rho(y)) \text{ and } y \in C_\rho(0),
    \end{cases}
    \]
    and we complete its definition for $(y,t) \in (V \setminus C_\rho(0)) \times \mathbb{R}$ by taking its $C_\rho(0)$-periodic extension. 
\end{definition}

\begin{remark}
    \label{r:positivepotential}
    It is important to observe that, by virtue of our assumption $J >0$ and by virtue of Proposition \ref{p:unbounded}, the function $P_\rho^z(y,\cdot)$ is continuous and strictly positive on $(a_\rho(y),b_\rho(y))$.
\end{remark}

\begin{definition}
\label{d:onedimcalibration}
  For every measurable function $\gamma \colon \mathbb{R} \to \mathbb{R}$ consider the following non-local and one-dimensional operators 
\begin{align}
\textbf{K}^z(\gamma)&:= \frac{1}{4}\int_{\mathbb{R} \times \mathbb{R}} K^z(s) |\gamma(t+s) -\gamma(t)|^2 \, dsdt, \\
\textbf{K}^z_{\rho,\epsilon}(\gamma)&:= \frac{1}{4}\int_{\mathbb{R} \times (-\frac{\rho}{2\epsilon},\frac{\rho}{2\epsilon})} K^z(s) |\gamma(t+s) -\gamma(t)|^2 \, dsdt
\end{align}.
  For $(y,z) \in \xi^\bot \times V_\xi$ we define the following one-dimensional functionals 
\begin{align*}
&G^{yz}_{\rho}(\gamma) := \textbf{K}^z(\gamma) + \int_{\mathbb{R}} P^z_{\rho}(y,\gamma(t)) \, dt\\
&G^{yz}_{\rho,\epsilon}(\gamma) := \textbf{K}^z_{\rho,\epsilon}(\gamma) + \int_{(-\frac{\rho}{2\epsilon},\frac{\rho}{2\epsilon})} P^z_{\rho}(y,\gamma(t)) \, dt,
\end{align*}
    for every measurable function $\gamma \colon \mathbb{R} \to \mathbb{R}$. \EEE
\end{definition}

 The next proposition provides a lower-bound for the family of our original energies. \EEE

\begin{proposition}
\label{p:calibration}
    Let $u \colon\{y+t\xi \ | \ (y,t) \in C_\rho(0) \times\mathbb{R}\}\to \mathbb{R}$ be a measurable function and let $\overline{u} \colon  \mathbb{R}^m \EEE \to \mathbb{R}$ denote its $C_\rho(0)$-periodic extension. Then, for every $\rho >0$ small enough (depending on $x_0$) and for every $\epsilon >0$, there holds 
    \begin{equation}
    \label{e:calibration4}
        F_{\epsilon,x_0}(u;S_\rho^\epsilon) \geq \int_{V_\xi} \bigg( \int_{C_\rho(0)} G^{yz}_{\rho,\epsilon}(\overline{u}^{yz}) \, dy\bigg)dz.
    \end{equation}
\end{proposition}

\begin{proof}
    In view of Proposition \ref{p:decoperator}, it remains to prove that
    \begin{equation}
    \label{e:calibration1}
\int_{S^\epsilon_\rho} W_{x_0}^\epsilon(x,u(x)) \, dx \geq \int_{V_\xi \times C_\rho(0)} \bigg( \int_{(-\frac{\rho}{2\epsilon},\frac{\rho}{2\epsilon})} P^z_{\rho}(y,\overline{u}^{yz}(t)) \, dt \bigg) dydz.
    \end{equation}
To this purpose, we notice that the operator $H_{\rho,\xi}: Y \to C_0([a_\rho(y),b_\rho(y)])^+$ given in \eqref{e:hoperator} relative to the kernel $J^\xi$ and wells $\{b_\rho(y),a_\rho(y)\}$ satisfies by \eqref{e:eqkj} the equality
\begin{equation}
\label{e:opidentity}
\int_{V_\xi} H^{z}_{\rho}\, dz = H_{\rho,\xi}, \ \ \text{ for every  $\rho>0$}.
\end{equation}
Therefore, by Theorem \ref{t:charmin} we infer
\begin{align}
    \label{e:calibration2}
    &\mathcal{W}^{\rho\xi}(y,t) \geq \int_{V_\xi} P^z_{\rho}(y,t) \, dz, \ \ \text{ for } y \in C_\rho(0) \text{ and } t \in \mathbb{R} \\
    \label{e:calibration3}
    &\mathcal{W}^{\rho\xi}(y,t) = \int_{V_\xi} P^z_{\rho}(y,t) \, dz, \ \ \text{ for } y \in C_\rho(0) \text{ and } t \in \text{supp} \, \dot{\gamma}_{y,\rho}^{-1}.
\end{align}
Notice that, by the very definition of $W^\epsilon_{x_0}$ (see \eqref{e:wx0epsilon}), for $y \in C_\rho(0)$, $s \in (-\rho/2\epsilon,\rho/2\epsilon)$, and $t \in \mathbb{R}$, there holds 
\[
W_{x_0}^\epsilon(y+ s\xi,t)\geq \mathcal{W}^{\rho\xi}(y,t).
\]
Therefore, by denoting $\pi_\xi \colon  \mathbb{R}^m \EEE \to \xi^\bot$ the orthogonal projection,  \eqref{e:calibration2} yields
\begin{align*}
    &\int_{S^\epsilon_\rho} W_{x_0}^\epsilon(x,u(x)) \, dx =\int_{S^\epsilon_\rho} W_{x_0}^\epsilon(y+t\xi,u(y+t\xi)) \, dx
    \geq \int_{S^\epsilon_\rho} \mathcal{W}^{\rho\xi}(y,u(y+t\xi)) \, dydt\\ 
    &\quad\geq \int_{S^\epsilon_\rho} \bigg( \int_{V_\xi} P^z_{\rho}(y,u(y+t\xi)) \, dz\bigg) dtdy = \int_{V_\xi}\bigg(\int_{S^\epsilon_\rho} P^z_{\rho}(\pi_\xi(x),\overline{u}(x)) \, dx\bigg) dz\\ 
    &\quad= \int_{V_\xi}\bigg(\int_{(C_\rho(0)-t\pi_\xi(z)) \times (-\frac{\rho}{2\epsilon},\frac{\rho}{2\epsilon})} P^z_{\rho}(\pi_\xi(y+tz),\overline{u}(y+tz)) \, dydt\bigg) dz \\
    &\quad= \int_{V_\xi \times C_\rho(0)} \bigg( \int_{(-\frac{\rho}{2\epsilon},\frac{\rho}{2\epsilon})} P^z_{\rho}(y,\overline{u}^{yz}(t)) \, dt \bigg) dydz,
\end{align*}
where in the second-to-last equality we used that the Jacobian of the map $\psi\colon \xi^\bot \times (-\rho/2\epsilon,\rho/2\epsilon) \to \xi^\bot \times (-\rho/2\epsilon,\rho/2\epsilon)$ defined as $\psi(y,t):=y+tz$ has determinant constantly equal to $1$, whereas in the last equality we used the $C_\rho(0)$-periodicity of both $P^z_{\rho}$ and $\overline{u}$. 
This concludes the proof. 
\end{proof}

Now we explicitly compute the value of the the integral on the right hand-side of \eqref{e:calibration4} when computed on optimal profiles $\gamma_{\rho,y}$:
\begin{align*}
\int_{C_\rho(0)}G^{yz}_{\rho}(\gamma_{y,\rho}) \, dy &=  \int_{C_\rho(0)} \bigg(\mathbf{K}^z(\gamma_{y,\rho}) +  \int_{\mathbb{R}} P^z_{\rho}(y,\gamma_{y,\rho}(t))  \,dt\bigg)dy \\
&= \int_{C_\rho(0)} \bigg(\mathbf{K}^z(\gamma_{y,\rho}) +  \int_{a_\rho(y)}^{b_\rho(y)} P^z_{\rho}(y,s)  \,d\dot{\gamma}^{-1}_{y,\rho}(s)\bigg)dy,
\end{align*}
hence, by taking the integral with respect to $z$ and recalling \eqref{e:opidentity} in combination with Theorem \ref{t:charmin}, we finally have
\begin{equation}
\label{e:lowerbound}
   \int_{V_\xi}\bigg(\int_{C_\rho(0)} G^{yz}_{\rho}(\gamma_{y,\rho}) \, dy\bigg)dz =  \int_{C_\rho(0)} \mathcal{F}^{\rho\xi}_y(\gamma_{y,\rho}) \, dy.
\end{equation}

\begin{remark}
\label{r:kernelinvariance}
 Consider $\psi \colon \xi^\bot \to \xi^\bot$ to be any $C^1$-diffeomorphism. Then, by defining $J_\psi \colon  \mathbb{R}^m \EEE \to (0,\infty)$ as $J(y +t\xi):=J(\psi(y)+t\xi) |\det(\nabla \psi(y))|$ and the corresponding kernel $K^z_\psi$ exactly as in \eqref{e:kernz}, we have
\begin{equation}
\label{e:kernelinvariance}
\int_{V_\xi} K^z_\psi(s)\, dz = \int_{V_\xi} K^z(s)\, dz= J^\xi(s), \ \ \text{ for every }s \in \mathbb{R}.
\end{equation}
Furthermore, if we denote by $(G^{yz}_{\rho,\psi})$ the family constructed with the kernel $J_\psi$, thanks to \eqref{e:opidentity}, equality \eqref{e:kernelinvariance} shows that \eqref{e:lowerbound} still remains valid, namely,
\begin{equation}
\label{e:kernelinvariance1.472}
\int_{V_\xi}\bigg(\int_{C_\rho(0)} G^{yz}_{\rho,\psi}(\gamma_{y,\rho}) \, dy\bigg)dz =  \int_{C_\rho(0)} \mathcal{F}^{y\xi}_\rho(\gamma_{y,\rho}) \, dy. \EEE
\end{equation}
\end{remark}

\subsection{The double liminf inequality}
\label{sub:doubleliminf}

 We begin by introducing suitable further modifications of  $C_\rho(0)$-periodic extensions depending on $z \in V_\xi$ and $\epsilon >0$.  \EEE

\begin{definition}
Given a family $(u_\epsilon)_\epsilon$ such that $u_\epsilon \colon S_\rho^\epsilon \to \mathbb{R}$, we define $\hat{u}_{\epsilon,z} \colon  \mathbb{R}^m \EEE \to \mathbb{R}$ to be the $C_\rho(0)$-periodic modification at scale $\epsilon$ and in the direction $z \in V_\xi$ as
\begin{equation}
\label{e:modification}
    \hat{u}_{\epsilon,z}(x) :=
    \begin{cases}
        \overline{u}_\epsilon(x), &\text{ if }  x=y+tz \in S_\rho^\xi \EEE  \\
        \overline{a}_\rho(y) \EEE &\text{ if }   x = y+tz, \ (y,t) \in \xi^\bot \times (-\infty,-\rho/2\epsilon) \EEE \\
        \overline{b}_\rho(y) \EEE &\text{ if }  x = y+tz, \ (y,t) \in \xi^\bot \times (\rho/2\epsilon,\infty) \EEE.
    \end{cases}
\end{equation}
\end{definition}
By construction $\hat{u}_{\epsilon,z}$ is $C_\rho(0)$-periodic. Notice that we do not explicitly write in the notation the dependence of $\hat{u}_{\epsilon,z}$ from $\rho$.

 We recall the definition of oscillation of a function \EEE.

\begin{definition}[Oscillation]
    Given a measurable function $v \colon C_\rho(0) \to \mathbb{R}$ we define the oscillation of $v$ in $C_\rho(0)$ as
    \begin{equation}
        \emph{osc}(v;C_\rho(0)) := \esup_{y \in C_\rho(0)} v(y) - \einf_{y \in C_\rho(0)} v(y).
    \end{equation}
\end{definition}

The next result shows that the difference $\textbf{K}^z(\hat{u}^{yz}_{\epsilon,z})-\textbf{K}^z(\mathring{u}^{yz}_{\epsilon,\xi})$  can be controlled asymptotically as $\epsilon \to 0^+$.

\begin{proposition}
\label{p:oscbound1}
    Let $x_0 \in \Omega$, let $(u_\epsilon)_\epsilon$ be a family of measurable functions $u_\epsilon \colon S_\rho^\epsilon \to [-M,M]$. Then, for every $\rho >0$ small enough (depending on $x_0$), and for $\mathcal{H}^{n-1}$-a.e. \EEE $z \in V_\xi$,  there holds
    \begin{equation}
    \label{e:oscbound}
    \begin{split}
\limsup_{\epsilon \to 0+}  \mint_{C_\rho(0)} |\textbf{\emph{K}}^z(\hat{u}^{yz}_{\epsilon,z})-\textbf{\emph{K}}^z(\mathring{u}^{yz}_{\epsilon,\xi})| \, dy &\leq C  \rho^{\alpha} \EEE( \|\hat{K}^z\|_{L^1} + \rho^{\alpha} \EEE \|K^z\|_{L^1}),
\end{split}
    \end{equation}
     where $C:= C(M,Z,\lfloor z_1\rfloor_{C^{0,\alpha}},\lfloor z_2\rfloor_{C^{0,\alpha}})>0$ \EEE.
\end{proposition}
\begin{proof}
 In order to simplify the notation we assume $x_0=0$. By using the identity $a^2-b^2 = (a-b)(a+b)$ ($a,b \in \mathbb{R}$) we obtain the following pointwise estimate for $y \in C_\rho(0)$ \EEE : 
    \begin{align*}
     &|\textbf{K}^z(\hat{u}^{yz}_{\epsilon,z})-\textbf{K}^z(\mathring{u}^{yz}_{\epsilon,\xi})| \\
    % &\leq \int_{(-\infty,- \frac{\rho}{2\epsilon}) \times (- \frac{\rho}{2\epsilon},\frac{\rho}{2\epsilon})}  K^z(t'-t) ||\hat{u}^{yz}_{ij\epsilon,z}(t') - \overline{u}^{yz}(t)|^2 -|\mathring{u}^{yz}_{ij\epsilon,\xi}(t') - \overline{u}^{yz}(t)|^2|\,dt'dt\\
    % &+ \int_{(\frac{\rho}{2\epsilon}, \infty) \times (- \frac{\rho}{2\epsilon},\frac{\rho}{2\epsilon})}  K^z(t'-t) ||\hat{u}^{yz}_{ij\epsilon,z}(t') - \overline{u}^{yz}(t)|^2 -|\mathring{u}^{yz}_{ij\epsilon,\xi}(t') - \overline{u}^{yz}(t)|^2|\,dt'dt\\
    % &+ \int_{(\frac{\rho}{2\epsilon}, \infty) \times (\frac{\rho}{2\epsilon}, \infty)}  K^z(t'-t) ||\hat{u}^{yz}_{ij\epsilon,z}(t') - \hat{u}^{yz}_{ij\epsilon,z}(t)|^2 -|\mathring{u}^{yz}_{ij\epsilon,\xi}(t') - \mathring{u}^{yz}_{ij\epsilon,\xi}(t)|^2|\,dt'dt\\
    % &+ \int_{(-\infty,\frac{\rho}{2\epsilon}) \times (-\infty,\frac{\rho}{2\epsilon})}  K^z(t'-t) ||\hat{u}^{yz}_{ij\epsilon,z}(t') - \hat{u}^{yz}_{ij\epsilon,z}(t)|^2 -|\mathring{u}^{yz}_{ij\epsilon,\xi}(t') - \mathring{u}^{yz}_{ij\epsilon,\xi}(t)|^2|\,dt'dt\\
    % &+\int_{(-\infty,-\frac{\rho}{2\epsilon}) \times ( \frac{\rho}{2\epsilon},\infty)}  K^z(t'-t) ||\hat{u}^{yz}_{ij\epsilon,z}(t') - \hat{u}^{yz}_{ij\epsilon,z}(t)|^2 -  |\hat{u}^{yz}_{i\epsilon,\xi}(t') - \mathring{u}^{yz}_{ij\epsilon,\xi}(t)|^2|\,dt'dt\\
    % &+\int_{(\frac{\rho}{2\epsilon},\infty) \times (-\infty,- \frac{\rho}{2\epsilon})}  K^z(t'-t) ||\hat{u}^{yz}_{ij\epsilon,z}(t') - \hat{u}^{yz}_{ij\epsilon,z}(t)|^2 -  |\mathring{u}^{yz}_{ij\epsilon,\xi}(t') - \mathring{u}^{yz}_{ij\epsilon,\xi}(t)|^2|\,dt'dt\\
    &\leq c \int_{(-\infty,- \frac{\rho}{2\epsilon}) \times(- \frac{\rho}{2\epsilon},\frac{\rho}{2\epsilon})}  K^z(t'-t) |\overline{a}_\rho(y) - \overline{z}_{1}(\frac{y}{t'}+\pi_\xi(z))|\,dt'dt\\
    &+ c\int_{(\frac{\rho}{2\epsilon},\infty) \times (- \frac{\rho}{2\epsilon},\frac{\rho}{2\epsilon})}  K^z(t'-t) |\overline{b}_\rho(y) - \overline{z}_{2}(\frac{y}{t'}+\pi_\xi(z))|\,dt'dt\\
&+\int_{( \frac{\rho}{2\epsilon},\infty) \times( \frac{\rho}{2\epsilon},\infty)}  K^z(t'-t) |\overline{z}_{2}(\frac{y}{t'}+\pi_\xi(z)) - \overline{z}_{2}(\frac{y}{t}+\pi_\xi(z))|^2\,dt'dt\\
&+\int_{(-\infty, -\frac{\rho}{2\epsilon}) \times(-\infty, -\frac{\rho}{2\epsilon})}  K^z(t'-t) |\overline{z}_{1}(\frac{y}{t'}+\pi_\xi(z)) - \overline{z}_{1}(\frac{y}{t}+\pi_\xi(z))|^2\,dt'dt\\
    &+c\int_{(-\infty,-\frac{\rho}{2\epsilon})\times ( \frac{\rho}{2\epsilon},\infty)}  K^z(t'-t)\,dtdt'
    \end{align*}
    where the constant $c$ depends on $M$ and $Z$ ($\rho \ll 1$), and where $\pi_\xi \colon  \mathbb{R}^m \EEE \to \xi^\bot$ denotes the orthogonal projection. Now we estimate separately for every $y \in C_\rho(0)$
\begin{align}
\nonumber
    |\overline{b}_\rho(y) - \overline{z}_{2}(\frac{y}{t'}+\pi_\xi(z))| &\leq |\overline{b}_\rho(y) - \overline{b}_\rho(0)|+|\overline{z}_2(0)-\overline{b}_\rho(0)| \\
    \label{e:oscbound1}
    &+|\overline{z}_2(\frac{y}{t'}+\pi_\xi(z)) - \overline{z}_2(0)|\\
    \nonumber
    & \leq \text{osc}(b_\rho,C_\rho(0)) +2\,\text{osc}(z_2,Q_\rho)\\
    &\leq 3\,\text{osc}(z_2,Q_\rho) \leq 3 n^{\frac{\alpha}{2}} \lfloor z_2\rfloor_{C^{0,\alpha}} \rho^{\alpha}   
\end{align}
and analogously for every $y \in C_\rho(0)$
\begin{equation}
\label{e:oscbound2}
    |\overline{a}_\rho(y) - \overline{z}_{1}(\frac{y}{t'}+\pi_\xi(z))| \leq 3 n^{\frac{\alpha}{2}} \lfloor z_1\rfloor_{C^{0,\alpha}} \rho^{\alpha} .  
\end{equation}
    Additionally, for every $z \in V_\xi$ for which $\pi_\xi(z) \notin \{\sum_{i=1}^{n-1} m_i \xi_i/2 \ | \ m \in \rho\mathbb{Z}^{n-1}  \} $ (hence for $\mathcal{H}^{n-1}$-a.e. $z \in V_\xi$) there exists a small enough threshold $0 < \epsilon_z <1 $ and there exists $m \in \rho \mathbb{Z}^{n-1}$ such that $y+\pi_\xi(z)$ belongs to $C_\rho(0) + \sum_{i=1}^{n-1}m_i\xi_i$ whenever $|y| < \epsilon_z$ (loosely speaking they belong to the same copy of $C_\rho(0)$ obtained by translating of $\sum_i m_i \xi_i$). This means that, for $\mathcal{H}^{n-1}$-a.e. $z \in V_\xi$ we find $0 <\epsilon_z <1$ such that for every $0 < \epsilon <\epsilon_z\rho/2$, $y \in C_\rho(0)$, and $t,t' \in (\frac{\rho}{2\epsilon},\infty)$, we can estimate \EEE
\begin{align}
\nonumber
    |\overline{z}_{2}(\frac{y}{t'}+\pi_\xi(z)) - \overline{z}_{2}(\frac{y}{t}+\pi_\xi(z))| &= |z_{2}(\frac{y}{t'}+\pi_\xi(z)) - z_{2}(\frac{y}{t}+ \pi_\xi(z))| \\
    \label{e:oscbound3}
    &\leq  \lfloor z_2\rfloor_{C^{0,\alpha}} |y|^{\alpha}  \bigg| \frac{1}{t}-\frac{1}{t'}\bigg|^{\alpha}.
\end{align}
Analogously, for every $0 < \epsilon <\epsilon_z\rho/2$, $y \in C_\rho(0)$, and $t,t' \in (\frac{\rho}{2\epsilon},\infty)$, we have
\begin{equation}
\label{e:oscbound4}
     |\overline{z}_{1}(\frac{y}{t'}+\pi_\xi(z)) - \overline{z}_{1}(\frac{y}{t}+\pi_\xi(z))| \leq  \lfloor z_1\rfloor_{C^{0,\alpha}} |y|^{\alpha}  \bigg| \frac{1}{t}-\frac{1}{t'}\bigg|^{\alpha}. 
\end{equation}

By recalling the definition of the defect functional \eqref{e:deficit}, we can thus write for every $y \in C_\rho(0)$, for $\mathcal{H}^{n-1}$-a.e. $z \in V_\xi$, and for every $0 < \epsilon < \epsilon_z \rho/2$ 
\begin{align*}
|\textbf{K}^z(\hat{u}^{yz}_{\epsilon,z})-\textbf{K}^z(\mathring{u}^{yz}_{\epsilon,\xi})| &\leq C'\,\rho^{\alpha} \, \Lambda_\epsilon(\mathbbm{1}_{(-\frac{\rho}{2},\frac{\rho}{2})},(-\infty,-\rho/2),(-\rho/2,\rho/2)) \\
    &+ C' \, \rho^{\alpha} \, \Lambda_\epsilon(\mathbbm{1}_{(-\frac{\rho}{2},\frac{\rho}{2})},(\rho/2,\infty),(-\rho/2,\rho/2)) \\
    &+ C'\|K^z\|_{L^1(\mathbb{R})}\rho^{2\alpha} \int_{\{|t| \geq 1\}} \frac{dt}{t^{2\alpha}}\\ 
    &+C' \Lambda_\epsilon(\mathbbm{1}_{(\mathbb{R} \setminus (-\frac{\rho}{2},\frac{\rho}{2}))},(-\infty,-\rho/2),(\rho/2,\infty)),
\end{align*}
where $C':= C'(M,Z,\lfloor z_1\rfloor_{C^{0,\alpha}},\lfloor z_2\rfloor_{C^{0,\alpha}})>0$.
In particular, the above chain of inequalities tells us that the family of integrands in the left-hand side of \eqref{e:oscbound} is uniformly dominated with respect to $y$ for every $\epsilon$ small enough (depending on $z \in V_\xi$). Moreover, from Proposition \ref{p:deficit} we have
   \begin{align}
   \label{e:deficit1}
    &\limsup_{\epsilon} \Lambda_\epsilon(\mathbbm{1}_{(-\frac{\rho}{2},\frac{\rho}{2})},(\rho/2,\infty),(-\rho/2,\rho/2)) \leq \frac{1}{2}\|\hat{K}^z\|_{L^1(\mathbb{R})} \\
    \label{e:deficit2}
    &\limsup_{\epsilon}\Lambda_\epsilon(\mathbbm{1}_{(-\frac{\rho}{2},\frac{\rho}{2})},(-\infty,-\rho/2),(-\rho/2,\rho/2)) \leq \frac{1}{2}\|\hat{K}^z\|_{L^1(\mathbb{R})}\\
    \label{e:deficit3}
    &\limsup_{\epsilon}\Lambda_\epsilon(\mathbbm{1}_{(\mathbb{R} \setminus (-\frac{\rho}{2},\frac{\rho}{2}))},(-\infty,-\rho/2),(\rho/2,\infty))=0,
    \end{align}
    where we have set $\Sigma=\{\frac{\rho}{2}\}$, $\Sigma=\{-\frac{\rho}{2}\}$, and $\Sigma=\{0\}$ in \eqref{e:deficit1}, \eqref{e:deficit2}, and \eqref{e:deficit3}, respectively. 

We are thus in position to apply the limsup version of Fatou's lemma in the left-hand side of \eqref{e:oscbound} and to infer the validity of the proposition.
\end{proof}

We further show that for sequences $(u_\epsilon)$ converging in $L^1$ the difference $\textbf{K}^z(\mathring{u}^{yz}_{\epsilon,\xi})-\textbf{K}^z_{\rho,\epsilon}(\mathring{u}^{yz}_{\epsilon,\xi})$ can be controlled asymptotically as $\epsilon \to 0^+$.

\begin{proposition}
\label{p:oscbound2}
    Let $x_0 \in \Omega$, let $(u_\epsilon)_\epsilon$ be a family of measurable functions $u_\epsilon \colon \Omega \to [-M,M]$, and let $v_{\epsilon} \colon S_\rho^{\ep} \to [-M,M]$ be defined as $v_\epsilon(y+t\xi):=u_\epsilon(x_0+ y+\epsilon t \xi )$. Assume that
    \begin{equation}
    \label{e:l1conv149}
        u_\epsilon \to u, \ \ \text{ in $L^1(\Omega)$ as $\epsilon \to 0^+$}.
    \end{equation}
    Then, for $\mathcal{H}^1$-a.e. \EEE sufficiently small $\rho >0$ (depending on $x_0$), and for $\mathcal{H}^{n-1}$-a.e. \EEE $z \in V_\xi$, there holds
    \begin{equation}
    \label{e:oscbound1234}
    \begin{split}
\limsup_{\epsilon \to 0+}  \mint_{C_\rho(0)} |\textbf{\emph{K}}^z(\mathring{v}^{yz}_{\epsilon,\xi})-\textbf{\emph{K}}_{\rho,\epsilon}^z(\mathring{v}^{yz}_{\epsilon,\xi})| \, dy &\leq C \|\hat{K}^z\|_{L^1} \mint_{\partial Q^-_\rho(x_0)}|u(x)-z_1(x_0)| \, d\mathcal{H}^{n-1} \\
&+ C \|\hat{K}^z\|_{L^1} \mint_{\partial Q^+_\rho(x_0)}|u(x)-z_2(x_0)| \, d\mathcal{H}^{n-1} \\
&+C  \rho^{2\alpha} \EEE(\|K^z\|_{L^1}+  \|\hat{K}^z\|_{L^1} \EEE ), 
\end{split}
    \end{equation}
    where $\partial Q^\pm_\rho(x_0):= \partial Q_\rho(x_0) \cap \{x \ |\ \pm(x-x_0) \cdot \xi >0 \}$ and $C:=C(M,Z,\lfloor z_1\rfloor_{C^{0,\alpha}},\lfloor z_2\rfloor_{C^{0,\alpha}})>0$. 
\end{proposition}

\begin{proof}
In order to simplify the notation, we assume that $x_0=0$ and $\{y+t\xi \ | \ y \in C_\rho(0), \ t \in (-1/2,1/2)\} \subset \Omega$ for every $\rho >0$ small enough.
     
     We notice that, from the very definition of $\mathring{v}_\epsilon$, we have for every $y \in C_\rho(0)$
    \begin{align*}
    \textbf{K}^z(\mathring{v}^{yz}_{\epsilon,\xi})-\textbf{K}^z_{\rho,\epsilon}(\mathring{v}^{yz}_{\epsilon,\xi}) &= 2\int_{\mathbb{R} \times (\mathbb{R} \setminus (-\frac{\rho}{2\epsilon},\frac{\rho}{2\epsilon}))} K^z(t'-t) |\mathring{v}^{yz}_{\epsilon,\xi}(t') - \mathring{v}^{yz}_{\epsilon,\xi}(t)|^2 \, dt'dt.
    \end{align*}
    We split the integration domain as
    \begin{align*}
    \mathbb{R} \times (\mathbb{R} \setminus (-\rho/2\epsilon,\rho/2\epsilon)) = (\mathbb{R} \setminus (-\rho/2\epsilon,\rho/2\epsilon))^2
    &\cup (-\rho/2\epsilon,\rho/2\epsilon) \times (-\infty,-\rho/2\epsilon,) \\
    &\cup (-\rho/2\epsilon,\rho/2\epsilon) \times (\rho/2\epsilon,\infty).
    \end{align*}
    Arguing exactly as in the proof of Proposition \ref{p:oscbound1}, we infer the existence of a constant $C':= C'(M,\lfloor z_1\rfloor_{C^{0,\alpha}},\lfloor z_2\rfloor_{C^{0,\alpha}})>0$ such that for $\mathcal{H}^{n-1}$-a.e. $z \in V_\xi$
    \begin{equation}
    \label{e:defest1118}
       \limsup_{\epsilon \to 0+}  \mint_{C_\rho(0)} \bigg(\int_{(\mathbb{R} \setminus (-\frac{\rho}{2\epsilon},\frac{\rho}{2\epsilon}))^2} K^z(t'-t) |\mathring{v}^{yz}_{\epsilon,\xi}(t') - \mathring{v}^{yz}_{\epsilon,\xi}(t)|^2 \, dt'dt\bigg) dy \leq C' \rho^{2\alpha}\|K^z\|_{L^1}.
    \end{equation}
     For the remaining term, by performing the change of variables $s=\epsilon^{-1}t$ and $s'=\epsilon^{-1}t'$ and by using the definition of $\mathring{v}_\epsilon$ we write for every $y \in C_\rho(0)$
    \begin{align*}
        &\int_{(-\frac{\rho}{2\epsilon},\frac{\rho}{2\epsilon}) \times (-\infty,-\frac{\rho}{2\epsilon})} K^z(s'-s) |\mathring{v}^{yz}_{\epsilon,\xi}(s') - \mathring{v}^{yz}_{\epsilon,\xi}(s)|^2 \, ds'ds \\
        % &=\frac{1}{\ep}\int_{(-\frac{\rho}{2},\frac{\rho}{2}) \times (\mathbb{R} \setminus (-\frac{\rho}{2},\frac{\rho}{2}))} K^z_{\ep}(t'-t) |\mathring{v}^{yz}_{ij\ep,\xi}(\ep^{-1}t') - \mathring{v}^{yz}_{ij\ep,\xi}(\ep^{-1}t)|^2 \, dt'dt \\
        % &=\frac{1}{\ep}\int_{(-\frac{\rho}{2},\frac{\rho}{2}) \times (-\infty,-\frac{\rho}{2})} K^z_\ep(t'-t) |\overline{u}_\ep(y + \frac{t'}{\ep}\pi_\xi(z)+t'\xi) - \overline{z}_i(x_0+ \frac{\ep y}{t}+\pi_\xi(z))|^2 \, dt'dt \\
        &\leq \frac{3}{\epsilon}\int_{(-\frac{\rho}{2},\frac{\rho}{2}) \times (-\infty,-\frac{\rho}{2})} K^z_\epsilon(t'-t) |\overline{u}_\epsilon(y +\frac{t'}{\epsilon}\pi_\xi(z)+t'\xi) - \overline{u}(y +\frac{t'}{\epsilon}\pi_\xi(z) +t'\xi)|^2 \, dt'dt \\
        &+ \frac{3}{\epsilon}\int_{(-\frac{\rho}{2},\frac{\rho}{2}) \times (-\infty,-\frac{\rho}{2})} K^z_\epsilon(t'-t) |\overline{u}(y+\frac{t'}{\epsilon}\pi_\xi(z) +t'\xi)-z_1(0)|^2 \, dt'dt \\
        &+ \frac{3}{\epsilon}\int_{(-\frac{\rho}{2},\frac{\rho}{2}) \times (-\infty,-\frac{\rho}{2})} K^z_\epsilon(t'-t) |\overline{z}_1( \frac{\epsilon y}{t}+\pi_\xi(z)) - z_1(0)|^2 \, dt'dt. 
    \end{align*}
    By taking the integral with respect to $y$ and by exploiting the $C_\rho(0)$-periodicity we have
    \begin{align*}
        &\int_{C_\rho(0)} \bigg(\int_{(-\frac{\rho}{2\epsilon},\frac{\rho}{2\epsilon}) \times (-\infty,-\frac{\rho}{2\epsilon})} K^z(t'-t) |\mathring{v}^{yz}_{\epsilon,\xi}(t') - \mathring{v}^{yz}_{\epsilon,\xi}(t)|^2 \, dt'dt\bigg) dy \\
        &\leq \int_{C_\rho(0)}\bigg(\frac{3}{\epsilon}\int_{(-\frac{\rho}{2},\frac{\rho}{2}) \times (-\infty,-\frac{\rho}{2})} K^z_\epsilon(t'-t) |\overline{u}_\epsilon(y +t'\xi) - \overline{u}(y +t'\xi)|^2 \, dt'dt\bigg)dy \\
        &+ \int_{C_\rho(0)} \bigg(\frac{3}{\epsilon}\int_{(-\frac{\rho}{2},\frac{\rho}{2}) \times (-\infty,-\frac{\rho}{2})} K^z_\epsilon(t'-t) |\overline{u}(y +t'\xi)-z_1(0)|^2  \, dt'dt\bigg)dy \\
        &+ \int_{C_\rho(0)} \bigg(\frac{3}{\epsilon}\int_{(-\frac{\rho}{2},\frac{\rho}{2}) \times (-\infty,-\frac{\rho}{2})} K^z_\epsilon(t'-t) |\overline{z}_1(\frac{\epsilon y}{t}+\pi_\xi(z)) -z_1(0)|^2  \, dt'dt\bigg)dy 
  \end{align*}
    From \eqref{e:l1conv149}, we deduce that, by letting $f_\epsilon \colon (-\frac{1}{2},\frac{1}{2}) \to \mathbb{R}$ be defined as
    \[
    f_\epsilon(t'):=\int_{C_\rho(0)} |\overline{u}_\epsilon(y +t'\xi) - \overline{u}(y +t'\xi)|^2\, dy,
    \]
    then $f_\epsilon \to 0$ in $L^1$ as $\epsilon \to 0^+$. Therefore, by Proposition \ref{p:etrace_convergence} we infer that for a.e. (sufficiently small) $\rho>0$ the $\epsilon_\ell$-traces of $f_{\epsilon_\ell}$ relative to $(-\frac{1}{2},\frac{1}{2})$ converge on $\{-\frac{\rho}{2}\} \cup \{\frac{\rho}{2}\}$ to $0$. By using that $(-\frac{\rho}{2},\frac{\rho}{2}) \subset (-\frac{1}{2},\frac{1}{2})$, from the monotonicity of the $\epsilon$-traces convergence with respect to set inclusion, we deduce also that the $\epsilon_\ell$-traces of $f_{\epsilon_\ell}$ relative to $(-\frac{\rho}{2},\frac{\rho}{2})$ converge on $\{-\frac{\rho}{2}\} \cup \{\frac{\rho}{2}\}$ to $0$, whenever $(\epsilon_\ell)$ is infinitesimal. We are thus in position to apply Proposition \ref{p:deficit} and infer 
    \begin{equation}
    \label{e:defest1111}
        \limsup_{\ell} \Lambda_\ell\left(f_{\epsilon_\ell} \mathbbm{1}_{(-\frac{\rho}{2},\frac{\rho}{2})}; (-\frac{\rho}{2},\frac{\rho}{2}),(-\infty,-\frac{\rho}{2})\right) =0, \ \ \text{ for a.e. $\rho>0$}
    \end{equation}
    Finally, by \eqref{e:defest1111}, for almost every $\rho>0$ we write
    \begin{align*}
        &\limsup_{\ell} \int_{C_\rho(0)}\bigg(\frac{3}{\epsilon_\ell}\int_{(-\frac{\rho}{2},\frac{\rho}{2}) \times (-\infty,-\frac{\rho}{2})} K^z_{\epsilon_\ell}(t'-t) |\overline{u}_{\epsilon_\ell}(y +t'\xi) - \overline{u}(y +t'\xi)|^2 \, dt'dt\bigg)dy \\
        &= \limsup_{\ell} \frac{3}{\epsilon_\ell}\int_{(-\frac{\rho}{2},\frac{\rho}{2}) \times (-\infty,-\frac{\rho}{2})}K^z_{\epsilon_\ell}(t'-t) |f_{\epsilon_\ell}\mathbbm{1}_{(-\frac{\rho}{2},\frac{\rho}{2})}(t') - f_{\epsilon_\ell}\mathbbm{1}_{(-\frac{\rho}{2},\frac{\rho}{2})}(t)|dt'dt=0.
    \end{align*}
    Hence, in view of the arbitrariness of the subsequence $(\epsilon_\ell)$ we infer
    \begin{equation}
        \label{e:defest1112}
        \limsup_{\epsilon} \int_{C_\rho(0)}\bigg(\frac{3}{\epsilon}\int_{(-\frac{\rho}{2},\frac{\rho}{2}) \times (-\infty,-\frac{\rho}{2})} K^z_{\epsilon}(t'-t) |\overline{u}_{\epsilon}(y +t'\xi) - \overline{u}(y +t'\xi)|^2 \, dt'dt\bigg)dy=0.
    \end{equation}

    With the very same technique, Propositions \ref{r:defuniform} and \ref{p:deficit} entail that for every $\rho >0$ there holds
     \begin{align}
      \nonumber
          &\limsup_{\epsilon \to 0^+}\int_{C_\rho(0)} \bigg(\frac{3}{\epsilon}\int_{(-\frac{\rho}{2},\frac{\rho}{2}) \times (-\infty,-\frac{\rho}{2})} K^z_\epsilon(t'-t) |\overline{u}(y+t'\xi)-z_1(0)|^2  \, dt'dt\bigg)dy \\
          \label{e:defest1113}
          & \leq C'' \|\hat{K}^z\|_{L^1} \int_{C_\rho(0)}|u(y-\frac{\rho}{2}\xi)-z_1(0)| \, dy,
          \end{align}
and by using that $\int_{C_\rho(0)} |\overline{z}_1(\frac{\epsilon y}{t} + \pi_\xi(z)) - z_1(0)|^2\, dy =: f_\epsilon(t) \to \rho^{n-1}|\overline{z}_1(\pi_\xi(z))-z_1(0)|^2$ uniformly in $t \in (-\infty,-\rho/2)$ as $\epsilon \to 0^+$, again the very same technique gives
\begin{align}
\label{e:defest1115}
   \limsup_{\epsilon \to 0^+} \mint_{C_\rho(0)} \bigg(&\frac{3}{\epsilon}\int_{(-\frac{\rho}{2},\frac{\rho}{2}) \times (-\infty,-\frac{\rho}{2})} K^z_\epsilon(t'-t) |\overline{z}_1(\frac{\epsilon y}{t}+\pi_\xi(z)) -z_1(0)|^2  \, dt'dt\bigg)dy \\
   \nonumber
   &\leq C'' \|\hat{K}^z\|_{L^1}|\overline{z}_1(\pi_\xi(z)) - z_1(0)|^2\\
   \nonumber
   &\leq C'' \|\hat{K}^z\|_{L^1} \lfloor z_1 \rfloor_{C^{0,\alpha}}^2 \rho^{2\alpha}.  
\end{align}
\EEE
By combining \eqref{e:defest1112}--\eqref{e:defest1115}, we infer that for a.e. $\rho >0$ (sufficiently small)
\begin{align}
    \nonumber
    &\limsup_{\epsilon \to 0^+}\mint_{C_\rho(0)} \bigg(\int_{(-\frac{\rho}{2\epsilon},\frac{\rho}{2\epsilon}) \times (-\infty,-\frac{\rho}{2\epsilon})} K^z(t'-t) |\mathring{v}^{yz}_{\epsilon,\xi}(t') - \mathring{v}^{yz}_{\epsilon,\xi}(t)|^2 \, dt'dt\bigg) dy \\
    \label{e:defest1116}
    &\leq  C \|\hat{K}^z\|_{L^1} \bigg(\mint_{C_\rho(0)}|u(y-\frac{\rho}{2}\xi)-z_1(0)| \, dy + \rho^{2\alpha} \EEE \bigg).
\end{align}
Analogously we obtain that for a.e. $\rho>0$ (sufficiently small)
           \begin{align}
    \nonumber
    &\limsup_{\epsilon \to 0^+}\mint_{C_\rho(0)} \bigg(\int_{(\frac{\rho}{2\epsilon},\frac{\rho}{2\epsilon}) \times (\frac{\rho}{2\epsilon},\infty)} K^z(t'-t) |\mathring{v}^{yz}_{\epsilon,\xi}(t') - \mathring{v}^{yz}_{\epsilon,\xi}(t)|^2 \, dt'dt\bigg) dy \\
    \label{e:defest1117}
    &\leq  C \|\hat{K}^z\|_{L^1} \bigg(\mint_{C_\rho(0)}|u(y+\frac{\rho}{2}\xi)-z_2(0)| \, dy + \rho^{2\alpha} \EEE\bigg).
\end{align}
To conclude, properties \eqref{e:defest1118} \eqref{e:defest1116}--\eqref{e:defest1117} yield \eqref{e:oscbound1234}.
\end{proof} 

In the next lemma, we provide a fundamental estimate for the $\Gamma$-$\liminf$ inequality.

\begin{lemma}[The double liminf inequality]
\label{l:est_gammainf}
     Let $x_0 \in \Omega$, let $(u_\epsilon)_\epsilon$ be a family of measurable functions $u_\epsilon \colon \Omega \to [-M,M]$. Assume that
    \begin{equation}
    \label{e:l1conv}
        u_\epsilon \to u, \ \ \text{ in $L^1(\Omega)$ as $\epsilon \to 0^+$}.
    \end{equation}
    Then, by letting $v_\epsilon \colon S_\rho^\epsilon \to [-M,M]$ be defined as $v_\epsilon(y+t\xi):=u_\epsilon(x_0+ y+\epsilon t \xi )$, and  by letting $I_{x_0}$ be the set of $\rho>0$ satisfying Proposition \ref{p:oscbound2}, we have for every infinitesimal sequence $(\rho_\ell) \subset I_{x_0}$ \EEE
     \begin{align}
         \label{e:est_gammainf}
          \liminf_{\rho_\ell \to 0} \EEE \liminf_{\epsilon \to 0^+}F_{\epsilon,x_0}(\mathring{v}_{\epsilon,\xi}; &S_{\rho_\ell}^\epsilon) \rho_\ell^{1-n} \geq \inf_{\gamma \in X_{z_1(x_0)}^{z_2(x_0)}} F^\xi_{x_0}(\gamma) \\
         &- C\|\hat{J}\|_{L^1} \limsup_{\rho_\ell \to 0} \EEE  \mint_{\partial Q^-_{\rho_\ell}(x_0)}|u(x)-z_1(x_0)| \, d\mathcal{H}^{n-1} \\
&- C  \|\hat{J}\|_{L^1}  \limsup_{\rho_\ell\to 0} \EEE  \mint_{\partial Q^+_{\rho_\ell}(x_0)}|u(x)-z_2(x_0)| \, d\mathcal{H}^{n-1},
     \end{align}
    where  $C:=C(M,Z,\lfloor z_1 \rfloor_{C^{0,\alpha}},\lfloor z_2 \rfloor_{C^{0,\alpha}})>0$ \EEE, $\partial Q_\rho(x_0)^\pm$ are defined in Proposition \ref{p:oscbound2}, and $F^\xi_{x_0}$ is given by \eqref{e:functionalxi148}.
\end{lemma}

\begin{proof}
In order to simplify the notation we write $\rho$ in place of $\rho_\ell$. Without loss of generality, we may assume that for every $\rho >0$ small enough we have $\liminf_{\epsilon \to 0^+}F_{\epsilon,x_0}(\mathring{v}_{\epsilon,\xi}; S_\rho^\epsilon) < \infty$. 
% In particular, for every $\rho$ we can work on a (not relabelled) subsequence of $(u_\epsilon)$ (depending on $\rho$) on which the $\liminf$ is realized as a limit. 
In what follows, we simplify the notation by assuming $x_0=0$.

We claim that 
    \begin{equation}
    \label{e:est_gammainf1}
    \begin{split}
   \liminf_{\rho \to 0}\liminf_{\epsilon \to 0^+} F_{\epsilon,0}(\mathring{v}_{\epsilon,\xi};S_\rho^\epsilon) &\rho^{1-n} \geq \liminf_{\rho \to 0}  \mint_{C_\rho(0)} \mathcal{F}^{\rho\xi}_y(\gamma_{y,\rho}) \, dy\\
   &- C\|\hat{J}\|_{L^1}\limsup_{\rho \to 0}\mint_{\partial Q^-_\rho}|u(x)-z_1(0)| \, d\mathcal{H}^{n-1}\\
    &- C\|\hat{J}\|_{L^1}\limsup_{\rho \to 0} \mint_{\partial Q^+_\rho}|u(x)-z_2(0)| \, d\mathcal{H}^{n-1}
        \end{split}
    \end{equation} 
     where $\gamma_{y,\rho}$ denotes any increasing minimizer of $\mathcal{F}^{\rho\xi}_y$ in the class $X^{b_\rho(y)}_{a_\rho(y)}$ provided by Proposition \ref{t:eximin}. 
    
    To prove the claim, we first notice that, since $\mathring{v}_{\epsilon,\xi}$ is $C_\rho(0)$-periodic, then \eqref{e:calibration4} entails that for every $\rho >0$ small enough, and for every $\epsilon >0$ there holds
    \begin{equation}
    \label{e:est_gamma12345}
     F_{\epsilon,0}(\mathring{v}_{\epsilon,\xi};S_\rho^\xi) \geq \int_{V_\xi} \bigg( \int_{C_\rho(0)} G^{yz}_{\rho,\epsilon}(\mathring{v}_{\epsilon,\xi}^{yz}) \, dy\bigg)dz.
    \end{equation}
    By \eqref{e:oscbound} and \eqref{e:oscbound1234}, we have for  $\mathcal{H}^1$-a.e. \EEE sufficiently small $\rho >0$ and for $\mathcal{H}^{n-1}$-a.e. \EEE $z \in V_\xi$
    \begin{align*}
    &\liminf_{\epsilon \to 0^+} \int_{C_\rho(0)} G^{yz}_{\rho,\epsilon}(\mathring{v}_{\epsilon,\xi}^{yz}) \, dy \geq \liminf_{\epsilon \to 0^+} \int_{C_\rho(0)} G^{yz}_{\rho}(\hat{v}_{\epsilon, z}^{yz}) \, dy-  C\rho^{\alpha+n-1}(\|\hat{K}^z\|_{L^1(\mathbb{R})} + \rho^{\alpha} \|K^z\|_{L^1(\mathbb{R})})\\
    &\qquad- c\|\hat{K}^z\|_{L^1(\mathbb{R})}\int_{\partial Q^-_\rho}|u(x)-z_1(0)| \, d\mathcal{H}^{n-1}
    - c\|\hat{K}^z\|_{L^1(\mathbb{R})}\int_{\partial Q^+_\rho}|u(x)-z_2(0)| \, d\mathcal{H}^{n-1},
    \end{align*}
    where we also used that for every $z \in V_\xi$ and $y \in V$ 
    \begin{align*}
    \hat{v}_{\epsilon, z}^{yz}(t)&=\overline{a}_\rho(y) \ \ \text{ for } t \in (-\infty,-\rho/2\epsilon)\\
    \hat{v}_{\epsilon, z}^{yz}(t)&=\overline{b}_\rho(y) \ \ \text{ for } t \in (\rho/2\epsilon,\infty)
    \end{align*}
    and that for every $z \in V_\xi$ and $y \in V$
    \[
    P^z_\rho(y,\overline{b}_\rho(y))=P^z_\rho(y,\overline{a}_\rho(y))=0.
    \]
   By noticing that $t \mapsto \hat{v}_{\epsilon, z}^{yz}(t)$ is admissible in the minimization of the functional $G^{yz}_\rho$ among the class $X^{b_\rho(y)}_{a_\rho(y)}$ for every $z \in V_\xi$ and $y \in C_\rho(0)$, we find 
   \begin{align*}
    &\liminf_{\epsilon \to 0^+} \int_{C_\rho(0)} G^{yz}_{\rho,\epsilon}(\mathring{v}_{\epsilon,\xi}^{yz}) \, dy \geq \int_{C_\rho(0)} G^{yz}_{\rho}(\gamma_{y,\rho})  \, dy 
    -  C \rho^{\alpha+n-1} \EEE(\|\hat{K}^z\|_{L^1(\mathbb{R})} +  \rho^{\alpha} \EEE \|K^z\|_{L^1(\mathbb{R})})\\
    &\qquad- C\|\hat{K}^z\|_{L^1(\mathbb{R})}\int_{\partial Q^-_\rho}|u(x)-z_1(0)| \, d\mathcal{H}^{n-1}\\
    &- C\|\hat{K}^z\|_{L^1(\mathbb{R})}\int_{\partial Q^+_\rho}|u(x)-z_2(0)| \, d\mathcal{H}^{n-1}.
   \end{align*}
  Hence by taking the integral with respect to $z \in V_\xi$ on both sides of the previous inequality and using Fatou's Lemma together with \eqref{e:lowerbound} we obtain for every $\rho >0$ sufficiently small 
   \begin{align*}
   \liminf_{\epsilon \to 0^+} \int_{V_\xi} \bigg( \mint_{C_\rho(0)} &G^{yz}_{\rho,\epsilon}(\mathring{v}_{\epsilon,\xi}^{yz}) \, dy\bigg)dz \geq \mint_{C_\rho(0)} \mathcal{F}^{\rho\xi}_{y}(\gamma_{y,\rho}) \, dy\\ 
    &-  C \rho^{\alpha} \EEE\bigg(\int_{V_\xi}\|\hat{K}^z\|_{L^1(\mathbb{R})} \, dz +  \rho^{\alpha} \EEE \int_{V_\xi}\|K^z\|_{L^1(\mathbb{R})} \, dz\bigg)\\
    &- C\bigg(\int_{V_\xi}\|\hat{K}^z\|_{L^1(\mathbb{R})}\, dz\bigg)\mint_{\partial Q^-_\rho}|u(x)-z_1(0)| \, d\mathcal{H}^{n-1}\\
    &- C\bigg(\int_{V_\xi}\|\hat{K}^z\|_{L^1(\mathbb{R})}\, dz\bigg)\mint_{\partial Q^+_\rho}|u(x)-z_2(0)| \, d\mathcal{H}^{n-1}.
   \end{align*}
   Now we first notice that $\int_{V_\xi}\|\hat{K}^z\|_{L^1(\mathbb{R})} \, dz = \int_{ \mathbb{R}^m \EEE} J(h)|h\cdot \xi| \, dh$ and $\int_{V_\xi}\|K^z\|_{L^1(\mathbb{R})} \, dz = \int_{ \mathbb{R}^m \EEE} J(h) \, dh$.   Hence, \eqref{e:est_gammainf1} follows by taking the liminf as $\rho \to 0$ on both sides of the above inequality in combination with \eqref{e:est_gamma12345}.
   
The proposition thus follows if we show that
\begin{equation}
\label{e:est_gammainf2}
\liminf_{\rho \to 0}  \mint_{C_\rho(0)} \mathcal{F}^{\rho\xi}_y(\gamma_{y,\rho}) \, dy \geq  \inf_{\gamma \in X_{z_2(0)}^{z_1(0)}} F^\xi_{0}(\gamma).
\end{equation}
A dilation in $y$ by a factor $\rho$ gives 
\begin{align}
\label{e:est_gammainf2.5}
    \mint_{C_\rho(0)} \mathcal{F}^{\rho\xi}_y(\gamma_{y,\rho}) \, dy
    &= \int_{C_1(0)}\bigg( \int_{\mathbb{R} \times \mathbb{R}} J^\xi(h)|\gamma_{\rho y,\rho}(t+h) -\gamma_{\rho y,\rho}(t)|^2 \,dtdh\bigg) dy\\
    \nonumber
    &+ \int_{C_1(0)} \bigg( \int_{\mathbb{R}} \mathcal{W}^{\rho\xi}(\rho y,\gamma_{\rho y,\rho}(t)) \,dt\bigg) dy.
\end{align}

By virtue of \ref{Y1}--\ref{Y4} we are in position to apply Proposition \ref{p:contdep} with $a_\rho,b_\rho,W_\rho,W_0,\gamma_\rho$ replaced by $a_\rho(\rho y),b_\rho(\rho y),\mathcal{W}^{\rho\xi}_y(\rho y,\cdot),W(0,\cdot),\gamma_{\rho y,\rho}$, respectively, and infer for every $y \in C_\rho(0)$ that
\begin{equation}
\label{e:est_gammainf3}
\lim_{\rho \to 0} \int_{\mathbb{R} \times \mathbb{R}} J^\xi(h)|\gamma_{\rho y,\rho}(t+h) -\gamma_{\rho y,\rho}(t)|^2 \,dtdh
    +  \int_{\mathbb{R}} \mathcal{W}^{\rho \xi}(\rho y,\gamma_{\rho y,\rho}(t)) \,dt  = F^\xi_{0}(\gamma_0), 
\end{equation}
where $\gamma_0$ is any increasing minimizer of $F^\xi_{0}$ in the class $X_{z_1(0)}^{z_2(0)}$ (recall that $a_\rho(\rho y) \to z_1(0)$ and $b_\rho(\rho y) \to z_2(0)$ for every $y \in C_1(0)$ ). Therefore, we conclude the proof by applying Fatou's lemma in \eqref{e:est_gammainf2.5}. 
\end{proof}
We are finally in a position to prove Proposition \ref{c:est_gamminf}.
\begin{proof}[Proof of Proposition \ref{c:est_gamminf}]
The statement follows by by letting $I_{x_0}$ be the set of $\rho>0$ satisfying Proposition \ref{p:oscbound2}, putting together Lemma \ref{l:est_gammainf} and Remark \ref{r:kernelinvariance}, as well as by observing that the result of Lemma \ref{l:est_gammainf} is invariant under reparametrizations of the kernel $J$ in the $y$-variable. 
\end{proof}

\section*{Acknowledgements} 
This research has been supported by the Austrian Science Fund (FWF) through grants \href{https://doi.org/10.55776/F65}{10.55776/F65}, \href{https://doi.org/10.55776/Y1292}{10.55776/Y1292}, \href{https://doi.org/10.55776/P35359}{10.55776/P35359}, \href{https://doi.org/10.55776/F100800}{10.55776/F100800} as well as by the OeAD-WTZ project CZ04/2019 (M\v{S}MT\v{C}R 8J19AT013). For open access purposes, the author has applied a CCBY public copyright license to any author accepted manuscript version arising from this submission.

\section*{Data Availability Statement}
The authors hereby confirm that data sharing is
not applicable to this article, for no datasets were generated or analyzed during the current study.

\section*{Competing interests}
The authors declare no competing interests.


\begin{thebibliography}{99}

\bibitem{alb-bel1}G. Alberti, and G. Bellettini. \emph{A non-local anisotropic model for phase transitions: the optimal profile problem}. Mathematische Annalen 310.3 (1998): 527-560

\bibitem{alb-bel2}G. Alberti, and G. Bellettini. \emph{A non-local anisotropic model for phase transitions: asymptotic behaviour of rescaled energies.} European Journal of Applied Mathematics 9.3 (1998): 
261-284

\bibitem{abcp} G. Alberti, G. Bellettini, M. Cassandro, and E. Presutti.  \emph{Surface tension in Ising systems with Kac potentials}. Journal of statistical physics 82 (1996): 743-796.


\bibitem{alt} H.W. Alt, and I. Pawlow. \emph{A mathematical model of dynamics of non-isothermal phase separation}. Physica D: Nonlinear Phenomena 59.4 (1992): 389-416.


\bibitem{amb} L. Ambrosio. \emph{Metric space valued functions of bounded variation}. Annali della Scuola Normale Superiore di Pisa-Classe di Scienze 17.3 (1990): 439-478.


\bibitem{ags} L. Ambrosio, N. Gigli, and G. Savaré. \emph{Gradient flows: in metric spaces and in the space of probability measures.} Springer Science \& Business Media, 2005.


\bibitem{bal} S. Baldo. \emph{Minimal interface criterion for phase transitions in mixtures of Cahn-Hilliard fluids}. Annales de l'Institut Henri Poincaré C, Analyse non linéaire. Vol. 7. No. 2. No longer published by Elsevier, 1990.


\bibitem{bate} P.W. Bates, P.C. Fife, X. Ren, and X. Wang. \emph{Traveling waves in a convolution model for phase transitions}. Archive for Rational Mechanics and Analysis 138 (1997): 105-136.


\bibitem{buc} G. Bouchitté. \emph{Singular perturbations of variational problems arising from a two-phase transition model}. Applied mathematics \& optimization 21.1 (1990): 289-314.


\bibitem{caff1} L. Caffarelli, and L. Silvestre. \emph{Regularity results for non-local equations by approximation}. Archive for rational mechanics and analysis 200 (2011): 59-88.


\bibitem{caff2} L. Caffarelli, and L. Silvestre. \emph{Regularity theory for fully nonlinear integro‐differential equations}. Communications on Pure and Applied Mathematics 62.5 (2009): 597-638.


\bibitem{cahn} J.W. Cahn, and J.E. Hilliard. \emph{Free energy of a nonuniform system. I. Interfacial free energy}. The Journal of chemical physics 28.2 (1958): 258-267.

\bibitem{caldwell} W. Caldwell. \emph{Nonlocal Phase Transitions with Singular Heterogeneous Kernels}. Preprint arXiv 2410.19624.


\bibitem{cozzi} M. Cozzi, and T. Passalacqua. \emph{One-dimensional solutions of non-local Allen–Cahn-type equations with rough kernels.} Journal of Differential Equations 260.8 (2016): 6638-6696.


\bibitem{crigra} R. Cristoferi, and G. Gravina. \emph{Sharp interface limit of a multi-phase transitions model under nonisothermal conditions}. Calculus of Variations and Partial Differential Equations 60.4 (2021): 142.


\bibitem{dong} H. Dong, and D. Kim. \emph{On $L^p$-estimates for a class of non-local elliptic equations.} Journal of Functional Analysis 262.3 (2012): 1166-1199.


\bibitem{fef} C. Fefferman, and P. Shvartsman. \emph{Sharp finiteness principles for Lipschitz selections.} Geometric and Functional Analysis 28.6 (2018): 1641-1705.


\bibitem{fonleo} I. Fonseca, and G. Leoni. \emph{Modern methods in the calculus of variations: $L^p$ spaces}. Springer Science \& Business Media, 2007.

\bibitem{fonmul} I. Fonseca, and S. M\"uller. \emph{Relaxation of quasiconvex functionals in $BV(\Omega, \mathbb{R}^p)$ for integrands $f(x, u,\nabla u)$}. Arch. Ration. Mech. Anal 123.1 (1993): 1-49.


\bibitem{fontar} I. Fonseca, and L. Tartar. \emph{The gradient theory of phase transitions for systems with two potential wells}. Proceedings of the Royal Society of Edinburgh Section A: Mathematics 111.1-2 (1989): 89-102.


\bibitem{giusti} E. Giusti. \emph{Minimal surfaces and functions of bounded variation}. Vol. 80. Boston: Birkhäuser, 1984.


\bibitem{gur1} M. Gurtin. \emph{Some results and conjectures in the gradient theory of phase transitions}. Metastability and incompletely posed problems. New York, NY: Springer New York, 1987. 135-146.


\bibitem{gur2} J. Carr, M.E. Gurtin, and M. Slemrod. \emph{Structured phase transitions on a finite interval}. Archive for rational mechanics and analysis 86 (1984): 317-351.


\bibitem{kac} M. Kac, G.E. Uhlenbeck, and P.C. Hemmer. \emph{On the van der Waals Theory of the Vapor‐Liquid Equilibrium. I. Discussion of a One‐Dimensional Model}. Journal of Mathematical Physics 4.2 (1963): 216-228.


\bibitem{kass} M. Kassmann. \emph{A priori estimates for integro-differential operators with measurable kernels}. Calculus of Variations and Partial Differential Equations 34.1 (2009): 1-21


\bibitem{mod2} L. Modica. \emph{The gradient theory of phase transitions and the minimal interface criterion}. Archive for Rational Mechanics and Analysis 98 (1987): 123-142.


\bibitem{modica} S. Luckhaus, and L. Modica. \emph{The Gibbs-Thompson relation within the gradient theory of phase transitions.} Archive for Rational Mechanics and Analysis 107 (1989): 71-83


\bibitem{pen} O. Penrose, and P.C. Fife. \emph{Thermodynamically consistent models of phase-field type for the kinetic of phase transitions.} Physica D: Nonlinear Phenomena 43.1 (1990): 44-62.


\bibitem{sav} O. Savin, and E. Valdinoci. \emph{$\Gamma$-convergence for non-local phase transitions}. Annales de l'Institut Henri Poincaré C, Analyse non linéaire. Vol. 29. No. 4. No longer published by Elsevier, 2012.


\bibitem{serra} J. Serra. \emph{$C^{\sigma+\alpha}$ regularity for concave non-local fully nonlinear elliptic equations with rough kernels}. Calculus of Variations and Partial Differential Equations 54 (4) (2015) 3571–3601.


\bibitem{silv1} L. Silvestre. \emph{H\"older estimates for solutions of integro-differential equations like the fractional Laplace}. Indiana University mathematics journal (2006): 1155-1174.

\bibitem{silv2} L. Silvestre. \emph{Regularity of the obstacle problem for a fractional power of the Laplace operator}. Communications on Pure and Applied Mathematics 60.1 (2007): 67-112.


\bibitem{ster2} P. Sternberg. \emph{Vector-valued local minimizers of nonconvex variational problems}. The Rocky Mountain Journal of Mathematics 21.2 (1991): 799-807.


\bibitem{ster1} P. Sternberg. \emph{The effect of a singular perturbation on nonconvex variational problems}. Archive for Rational Mechanics and Analysis 101 (1988): 209-260.


\bibitem{ster3} P. Sternberg, and K. Zumbrun. \emph{Connectivity of phase boundaries in strictly convex domains}. Archive for Rational Mechanics and Analysis 141.4 (1998): 375-400.


\bibitem{vander} J.D. Van der Waals. \emph{The thermodynamic theory of capillarity under the hypothesis of a continuous variation of density}. Journal of Statistical Physics 20.2 (1979): 200-244.

%\bibitem{zara} Grünbaum, F., and Eduardo Héctor Zarantonello. \emph{On the extension of uniformly continuous mappings.} Michigan Math. J. 15.1 (1968): 65-74.

\end{thebibliography}
\end{document}